\documentclass[11pt]{article}

\pdfoutput=1

\usepackage{amsfonts}
\usepackage{amsmath}

\usepackage{mdframed} 
\usepackage{thmtools}
\usepackage{mathtools}

\newtheorem{theorem}{Theorem}[section]
\newtheorem{proposition}{Proposition}[section]
\newtheorem{lemma}{Lemma}[section]
\newtheorem{corollary}{Corollary}[section]
\newtheorem{assumption}{Assumption}[section]

\newtheorem{definition}{Definition}[section]

\declaretheorem[sibling=definition]{remark}
\usepackage{colordvi,epsfig}
\usepackage[table]{xcolor}   % for highlighting table cells
\usepackage{verbatim}

\usepackage{algpseudocode} 
\usepackage{algorithmicx}
\usepackage{algorithm}

\usepackage{longtable}

\usepackage{verbatim,tikz}
\usepackage{caption,subcaption}   % Required for side-by-side figures with individual captions
\DeclareMathOperator*{\argmin}{argmin}

\graphicspath{{./figures/}}  % Path to figures folder

\newcommand{\ignore}[1]{}
\newcommand{\R}{\mathbb{R}}

\newcommand{\proxop}{\mathop{\mathrm{prox}}\nolimits}
\newcommand{\prox}{\proxop_{\alpha \psi}}

\newcommand{\eqdef}{\coloneqq} 

\newcommand{\diag}{\mD} 
\newcommand{\trace}{{\rm Trace}}

\newcommand{\cA}{{\cal A}}
\newcommand{\cB}{{\cal B}}

\newcommand{\cD}{{\cal D}}

\newcommand{\cI}{{\cal I}}

\newcommand{\cM}{{\cal M}}

\newcommand{\cO}{{\cal O}}
\newcommand{\cP}{{\cal P}}

\newcommand{\cR}{{\cal R}}
\newcommand{\cS}{{\cal S}}

\newcommand{\cU}{{\cal U}}

\newcommand{\cX}{{\cal X}}

\newcommand{\mA}{{\bf A}}
\newcommand{\mB}{{\bf B}}

\newcommand{\mD}{{\bf D}}

\newcommand{\mG}{{\bf G}}

\newcommand{\mI}{{\bf I}}
\newcommand{\mJ}{{\bf J}}

\newcommand{\mL}{{\bf L}}
\newcommand{\mM}{{\bf M}}

\newcommand{\mN}{{\bf N}}

\newcommand{\mR}{{\bf R}}

\newcommand{\mX}{{\bf X}}
\newcommand{\mY}{{\bf Y}}
\newcommand{\mZ}{{\bf Z}}

\newcommand{\piop}{{\Gamma}}

\newcommand{\EE}[2]{\mathbb{E}_{#1}\left[#2\right] }
\newcommand{\E}[1]{\mathbb{E}\left[#1\right] } 
\newcommand{\ED}[1]{\mathbb{E}_{\cD}\left[#1\right] } 
\newcommand{\Prob}[1]{\mathbb{P}\left(#1\right) }

\newcommand{\norm}[1]{\left \| #1 \right\|}
\newcommand{\dotprod}[1]{\left< #1\right>}
\newcommand{\Tr}[1]{\mbox{Tr}\left( #1\right)}
\providecommand{\Null}[1]{{\rm Null}\left( #1\right)}

\providecommand{\Range}[1]{{\rm Range}\left( #1\right)}

\usepackage[colorinlistoftodos,bordercolor=orange,backgroundcolor=orange!20,linecolor=orange,textsize=scriptsize]{todonotes}

\usepackage{hyperref}

\def\<#1,#2>{\left\langle #1,#2\right\rangle}

\newcommand{\NORMG}[1]{\left \| #1\right  \|^2}

\newcommand{\compactify}{} % use this for arXiv

\newcommand{\eL}{{\color{red}e}} % standard unit basis vector in R^d (i.e., in the left space)
\newcommand{\eR}{{\color{blue}e}} % 
\newcommand{\eRR}{{\color{blue}e_R}} %standard unit basis vector in R^n (i.e., in the right space)
\newcommand{\pL}{{\color{red}p}} % probability vector in R^d (i.e., in the left space)
\newcommand{\pR}{{\color{blue}p}} %  probability vector in R^n (i.e., in the right space)
\newcommand{\eLi}{{\color{red}e_i}} % standard unit basis vector in R^d (i.e., in the left space)
\newcommand{\eRj}{{\color{blue}e_j}} % standard unit basis vector in R^n (i.e., in the right space)
\newcommand{\pLi}{{\color{red}p_i}} % probability vector in R^d (i.e., in the left space)
\newcommand{\pRj}{{\color{blue}p_j}} %  probability vector in R^n (i.e., in the right space)
\newcommand{\pLL}{{\color{red}p_L}} % probability vector in R^d (i.e., in the left space)
\newcommand{\pRR}{{\color{blue}p_R}} %  probability vector in R^n (i.e., in the right space)
\newcommand{\dL}{{\color{red}d}} %  probability vector in R^n (i.e., in the right space)
\newcommand{\nR}{{\color{blue}n}} %  probability vector in R^n (i.e., in the right space)

\newcommand{\NRt}{{\color{blue}N_t}} %  probability vector in R^n (i.e., in the right space)
\newcommand{\tR}{{t}} %  probability vector in R^n (i.e., in the right space)
\newcommand{\TR}{{T}} %  probability vector in R^n (i.e., in the right space)

\newcommand{\cDR}{{\color{blue}\cD}} 
\newcommand{\cDL}{{\color{red} \cD}} 

\newcommand{\PR}{{\color{blue} \bf P}} %  probability vector in R^n (i.e., in the right space)
\newcommand{\qR}{{\color{blue}q}} %  probability vector in R^n (i.e., in the right space)

\newcommand{\qRj}{{\color{blue}q_j}} %  probability vector in R^n (i.e., in the right space)

\newcommand{\ptRj}{{\color{blue} p^t_j}} %  probability vector in R^n (i.e., in the right space)
\newcommand{\ptLi}{{\color{red}p}_{\color{red} i}^{\color{blue}t}} %  probability vector in R^n (i.e., in the right space)

\newcommand{\ptR}{{\color{blue} p^t}} %  probability vector in R^n (i.e., in the right space)
\newcommand{\ptL}{{\color{red}p^{\color{blue} t}}} %  probability vector in R^n (i.e., in the right space)

\newcommand{\PtR}{{\color{blue} {\bf P}^t}} %  probability vector in R^n (i.e., in the right space)

\newcommand{\qtR}{{\color{blue} q^t}} %  probability vector in R^n (i.e., in the right space)
\newcommand{\qtRj}{{\color{blue} q^t_j}} %  probability vector in R^n (i.e., in the right space)

\newcommand{\probx}{{\color{cyan} \rho}}
\newcommand{\proby}{{\color{cyan} \delta }} 

\newcommand{\ugly}{{\eta}} 

\graphicspath{{../code/images/}}

\usepackage[utf8]{inputenc} % allow utf-8 input
\usepackage[T1]{fontenc}    % use 8-bit T1 fonts
\usepackage{hyperref}       % hyperlinks
\usepackage{url}            % simple URL typesetting
\usepackage{amsfonts}       % blackboard math symbols
\usepackage{nicefrac}       % compact symbols for 1/2, etc.
\usepackage{microtype}      % microtypography
\usepackage[margin=1in]{geometry}

\begin{document}

\title{One Method to Rule Them All: Variance Reduction for Data, Parameters and Many New Methods}

\author{Filip Hanzely\footnote{King Abdullah University of Science and Technology, Thuwal, Saudi Arabia} \and Peter Richt\'{a}rik\footnote{King Abdullah University of Science and Technology, Thuwal, Saudi Arabia} }

\maketitle

\begin{abstract}%   <- trailing '%' for backward compatibility of .sty file
We propose a remarkably general variance-reduced method suitable for solving regularized empirical risk minimization problems with either a large number of training examples, or a large model dimension, or both. In special cases, our method reduces to several known and previously thought to be unrelated methods, such as {\tt SAGA}~\cite{SAGA}, {\tt LSVRG}~\cite{hofmann2015variance, LSVRG}, {\tt JacSketch}~\cite{gower2018stochastic}, {\tt SEGA}~\cite{hanzely2018sega} and {\tt ISEGA}~\cite{mishchenko201999}, and their arbitrary sampling and proximal generalizations. However, we also highlight a large number of new specific algorithms with interesting properties. We provide a single theorem establishing linear convergence of the method under smoothness and quasi strong convexity assumptions. With this theorem we recover best-known and sometimes improved rates for known methods arising in special cases. As a by-product, we provide the first unified method and theory for stochastic gradient and stochastic coordinate descent type methods. 
\end{abstract}

\section{Introduction}

In this work we are studying stochastic algorithms for solving regularized empirical risk minimization problems, i.e., optimization problems of the form 
\begin{equation}\label{eq:problem}
\compactify \min_{x\in \R^d}   \frac1n \sum \limits_{j=1}^n f_j(x) + \psi(x),
\end{equation}
We assume  that the functions $f_j:\R^d \to \R$ are smooth and convex, and $\psi:\R^d\to \R\cup \{+\infty\}$ is a proper, closed and convex regularizer, admitting a cheap proximal operator. We write $f\eqdef \tfrac{1}{n}\sum_j f_j$. 

{\bf Proximal gradient descent.} A baseline method for solving problem \eqref{eq:problem} is  {\em (proximal) gradient descent} ({\tt GD}). This method performs a gradient step in $f$, followed by a proximal step\footnote{The proximal operator is defined via $\prox(x) \eqdef \argmin_{u\in \R^d} \{ \alpha\psi(u) + \frac{1}{2}\|u-x\|^2\}$.} in $\psi$, i.e., 
\begin{equation}\label{eq:PGM} x^{k+1} = \prox( x^k - \alpha \nabla f(x^k)),\end{equation}
where $\alpha>0$ is a stepsize. {\tt GD} performs well when both $n$ and $d$ are not too large. However, in the big data (large $n$) and/or big parameter (large $d$) case,  the formation of the gradient becomes overly expensive, rendering {\tt GD} inefficient in both theory and practice. A typical remedy is to replace the gradient by a cheap-to-compute random approximation. Typically, one replaces $\nabla f(x^k)$ with a random vector $g^k$ whose mean is the gradient: $\E{g^k} = \nabla f(x^k)$, i.e., with a stochastic gradient. This results in the {\em (proximal) stochastic gradient descent} ({\tt SGD}) method:
\begin{equation}\label{eq:PSGD} x^{k+1} = \prox( x^k - \alpha g^k).\end{equation}

Below we comment on the typical approaches to constructing $g^k$ in the big $n$ and big $d$ regimes.

{\bf Proximal {\tt SGD}.} In the big $n$ regime, the simplest choice is to set \begin{equation} \label{eq:nbu9gff}g^k = \nabla f_j(x^k)\end{equation} for an index $j\in [n]\eqdef \{1,2,\dots,n\}$ chosen uniformly at random. By construction, it is $n$ times cheaper to compute this estimator than the gradient, which is a key driving force behind the efficiency of this variant of {\tt SGD}. However, there is an infinite array of other possibilities of constructing an unbiased estimator~\cite{needell2016batched, gower2019sgd}. Depending on how $g^k$ is formed, \eqref{eq:PSGD} specializes to one of the many existing variants of proximal {\tt SGD}, each with different convergence properties and proofs.

{\bf Proximal {\tt RCD}.} In the big $d$ regime (this is interesting even if $n=1$), the simplest choice is to set \begin{equation} \label{eq:b98gf98f} g^k = d  \langle \nabla f(x^k), \eLi \rangle \eLi,\end{equation} where $\langle x,y\rangle =\sum_i x_i y_i$ is the standard Euclidean inner product, $\eLi$ is the $i$-th standard unit basis vector in $\R^d$, and $i$ is chosen uniformly at random from $[d]\eqdef \{1,2,\dots,d\}$. With this estimator, \eqref{eq:PSGD} specializes to  (proximal) randomized coordinate decent ({\tt RCD}). There are situations where it is $d$ times cheaper to compute the partial derivative $\nabla_i f(x^k)\eqdef \langle \nabla f(x^k), \eLi \rangle$ than the gradient, which is a key driving force behind the efficiency of {\tt RCD} \cite{RCDM}.  However, there is an infinite array of other possibilities for constructing an unbiased estimator of the gradient in a similar way~\cite{NSync,PCDM,ESO}. 
% Depending on how $g^k$ is constructed, \eqref{eq:PSGD} specializes to one of the many existing variants of proximal {\tt RCD}, each with different convergence properties and proofs.

{\bf Issues.}  For the sake of argument in the rest of this section, assume that $f$ is a $\sigma$-strongly convex function, and let $x^*$ be the (necessarily) unique solution of \eqref{eq:problem}.  It is well known that in this case, method \eqref{eq:PSGD}  with estimator $g^k$ defined as in \eqref{eq:nbu9gff} does {\em not} in general converge to $x^*$. Instead, {\tt SGD}  converges linearly to a neighborhood of $x^*$ of size proportional to the stepsize $\alpha$, noise $\nu^2\eqdef \frac{1}{n}\sum_j \norm{\nabla f_j(x^*)}^2$, and inversely proportional to $\sigma$~\cite{moulines2011non, NeedellWard2015}.  In the generic regime with $\nu^2>0$, the neighbourhood is nonzero, causing issues with convergence. This situation does not change even when tricks such as {\em mini-batching} or {\em importance sampling} (or a combination of both) are applied~\cite{needell2014stochastic,needell2016batched, gower2019sgd}. While these tricks affect both the (linear convergence) rate and the size of the neighbourhood, they are incapable\footnote{Unless, of course, in the special case when one  uses the full batch approximation $g^k = \nabla f(x^k)$.} of ensuring convergence to the solution. However, a remedy does exist: the situation with non-convergence can be  resolved by using one of the many {\em variance-reduction}  strategies for constructing $g^k$ developed over the last several years~\cite{SAG, SAGA, johnson2013, mairal2013optimization, SDCA}.  Further, while it is well known that method \eqref{eq:PSGD}  with estimator $g^k$  defined as in \eqref{eq:b98gf98f} (i.e., randomized coordinate descent) converges to $x^*$ for $\psi\equiv 0$~\cite{RCDM, UCDC, NSync}, it is also known that it does {\em not}  generally converge to $x^*$ unless the regularizer $\psi$ is separable (e.g., $\psi(x) = \norm{x}_1$ or $\psi(x)=c_1 \|x\|_1 + c_2 \|x\|_2^2$). In~\cite{hanzely2018sega}, an alternative estimator (known as {\tt SEGA}) was constructed from the same (random) partial derivative information $\nabla f_i(x^k)$, one that does not suffer from this incompatibility with general regularizers $\psi$. This work  resolved a long standing open problem in the theory of {\tt RCD} methods.

% In the very special case when $\sigma^2=0$, this problem does not occur and the iterates converge to $x^*$.

\section{Contributions} 

Having experienced a ``Cambrian explosion'' in the last 10 years, the world of efficient {\tt SGD} methods is remarkably complex. There is a large and growing set of rules for constructing the gradient estimators $g^k$, with differing levels of sophistication and varying theoretical and practical properties. It includes the classical estimator \eqref{eq:nbu9gff}, as well as an infinite array of mini-batch \cite{Li2014} and importance sampling \cite{NeedellWard2015, IProx-SDCA} variants, and a growing list of variance-reduced variants \cite{SAGA}. Furthermore, there are estimators of the coordinate descent variety, including the simplest one based on \eqref{eq:b98gf98f}~\cite{RCDM}, more elaborate variants utilizing the arbitrary sampling paradigm \cite{ALPHA}, and variance reduced methods capable of handling general non-separable regularizers~\cite{hanzely2018sega}.

$\triangleright$ {\bf New general method and a single convergence theorem}. In this paper we propose a {\em general method}---which we call {\tt GJS}---which reduces to many of the aforementioned classical and several recently developed {\tt SGD} type methods in special cases.  We provide a {\em single convergence theorem}, establishing a linear convergence rate for {\tt GJC}, assuming $f$ to be smooth and quasi strongly convex. In particular, we obtain  the following methods in special cases, or their generalizations, always recovering the best-known convergence guarantees or improving upon them: {\tt SAGA}~\cite{SAGA, qian2019saga, gazagnadou2019optimal}, {\tt JacSketch}~\cite{gower2018stochastic},  {\tt LSVRG}~\cite{hofmann2015variance, LSVRG}, {\tt SEGA}~\cite{hanzely2018sega},  and {\tt ISEGA}~\cite{mishchenko201999}  (see Table~\ref{tbl:all_special_cases}, in which we list 17 special cases).  This is the first time such a direct connection is made between many of these methods, which previously required different intuitions and dedicated analyses. Our general method, and hence also all special cases we consider, can work with a regularizer. This provides novel (although not hard) results for some methods, such as {\tt LSVRG}.

$\triangleright$ {\bf Unification of {\tt SGD} and {\tt RCD}.} As a by-product of the generality of {\tt GJS}, we obtain the {\em unification of variance-reduced {\tt SGD} and variance reduced {\tt RCD} methods.} To the best of our knowledge, there is no algorithm besides {\tt GJS}, one whose complexity is captured by a single theorem, which specializes to {\tt SGD} and {\tt RCD} type methods at the same time and recovers best known rates in both cases.\footnote{A single theorem (not a single algorithm) to obtain rates for both variance-reduced {\tt SGD} and variance reduced {\tt RCD} methods was done in the concurrent work~\cite{sigmak}. However, ~\cite{sigmak} does not capture the best known rates for {\tt RCD} methods and focuses in orthogonal direction instead -- includes non-variance reduced methods.}

$\triangleright$ {\bf Generalizations to arbitrary sampling.}  Many specialized methods we develop are cast in a very general {\em arbitrary sampling} paradigm~\cite{NSync, quartz, ALPHA}, which allows for the estimator $g^k$ to be formed through information contained in a random subset $R^k\subseteq [n]$ (by computing $\nabla f_j(x^k)$ for $j\in \R^k$) or a random subset $L^k\subseteq [d]$ (by computing $\nabla_i f(x^k)$ for $i\in L^k$), where these subsets are allowed to follow an arbitrary distribution. In particular, we  extend {\tt SEGA}~\cite{hanzely2018sega}, {\tt LSVRG}~\cite{hofmann2015variance, LSVRG} or {\tt ISEGA}~\cite{mishchenko201999} to this setup.  Likewise, {\tt GJS} specializes to an arbitrary sampling extension of the {\tt SGD}-type method {\tt SAGA}~\cite{SAGA, qian2019saga}, obtaining state-of-the-art rates. As a special case of the arbitrary sampling paradigm, we obtain \emph{importance sampling} versions of all mentioned methods.

$\triangleright$ {\bf New methods.} {\tt GJS} can be specialized to many new specific methods. To illustrate this, we construct 10 specific {\em new} methods in special cases, some with intriguing structure and properties (see Section~\ref{sec:Special_Cases}; Table~\ref{tbl:all_special_cases}; and Table~\ref{tbl:all_special_cases_theory} for a summary of the rates). 

%For instance, we obtain a i) {\tt ISAEGA} (Algorithm~\ref{alg:isaega}), which is a combination of {\tt ISEGA}~\cite{mishchenko201999} and {\tt ISAGA}~\cite{mishchenko201999},  with arbitrary sampling, ii) {\tt LSVRG}~\cite{hofmann2015variance, LSVRG} with arbitrary sampling (Algorithm~\ref{alg:LSVRG-AS}), iii) {\tt SEGA}~\cite{hanzely2018sega} with arbitrary sampling, iv) {\tt SVRCD} with arbitrary sampling.

$\triangleright$ {\bf Relation to {\tt JacSketch}.} Our method can be seen as a vast generalization of the recently proposed Jacobian sketching method {\tt JacSketch}~\cite{gower2018stochastic} in several directions, notably by enabling {\em arbitrary randomized linear} (i.e., sketching) operators, allowing different linear operators to learning Jacobian and constructing control variates, extending the analysis to the proximal case, and replacing strong convexity assumption by quasi strong convexity or strong growth (see Appendix~\ref{sec:sg}). In particular, from all methods we recover, only variants of {\tt SAGA} can be obtained from {\tt JacSketch}~\cite{gower2018stochastic} (even in that case, rates obtained from~\cite{gower2018stochastic} are suboptimal).

%\filip{ \cite{kulunchakov2019estimate} provides variance reduction framework with nonuniform sampling for SAGA/SVRG/MISO/SDCA. However, they do one iteration at a time and at the same time, their analysis does not provide importance sampling for SEGA }

$\triangleright$  {\bf Limitations.} 
We  focus on developing methods capable of enjoying a linear convergence rate with a fixed stepsize $\alpha$ and do not consider the non-convex setting.  Although there exist several {\em accelerated} variance reduced algorithms~\cite{lan2018optimal, allen2017katyusha, pmlr-v80-zhou18c, pmlr-v89-zhou19c, LSVRG, kulunchakov2019estimate}, we do not consider such methods here. 

{\em Notation.} Let $\eR$ (resp.\ $\eL$) be the vector of all ones in $\R^n$ (resp.\ $\R^d$), and  $\eRj$ (resp.\ $\eLi$) be the $j$-th (resp.\ $i$-th) unit basis vector in $\R^n$ (resp.\ $\R^d$).  By $\|\cdot \|$ we denote the standard Euclidean norm in $\R^d$ and $\R^n$.  Matrices are denoted by upper-case bold letters. Given $\mX,\mY\in \R^{d\times n}$, let $\langle \mX, \mY\rangle \eqdef \Tr{\mX^\top \mY}$ and  $\|\mX\|\eqdef \langle \mX, \mX \rangle^{1/2}$ be the Frobenius norm. By $\mX_{:j}$ (resp.\ $\mX_{i:}$) we denote the $j$-th column (resp.\ $i$-th row) of matrix $\mX$.  By $\mI_n$ (resp.\ $\mI_d$) we denote the $n\times n$ (resp.\ $d\times d$) identity matrices. Upper-case calligraphic letters, such as $\cS,\cU, \cI, \cM, \cR$, are used to denote (deterministic or random) linear operators mapping $\R^{d\times n}$ to $\R^{d\times n}$. Most used notation is summarized in Table~\ref{tbl:notation}  in Appendix~\ref{sec:notation_table}.

\section{Sketching}

A key object in this paper is the Jacobian matrix $\mG(x) = [\nabla f_1(x),\dots, \nabla f_n(x)] \in \R^{d\times n}.$
Note that  \begin{equation}\label{eq:nbifg98dz}
\compactify \nabla f(x) = \frac{1}{n}\mG(x) \eR.\end{equation} 
Extending the insights from \cite{gower2018stochastic}, one of the key observations of this work is that {\em random linear transformations} (sketches) of $\mG$ can be used to {\em construct} unbiased estimators of the gradient of $f$.  For instance,  $\mG(x^k) \eRj$  leads to the simple {\tt SGD} estimator \eqref{eq:nbu9gff},  and $\frac{d}{n} \eLi \eLi^\top \mG(x^k) \eR $ gives the simple {\tt RCD} estimator \eqref{eq:b98gf98f}. We will consider more elaborate examples later on.  It will be useful to embed these estimators into $\R^{d\times n}$. For instance, instead of $\mG(x^k)\eRj$ we consider the matrix  $\mG(x^k) \eRj \eRj^\top$. Note that all columns of this matrix are zero, except for the $j$-th column, which is equal to  $\mG(x^k)\eRj$. Similarly, instead of $\frac{d}{n} \eLi \eLi^\top \mG(x^k) \eR $ we will consider the matrix $\frac{d}{n} \eLi \eLi^\top \mG(x^k)$. All rows of this matrix are zero, except for the $i$-th row, which consists of the $i$th partial derivatives of functions $f_j(x^k)$ for $j\in [n]$, scaled by $\frac{d}{n}$.

{\bf Random projections.}
 Generalizing from these examples, we  consider a random linear operator (``sketch'') $\cA:\R^{d\times n}\to \R^{d\times n}$. By $\cA^\ast$ we denote the adjoint of $\cA$, i.e., linear operator  satisfying $\langle \cA \mX, \mY \rangle = \langle \mX, \cA^\ast \mY \rangle$ for all $\mX,\mY\in \R^{d\times n}$.  Given $\cA$, we let $\cP_{\cA}$ be the (random) projection operator onto $\Range{\cA^\ast}$. That is, \[\cP_{\cA}(\mX) = \arg\min_{\mY \in \Range{\cA^\ast}} \norm{\mX - \mY } = \cA^\ast (\cA \cA^\ast)^\dagger \cA \mX,\]
where ${}^\dagger$ is the Moore-Penrose pseudoinverse. The identity operator is denoted by $\cI$. We say that $\cA$ is {\em identity in expectation}, or {\em unbiased} when $\E{\cA} = \cI$; i.e., when if $\E{\cA  \mX} =\mX$ for all $\mX\in \R^{d\times n}$.

\begin{definition} We will often consider the following\footnote{The algorithm we develop is, however, not limited to such sketches.} sketching operators $\cA$:

(i) {\bf Right sketch.} Let $\mR\in \R^{n\times n}$ be a random matrix. Define $\cA$ by $\cA \mX = \mX \mR$ (``R-sketch''). Notice that $\cA^\ast \mX = \mX \mR^\top$. In particular, if $R$ is random subset of $[n]$, we can define $\mR = \sum_{j \in R} \eRj \eRj^\top$. The resulting operator $\cA$ (``R-sampling'') satisfies: $\cA = \cA^\ast = \cA^2 = \cP_{\cA}$.  If we let $\pRj \eqdef \Prob{j \in R}$, and instead define $\mR = \sum_{j \in R} \tfrac{1}{\pRj} \eRj \eRj^\top$, then $\E{\mR} = \mI_n$ and hence  $\cA$ is unbiased.\\
(ii) {\bf Left sketch.} Let $\mL\in \R^{d\times d}$ be a random matrix. Define $\cA$ by $\cA \mX = \mL \mX $ (``L-sketch'').  Notice that $\cA^\ast \mX = \mL^\top \mX $. In particular, if $L$ is random subset of $[d]$, we can define $\mL = \sum_{i\in L} \eLi \eLi^\top$. The resulting operator $\cA$ (``L-sampling'') satisfies: $\cA = \cA^\ast = \cA^2 = \cP_{\cA}$. If we let $\pLi\eqdef \Prob{i\in L}$, and instead define $\mL = \sum_{i\in L} \tfrac{1}{\pLi} \eLi \eLi^\top$, then $\E{\mL} = \mI_d$  and hence $\cA$  us unbiased.\\
(iii) {\bf Scaling/Bernoulli.} Let $\xi$ be a Bernoulli random variable, i.e., $\xi = 1$ with probability $\probx $ and $\xi = 0$ with probability $1-\probx $, where $\probx \in [0,1]$.  Define $\cA$ by $\cA \mX = \xi \mX$ (``scaling''). Then $\cA = \cA^\ast = \cA^2 = \cP_{\cA}$. If we instead define $\cA \mX = \frac{1}{\probx } \xi \mX$, then  $\cA$ is unbiased.\\
(iv)   {\bf LR sketch.}  All the above operators can be combined. In particular, we can define $\cA \mX = \xi \mL \mX \mR$. All of the above arise as special cases of this: (i) arises for $\xi\equiv 1$ and $\mL\equiv \mI_d$, (ii) for  $\xi\equiv 1$ and $\mR\equiv \mI_n$, and (iii) for $\mL\equiv \mI_d$ and $\mR\equiv \mI_n$.% One can choose $\xi$, $\mL$ and $\mR$ to be any of the special cases above

\end{definition}

% {\bf Identity in expectation.} We shall also have a need for random linear operators $\cI: \R^{d\times n}\to \R^{d\times n}$ which become an identity in expectation: $\E{\cI \mX} = \mX$ for all $\mX\in \R^{d\times n}$.

\section{Generalized Jacobian Sketching ({\tt GJS})}
We are now ready to describe our method (formalized as Algorithm~\ref{alg:SketchJac}). 
\begin{algorithm}[!h]
\begin{algorithmic}[1]
\State \textbf{Parameters:} Stepsize $\alpha>0$, random projector $\cS$ and unbiased sketch $\cU$
\State \textbf{Initialization:} Choose  solution estimate $x^0 \in \R^d$ and Jacobian estimate $ \mJ^0\in \R^{d\times n}$ 
\For{$k =  0, 1, \dots$}
\State Sample realizations of $\cS$ and $\cU$, and perform sketches $\cS\mG(x^k)$ and $\cU\mG(x^k)$
\State  $\mJ^{k+1} = \mJ^k - \cS(\mJ^k - \mG(x^k))$ \quad \hfill update the Jacobian estimate via  \eqref{eq:nio9h8fbds79kjh}
\State $g^k = \frac1n \mJ^k \eR + \frac1n \cU \left(\mG(x^k) -\mJ^k\right)\eR$  \hfill construct the gradient estimator via \eqref{eq:ni98hffs}
    \State $x^{k+1} = \prox (x^k - \alpha g^k)$ \label{eq:alg_update} \hfill perform the proximal {\tt SGD} step \eqref{eq:PSGD}
\EndFor
\end{algorithmic}
\caption{Generalized {\tt JacSketch} ({\tt GJS}) }
\label{alg:SketchJac}
\end{algorithm}
Let $\cS$ be a random linear operator (e.g., right sketch, left sketch, or scaling)  such that $\cS = \cP_\cS$ and let $\cU$ be an unbiased operator. We propose to construct the gradient estimator as
\begin{equation}\label{eq:ni98hffs}
\compactify g^k = \frac{1}{n} \mJ^k \eR + \frac{1}{n}\cU (\mG(x^k) - \mJ^k) \eR,
\end{equation}
where the matrices $\mJ^k\in \R^{d\times n}$ are constructed iteratively. Note that, taking expectation in $\cU$, we get 
\begin{equation}\label{eq:unbiased_xx}
\compactify \E{g^k} \overset{\eqref{eq:ni98hffs}}{=}\frac{1}{n} \mJ^k \eR +  \frac{1}{n}(\mG(x^k) - \mJ^k) \eR = \frac{1}{n} \mG(x^k)e \overset{\eqref{eq:nbifg98dz}}{=} \nabla f(x^k),\end{equation}
and hence $g^k$ is indeed unbiased. We will construct $\mJ^k$ so that  $\mJ^k\to \mG(x^*)$. By doing so, the variance of $g^k$ decreases throughout the iterations, completely vanishing at $x^*$. The sequence $\{\mJ^k\}$ is updated as follows:
\begin{equation}\label{eq:nio9h8fbds79kjh}\mJ^{k+1} = \arg\min_{\mJ} \left\{ \|\mJ - \mJ^k\| \;:\; \cS \mJ = \cS \mG(x^k) \right\} = \mJ^k - \cS(\mJ^k - \mG(x^k)).\end{equation}
That is, we sketch the Jacobian $\mG(x^k)$, obtaining the sketch $\cS \mG(x^k)$, and seek to use this information to construct a new matrix $\mJ^{k+1}$ which is consistent with this sketch, and as close to $\mJ^k$ as possible. The intuition here is as follows: if we repeated the sketch-and-project process \eqref{eq:nio9h8fbds79kjh} for fixed $x^k$, the matrices $\mJ^k$ would converge to $\mG(x^k)$, at a linear rate~\cite{SIMAX2015, inverse}. This process can be seen as {\tt SGD} applied to a certain quadratic stochastic optimization problem~\cite{ASDA, gower2018stochastic}. Instead, we take just one step of this iterative process, change $x^k$, and repeat. Note that the unbiased sketch $\cU$ in \eqref{eq:ni98hffs} also claims access to $\mG(x^k)$. Specific variants of {\tt GJS} are obtained by choosing specific operators $\cS$ and $\cU$ (see Section~\ref{sec:Special_Cases}).

% We will often construct $\cS$ and $\cU$ in synchrony, so that both operators extract the same (random) information from the Jacobian. For instance, we may use $\cS$ to be the right sampling, and $\cU$ to be the right unbiased sampling. In this case, both of these operators have access to the same (random) columns of the Jacobian. However, it will be sometimes useful to keep these operators decoupled. 

%The latter property is easy to observe: if $x^k=x^*$ and $\mJ^k = \mG(x^*)$, then  $\E{\|g^k-\nabla f(x^*)\|^2} = $

\section{Theory}

We now describe the main result of this paper, which depends on a relaxed strong convexity assumption and a more precise smoothness assumption on $f$.

\begin{assumption} \label{as:smooth_strongly_convex}
Problem \eqref{eq:problem} has a unique minimizer $x^*$, and $f$ is $\sigma$-quasi strongly convex, i.e., 
\begin{equation} 
\compactify f(x^*) \geq   f(y) + \< \nabla f(y) ,x^*- y> + \frac{\sigma}{2} \norm{y - x^*}^2, \quad \forall y\in \R^d, \label{eq:strconv3}
\end{equation}
Functions $f_j$ are convex and $\mM_j$-smooth for some $\mM_j\succeq 0$, i.e.,
\begin{equation} 
\compactify f_j(y) + \< \nabla f_j(y) ,x- y> \leq f_j(x)\leq    f_j(y) + \< \nabla f_j(y) ,x- y> +\frac{1}{2} \norm{y - x}_{{\mM_j}}^2, \quad \forall x,y\in \R^d.  \label{eq:smooth_ass}
\end{equation}
\end{assumption}

Assumption~\ref{eq:smooth_ass} generalizes classical $L$-smoothness, which is obtained in the special case $\mM_j=L\mI_d$. The usefulness of this assumption comes from i) the fact that ERM problems typically satisfy \eqref{eq:smooth_ass} in a non-trivial way~\cite{ESO, gower2019sgd}, ii) our method is able to utilize the full information contained in these matrices for further acceleration (via increased stepsizes). Given matrices $\{\mM_{j}\}$ from Assumption~\ref{as:smooth_strongly_convex}, let $\cM$ be the linear operator defined via
$\left(\cM \mX\right)_{:j} = \mM_j \mX_{:j}$ for $j\in [n]$. It is easy to check that this operator is self-adjoint and positive semi-definite, and that its square root is given by $\left(\cM^{\nicefrac{1}{2}}  \mX\right)_{:j} = \mM_j^{\nicefrac{1}{2}} \mX_{:j}$. The pseudoinverse $\cM^\dagger$ of this operator  plays an important role in our  main result.

\begin{theorem}\label{thm:main} Let Assumption~\ref{as:smooth_strongly_convex} hold. Let $\cB$ be any linear operator commuting with $\cS$, and assume ${\cM^\dagger}^{\nicefrac{1}{2}}$ commutes with $\cS$. Let $\cR$ be any linear operator for which $\cR(\mJ^k) = \cR(\mG(x^*))$ for every $k\geq 0$. 
Define the Lyapunov function 
\begin{eqnarray}\label{eq:Lyapunov}
\Psi^k & \eqdef & \norm{ x^k - x^* }^2 + \alpha \NORMG{ \cB {\cM^\dagger}^{\frac12} \left( \mJ^k - \mG(x^*)\right)},
\end{eqnarray}
where $\{x^k\}$ and $\{\mJ^k\}$ are the random iterates produced by Algorithm~\ref{alg:SketchJac} with stepsize $\alpha>0$. Suppose that $\alpha$ and $\cB$ are chosen so that 
\begin{eqnarray}\label{eq:small_step}
\compactify \frac{2\alpha}{n^2}  \E{ \norm{ \cU  \mX \eR }^2 }   +   \NORMG{  \left(\cI - \E{\cS} \right)^{\frac12}\cB  {\cM^\dagger}^{\frac12} \mX }  &\leq & \compactify (1-\alpha \sigma) \NORMG{ \cB {\cM^\dagger}^{\frac12}\mX }
\end{eqnarray}
whenever  $ \mX\in \Range{\cR}^\perp$ and
\begin{eqnarray}
\compactify \frac{2\alpha}{n^2} \E{  \norm{ \cU  \mX  \eR }^2  } 
+   \NORMG{\left(\E{\cS}\right)^{\frac12}  \cB  {\cM^\dagger}^{\frac12}\mX  }   &  \leq & \compactify \frac{1}{n} \norm{{\cM^\dagger}^{\frac12}\mX}^2.\label{eq:small_step2}
\end{eqnarray}
for all $\mX\in \R^{d\times n}$.  Then  for all $k\geq 0$, we have
$\E{\Psi^{k}}\leq \left( 1-\alpha\sigma\right)^k \Psi^{0}.$
\end{theorem}

The above theorem is very general as it applies to essentially arbitrary random linear operators $\cS$ and $\cU$. It postulates a linear convergence rate of a Lyapunov function composed of two terms: distance of $x^k$ from $x^*$, and weighted distance of the Jacobian $\mJ^k$ from $\mG(x^*)$. Hence, we obtain convergence of both the iterates and the Jacobian to $x^*$ and $\mG(x^*)$, respectively. 
Inequalities \eqref{eq:small_step} and \eqref{eq:small_step2} are mainly assumptions one stepsize $\alpha$, and are used to define suitable weight operator $\cB$. See Lemma~\ref{lem:existence} for a general statement on when these inequalities are satisfied. However, we give concrete and simple answers in all special cases of {\tt GJS} in the appendix. For a summary of how the operator $\cB$ is chosen in special cases, and the particular complexity results derived from this theorem, we refer to Table~\ref{tbl:all_special_cases_theory}.

\begin{remark}We use the trivial choice $\cR\equiv 0$ in almost all special cases. With this choice of $\cR$,  the condition $\cR(\mJ^k) = \cR(\mG(x^*))$ is automatically satisfied, and inequality \eqref{eq:small_step2} is requested to hold for {\em all} matrices $\mX\in \R^{d\times n}$. However, a non-trivial choice of $\cR$ is sometimes useful;  e.g., in the analysis of a subspace variant of {\tt SEGA}~\cite{hanzely2018sega}.  Further, the results of Theorem~\ref{thm:main} can be generalized from a quasi strong convexity to a strong growth condition \cite{karimi2016linear}  on $f$~(see Appendix~\ref{sec:sg}). While interesting, these are not the key results of this work and we therefore suppress them to the appendix. 
\end{remark}

%In different special cases of our method  we define  the operator $\cB$ in a different way so as to ensure that $\cB$ commutes with $\cS$, and that \eqref{eq:small_step} and \eqref{eq:small_step2} hold. We use four concrete ways of defining $\cB$ in special cases: i) $\cB(\mX) = \beta \mX$, where $\beta\in \R$, ii) $\cB(\mX) = \diag(b)\mX$, where $b\in \R^d$, iii)  $\cB(\mX) = \mX\diag(b)$, where $b\in \R^n$ and iv) $\cB(\mX) = \mB \circ \mX$, where $\mB\in \R^{d\times n}$. See Table~\ref{tbl:all_special_cases} for a list of special cases and how $\cB$ is defined in each.

\section{Special Cases} \label{sec:Special_Cases}

As outlined in the introduction, {\tt GJS} (Algorithm~\ref{alg:SketchJac}) is a surprisingly versatile method. In Table~\ref{tbl:all_special_cases} we list {\em 7 existing methods }(in some cases, generalizations of existing methods), and construct also {\em 10 new variance reduced methods.} We also provide a summary of all specialized iteration complexity results, and a guide to the corollaries which state them (see Table~\ref{tbl:all_special_cases_theory} in the appendix).

\begin{table}[!t]
\begin{center}
\tiny
%\begin{adjustbox}{angle=90}
\begin{tabular}{|c|c|c|c|c|c|}
\hline
\multicolumn{2}{|c|}{\bf  Choice of random operators $\cS$ and $\cU$ defining Algorithm~\ref{alg:SketchJac}} &  \multicolumn{4}{|c|}{\bf Algorithm} \\
\hline
$\cS \mX $ & $ \cU \mX $  &  \#  &  Name  & Comment & Sec.  \\
\hline
\hline
 %%%%%%%%%%%%%%
 %%%%%%%%%%%%%%
 %%%%%%%%%%%%%%
$\mX \eRj \eRj^\top$ \text{w.p.}\; $\pRj = \frac1n$ & $\mX n \eRj \eRj^\top $ \text{w.p.}\; $\pRj = \frac1n$ & \ref{alg:SAGA} &  {\tt SAGA} & basic variant of {\tt SAGA} \cite{SAGA} & \ref{sec:saga_basic}  \\
 %%%%%%%%%%%%%%
 %%%%%%%%%%%%%%
 %%%%%%%%%%%%%%
\hline
 %%%%%%%%%%%%%%
 %%%%%%%%%%%%%%
 %%%%%%%%%%%%%%
 $\mX \sum \limits_{j\in R} \eRj \eRj^\top$ \text{w.p.}\; $\pRR$ & $\mX \sum \limits_{j\in R} \frac{1}{\pRj} \eRj \eRj^\top$ \text{w.p.}\; $\pRR$ &  \ref{alg:SAGA_AS_ESO} & {\tt SAGA} & 
 {\tt SAGA} with AS \cite{qian2019saga}   &  \ref{sec:saga_as} \\
  %%%%%%%%%%%%%%
 %%%%%%%%%%%%%%
 %%%%%%%%%%%%%%
\hline
 %%%%%%%%%%%%%%
 %%%%%%%%%%%%%%
 %%%%%%%%%%%%%%
  $\eLi \eLi^\top \mX $ \text{w.p.}\; $\pLi = \frac{1}{d}$
 & $d \eLi \eLi^\top  \mX $ \text{w.p.}\; $\pLi = \frac{1}{d}$ &   \ref{alg:SEGA} & {\tt SEGA} & basic variant  of  {\tt SEGA} \cite{hanzely2018sega} & \ref{sec:sega}  \\
  %%%%%%%%%%%%%%
 %%%%%%%%%%%%%%
 %%%%%%%%%%%%%%
\hline
 %%%%%%%%%%%%%%
 %%%%%%%%%%%%%%
 %%%%%%%%%%%%%%
 $\sum \limits_{i\in L}\eLi \eLi^\top \mX $ \text{w.p.}\; $\pLL$
 & $\sum \limits_{i\in L} \frac{1}{\pLi} \eLi \eLi^\top  \mX $ \text{w.p.}\; $\pLL$ &   \ref{alg:SEGAAS} & {\tt SEGA} &   {\tt SEGA}  \cite{hanzely2018sega} with AS and prox & \ref{sec:sega_is_v1}  \\
   %%%%%%%%%%%%%%
 %%%%%%%%%%%%%%
 %%%%%%%%%%%%%%
\hline
 %%%%%%%%%%%%%%
 %%%%%%%%%%%%%%
 %%%%%%%%%%%%%%
 $
=\begin{cases}
    0  & \text{w.p.}\; 1- \probx \\
    \mX              & \text{w.p.}\;  \probx
\end{cases} $
 & $\sum \limits_{i\in L} \frac{1}{\pLi} \eLi \eLi^\top \mX $ \text{w.p.}\; $\pLL$ & \ref{alg:SVRCD} & {\tt SVRCD} & {\bf NEW}  & \ref{sec:svrcd_is2} \\
    %%%%%%%%%%%%%%
 %%%%%%%%%%%%%%
 %%%%%%%%%%%%%%
\hline
 %%%%%%%%%%%%%%
 %%%%%%%%%%%%%%
 %%%%%%%%%%%%%%
 0 & $\mX \sum \limits_{j\in R} \frac{1}{\pRj} \eRj \eRj^\top$ \text{w.p.}\; $\pRR$ &  \ref{alg:SGD_AS} & {\tt SGD-star} & {\tt SGD-star}~\cite{sigmak} with AS & \ref{sec:SGD-AS-star}  \\
 %%%%%%%%%%%%%%
 %%%%%%%%%%%%%%
 %%%%%%%%%%%%%%
\hline
 %%%%%%%%%%%%%%
 %%%%%%%%%%%%%%
 %%%%%%%%%%%%%%
 $
=\begin{cases}
    0  & \text{w.p.}\; 1- \probx \\
    \mX              & \text{w.p.}\;  \probx
\end{cases} $
 & $\mX \sum \limits_{j\in R} \frac{1}{\pRj} \eRj \eRj^\top$ \text{w.p.}\; $\pRR$ &   \ref{alg:LSVRG-AS} & {\tt LSVRG}  & {\tt LSVRG} \cite{LSVRG} with AS and prox & \ref{sec:LSVRG-AS}  \\
 %%%%%%%%%%%%%%
 %%%%%%%%%%%%%%
 %%%%%%%%%%%%%%
\hline
 %%%%%%%%%%%%%%
 %%%%%%%%%%%%%%
 %%%%%%%%%%%%%%
$
=\begin{cases}
    0  & \text{w.p.}\; 1- \probx \\
    \mX              & \text{w.p.}\;  \probx
\end{cases} $
 &   
 $=\begin{cases}
    0  & \text{w.p.}\; 1- \proby \\
   \frac{1}{\proby} \mX              & \text{w.p.}\;  \proby
\end{cases} $ & \ref{alg:B2} & {\tt B2} & {\bf NEW} & \ref{sec:B2}  \\
  %%%%%%%%%%%%%%
 %%%%%%%%%%%%%%
 %%%%%%%%%%%%%%
\hline
 %%%%%%%%%%%%%%
 %%%%%%%%%%%%%%
 %%%%%%%%%%%%%%
$\mX \sum \limits_{j\in R} \eRj \eRj^\top  $ \text{w.p.}\; $\pRR$
 &   
 $=\begin{cases}
    0  & \text{w.p.}\; 1- \proby \\
   \frac{1}{\proby} \mX              & \text{w.p.}\;  \proby
\end{cases} $ &   \ref{alg:invsvrg}  & {\tt LSVRG-inv} &  {\bf NEW} & \ref{sec:SVRG-1}  \\
 %%%%%%%%%%%%%%
 %%%%%%%%%%%%%%
 %%%%%%%%%%%%%%
\hline
 %%%%%%%%%%%%%%
 %%%%%%%%%%%%%%
 %%%%%%%%%%%%%%
$\sum \limits_{i\in L} \eLi \eLi^\top \mX $ \text{w.p.}\; $\pLL$
 &   
 $=\begin{cases}
    0  & \text{w.p.}\; 1- \proby \\
  \frac{1}{\proby} \mX              & \text{w.p.}\;  \proby
\end{cases} $ & \ref{alg:B_sega} & {\tt SVRCD-inv} & {\bf NEW} &  \ref{sec:SVRCD-inv}\\
 %%%%%%%%%%%%%%
 %%%%%%%%%%%%%%
 %%%%%%%%%%%%%%
\hline
 %%%%%%%%%%%%%%
 %%%%%%%%%%%%%%
 %%%%%%%%%%%%%%
$\mX \sum \limits_{j\in R} \eRj \eRj^\top$ \text{w.p.}\;  $\pRR
$ 
& 
$\sum \limits_{i\in L} \frac{1}{\pLi} \eLi  \eLi^\top \mX$  \text{w.p.}\;  $\pLL
$ 
& \ref{alg:RL} & {\tt RL} & {\bf NEW} &  \ref{sec:RL} \\
 %%%%%%%%%%%%%%
 %%%%%%%%%%%%%%
 %%%%%%%%%%%%%%
\hline
 %%%%%%%%%%%%%%
 %%%%%%%%%%%%%%
 %%%%%%%%%%%%%%
$\sum \limits_{i\in L} \eLi  \eLi^\top \mX$ \text{w.p.}\;  $\pLL
$ 
& 
 $\mX \sum \limits_{j\in R} \frac{1}{\pRj} \eRj \eRj^\top$ \text{w.p.}\;  $\pRR
$ 
& \ref{alg:LR} & {\tt LR} & {\bf NEW} &  \ref{sec:LR} \\
 %%%%%%%%%%%%%%
 %%%%%%%%%%%%%%
 %%%%%%%%%%%%%%
\hline
 %%%%%%%%%%%%%%
 %%%%%%%%%%%%%%
 %%%%%%%%%%%%%%
$ \mI_{L:} \mX \mI_{:R}$ \text{w.p.}\; $\pLL \pRR$
 &   
 $  \mI_{L:}\left( \left(    \pL^{-1} \left(\pR^{-1}\right)^\top  \right) \circ \mX\right)\mI_{:R}$ \text{w.p.}\; $\pLL \pRR$& \ref{alg:saega} & {\tt SAEGA} & {\bf NEW} &  \ref{sec:SAEGA} \\
  %%%%%%%%%%%%%%
 %%%%%%%%%%%%%%
 %%%%%%%%%%%%%%
 \hline
  %%%%%%%%%%%%%%
 %%%%%%%%%%%%%%
 %%%%%%%%%%%%%%
$=\begin{cases}
0 & \text{w.p.}\; 1-\probx  \\
\mX & \text{w.p.}\; \probx \\
\end{cases}$
 &   
 $  \mI_{L:}\left( \left( \pL^{-1} \left(\pR^{-1}\right)^\top\right) \circ \mX\right)\mI_{:R}$ \text{w.p.}\; $\pLL \pRR$&  \ref{alg:svrcdg} & {\tt SVRCDG} & {\bf NEW} & \ref{sec:SVRCDG}  \\
  %%%%%%%%%%%%%%
 %%%%%%%%%%%%%%
 %%%%%%%%%%%%%%
 \hline
  %%%%%%%%%%%%%%
 %%%%%%%%%%%%%%
 %%%%%%%%%%%%%%
  $\sum \limits_{\tR=1}^\TR  \mI_{L_{\tR}:}\mX_{:\NRt} \mI_{:R_{\tR}} $
  &   
$\sum \limits_{\tR=1}^\TR   \left( {(\ptL)^{-1}} {(\ptR)^{-1}}^\top\right) \circ \left(\mI_{L_{\tR}:}\mX_{:\NRt}  \mI_{:R_{\tR}}\right)$ &  \ref{alg:isaega} & {\tt ISAEGA} & {\bf NEW} (reminiscent of \cite{mishchenko201999}) &  \ref{sec:ISAEGA}  \\
 %%%%%%%%%%%%%%
 %%%%%%%%%%%%%%
 %%%%%%%%%%%%%%
\hline
 %%%%%%%%%%%%%%
 %%%%%%%%%%%%%%
 %%%%%%%%%%%%%%
  $\sum \limits_{\tR=1}^\TR  \mI_{L_{\tR}:}\mX_{:\NRt} $ 
  &   
$\sum \limits_{\tR=1}^\TR   \left( {(\ptL)^{-1}} {\eR}^\top\right) \circ \left(\mI_{L_{\tR}:}\mX_{:\NRt}  \right)$ &  \ref{alg:isega} & {\tt ISEGA} &   {\tt ISEGA} \cite{mishchenko201999} with AS&  \ref{sec:ISAEGA}  \\
 %%%%%%%%%%%%%%
 %%%%%%%%%%%%%%
 %%%%%%%%%%%%%%
\hline
 %%%%%%%%%%%%%%
 %%%%%%%%%%%%%%
 %%%%%%%%%%%%%%
  $\mX \mR $  & $\mX \mR \E{\mR}^{-1} $ &   \ref{alg:jacsketch} & {\tt JS}  & {\tt JacSketch} \cite{gower2018stochastic}  with AS and prox & \ref{sec:jacsketch}  \\
 %%%%%%%%%%%%%%
 %%%%%%%%%%%%%%
 %%%%%%%%%%%%%%
\hline
\end{tabular}
%\end{adjustbox}
\end{center}
\caption{Selected special cases of {\tt GJS} (Alg.~\ref{alg:SketchJac}) arising by choosing operators $\cS$ and $\cU$ in particular ways. $R$ is a random subset of $[n]$, $L$ is a random subset of $[d]$, $\pLi = \Prob{i\in L}$, $\pRj = \Prob{j\in R}$.}
\label{tbl:all_special_cases}
\end{table}

$\triangleright$ {\tt SGD-star.} In order to illustrate why variance reduction is needed in the first place, let us start by describing one of the new methods---{\tt SGD-star}  (Algorithm~\ref{alg:SGD_AS})---which happens to be particularly suitable to shed light on this issue. In {\tt SGD-star} we assume that the Jacobian at optimum, $\mG(x^*)$, is known. While this is clearly an unrealistic assumption, let us see where it leads us.
If this is the case, we can choose $\mJ^0 = \mG(x^*)$, and  let $\cS \equiv 0$. This implies that $\mJ^k=\mJ^0$ for all $k$.  We then choose $\cU$ to be the right unbiased sampling operator, i.e.,  $\cU \mX =  \mX  \sum_{j\in R} \frac{1}{\pRj} \eRj \eRj^\top$, which gives
\[\compactify g^k = \frac{1}{n}\sum \limits_{j=1}^n \nabla f_j(x^*)  +\sum \limits_{j\in R^k} \frac{1}{n \pRj} \left( \nabla f_j(x^k) -\nabla f_j(x^*) \right).\]
This method does not need to learn the Jacobian at $x^*$ as it is known, and instead moves in a direction of average  gradient at the optimum, perturbed by a random estimator of the direction $\nabla f(x^k) - \nabla f(x^*)$ formed via sub-sampling $j\in R^k\subseteq [n]$ . What is special about this perturbation? As the method converges, $x^k\to x^*$ and the perturbations converge to zero, for {\em any} realization of the random set $R^k\subseteq [n]$. So,  gradient estimation stabilizes, we get $g^k\to \nabla f(x^*)$, and hence the variance of $g^k$ converges to zero. In view of Corollary~\ref{cor:sgd} of our main result (Theorem~\ref{thm:main}), the iteration complexity of {\tt SGD-star} is $\max_j  \frac{v_j}{\sigma  n \pRj} \log \frac{1}{\epsilon}$, where $\sigma$ is the quasi strong convexity parameter of $f$, and the smoothness constants $v_j$ are defined in Appendix~\ref{sec:SGD-AS-star}.

Since knowing  $\mG(x^*)$ is unrealistic,  {\tt GJS} is instead  {\em learning} these perturbations on the fly. Different variants of {\tt GJS} do this differently, but ultimately all attempt to learn the gradients $\nabla f_j(x^*)$ and use this information to stabilize the gradient estimation. Due to space restrictions, we do not describe all remaining 9 new methods in the main body of the paper, let alone the all 17 methods. We will briefly outline 2 more (not necessarily the most interesting) new methods:

$\triangleright$ {\tt SVRCD.} This method belongs to the {\tt RCD} variety, and constructs the gradient estimator via the rule
\[\compactify g^k = h^k+\sum \limits_{i\in L^k}  \frac{1}{\pLi}(\nabla_i f(x^k) - h_i^k) \eLi\; ,\]
where $L^k \subseteq [d]$  is sampled afresh in each iteration.
The auxiliary vector $h^k$ is updated using a simple biased coin flip: $h^{k+1} =   h^k$ with probability $1- \probx$, and   $h^{k+1} =  \nabla f(x^k) $        with probability $\probx$. So, a full pass over all coordinates is made in each iteration with probability $\probx$, and a partial derivatives $\nabla_i f(x^k)$ for $i\in L^k$  are computed in each iteration. This method has  a similar structure to {\tt LSVRG}, which instead sub-sampling coordinates sub-samples functions $f_j$ for $j\in R^k$ (see Table~\ref{tbl:all_special_cases}). The iteration complexity of this method is $\left(\frac{1}{\probx} + \max_i \frac{1}{\pLi}\frac{4 m_i}{\sigma}\right)\log \frac{1}{\epsilon} $, where $m_i$ is a smoothness parameter of $f$ associated with coordinate $i$ (see Table~\ref{tbl:all_special_cases_theory} and Corollary~\ref{cor:svrcd}).

$\triangleright$  {\tt ISAEGA.} In~\cite{mishchenko201999}, a strategy of running {\tt RCD} on top of a parallel implementation  of optimization algorithms such as {\tt GD}, {\tt SGD} or {\tt SAGA} was proposed. Surprisingly, it was shown that the runtime of the overall algorithm is unaffected whether one computes and communicates {\em all entries} of the stochastic gradient on each worker, or only a {\em fraction} of all entries of size  inversely proportional to the number of all workers. However, {\tt ISAGA}~\cite{mishchenko201999} (distributed {\tt SAGA} with {\tt RCD} on top of it), as proposed, requires the gradients with respect to the data owned by a given machine to be zero at the optimum. On the other hand, {\tt ISEGA}~\cite{mishchenko201999} does not have the issue, but it requires a computation of the exact partial derivatives on each machine and thus is expensive. As a special case of {\tt GJS} we propose {\tt ISAEGA} -- a method which cherry-picks the best properties from both {\tt ISAGA} (allowing for stochastic partial derivatives) and {\tt ISEGA} (not requiring zero gradients at the optimum). Further, we present the method in the arbitrary sampling paradigm. See Appendix~\ref{sec:ISAEGA} for more details.

\section{Experiments}
We perform extensive numerical testing for various special cases of Algorithm~\ref{alg:SketchJac}. Due to space limitations, we only give a quick taste using a single experiment here. The complete numerical evaluation is presented in Appendix~\ref{sec:extra_exp}.

In particular, in Appendix~\ref{sec:sega_exp} we demonstrate that {\tt SEGA} with importance sampling outperforms both basic {\tt SEGA} and proximal gradient descent, often significantly. Next,  Appendix~\ref{sec:svrcd_exp} demonstrates that, as predicted by theory, convergence speed of {\tt SVRCD} is influenced by the choice of $\probx$ only weakly. Further, in Appendix~\ref{sec:ISAEGA_exp} we demonstrate the claimed linear parallel scaling of {\tt ISAEGA} (in the sense of~\cite{mishchenko201999}). Lastly, in Appendix~\ref{sec:extra_lsvrg} we demonstrate the superiority of {\tt LSVRG} with importance sampling (a new variant of {\tt LSVRG} obtained here) to plain {\tt LSVRG}, plain {\tt SAGA} and {\tt SAGA} with importance sampling. We only outline the last experiment here; a complete description is given in Appendix~\ref{sec:extra_lsvrg}.

We consider a logistic regression problem on LibSVM~\cite{chang2011libsvm}. In order to conduct fair testing, we only compare methods where the expected minibatch size is fixed. We set $\E{|R|}=\tau$ and compare {\tt LSVRG} with $\probx = \frac{1}{2n}$, and {\tt SAGA} with importance sampling (imp) and uniform sampling (unif). The results are presented in Figure~\ref{fig:LSVRG_main}.

\begin{figure}[!h]
\centering
\begin{minipage}{0.3\textwidth}
  \centering
\includegraphics[width =  \textwidth ]{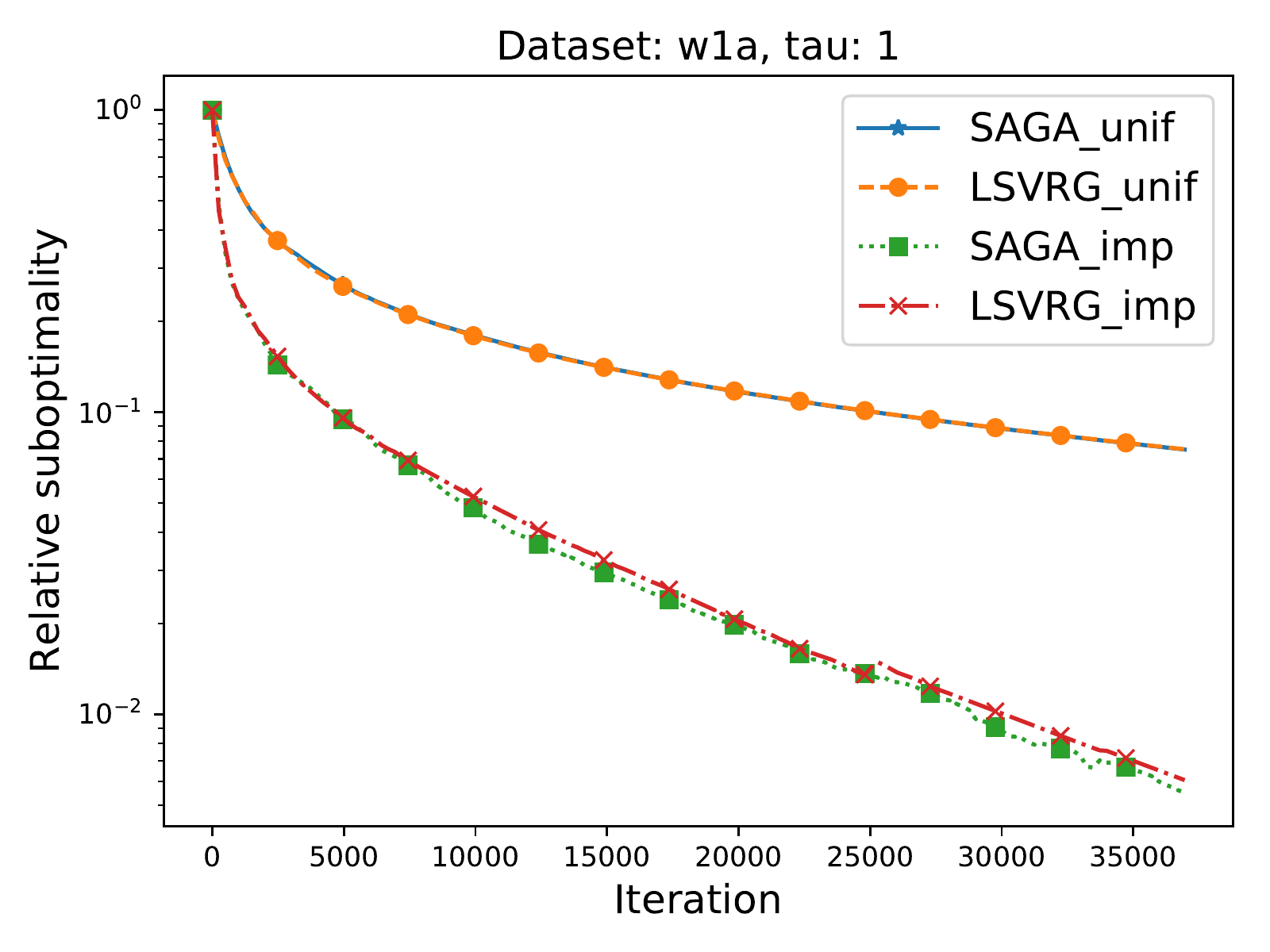}
        %\caption{ Residual vs. iteration  }\label{fig:bl_ex_flops}
\end{minipage}%
\begin{minipage}{0.3\textwidth}
  \centering
\includegraphics[width =  \textwidth ]{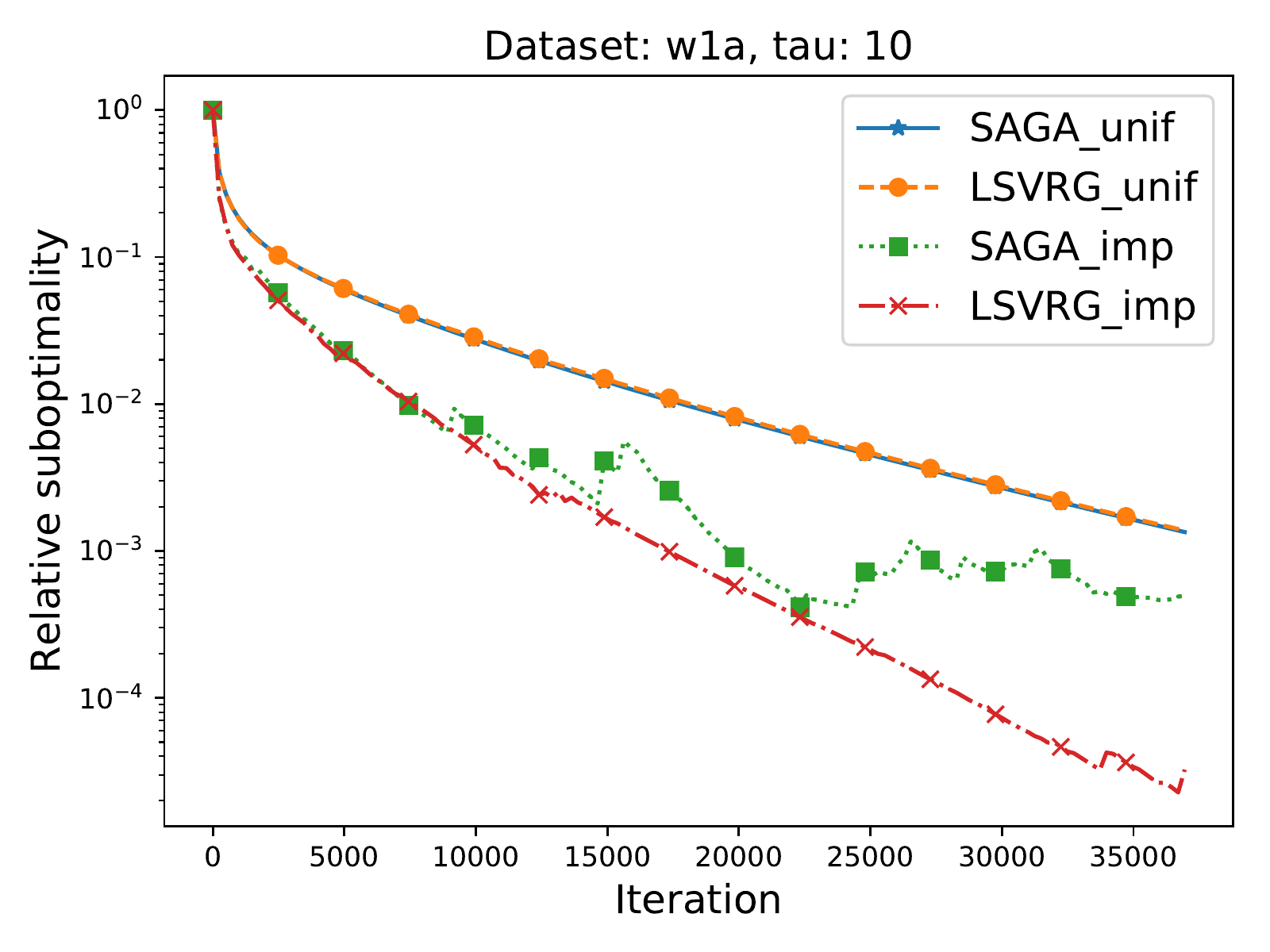}
        %\caption{ Residual vs. iteration  }\label{fig:bl_ex_flops}
\end{minipage}%
\begin{minipage}{0.3\textwidth}
  \centering
\includegraphics[width =  \textwidth ]{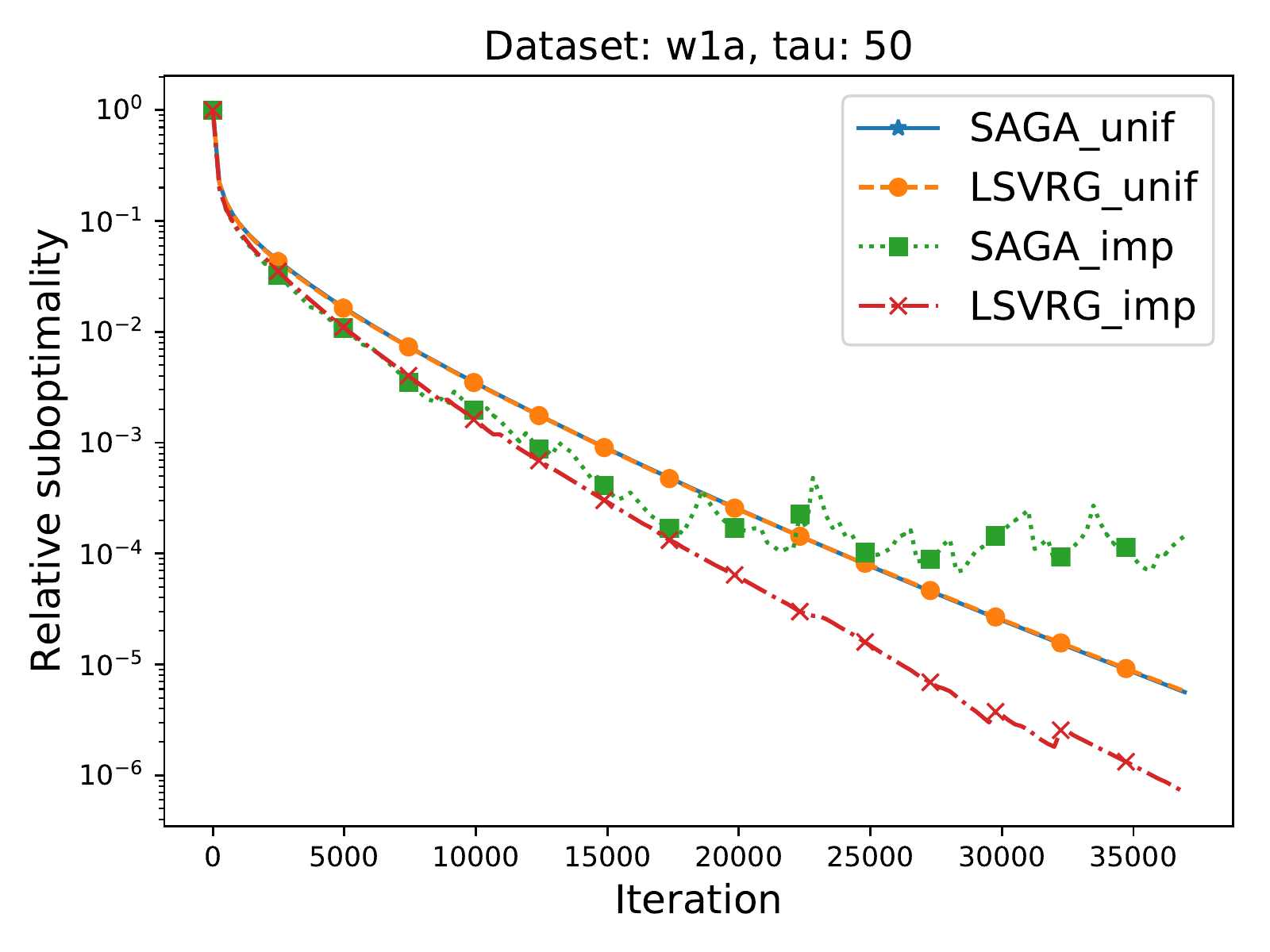}
        %\caption{ Residual vs. iteration  }\label{fig:bl_ex_flops}
\end{minipage}%
\\
\begin{minipage}{0.3\textwidth}
  \centering
\includegraphics[width =  \textwidth ]{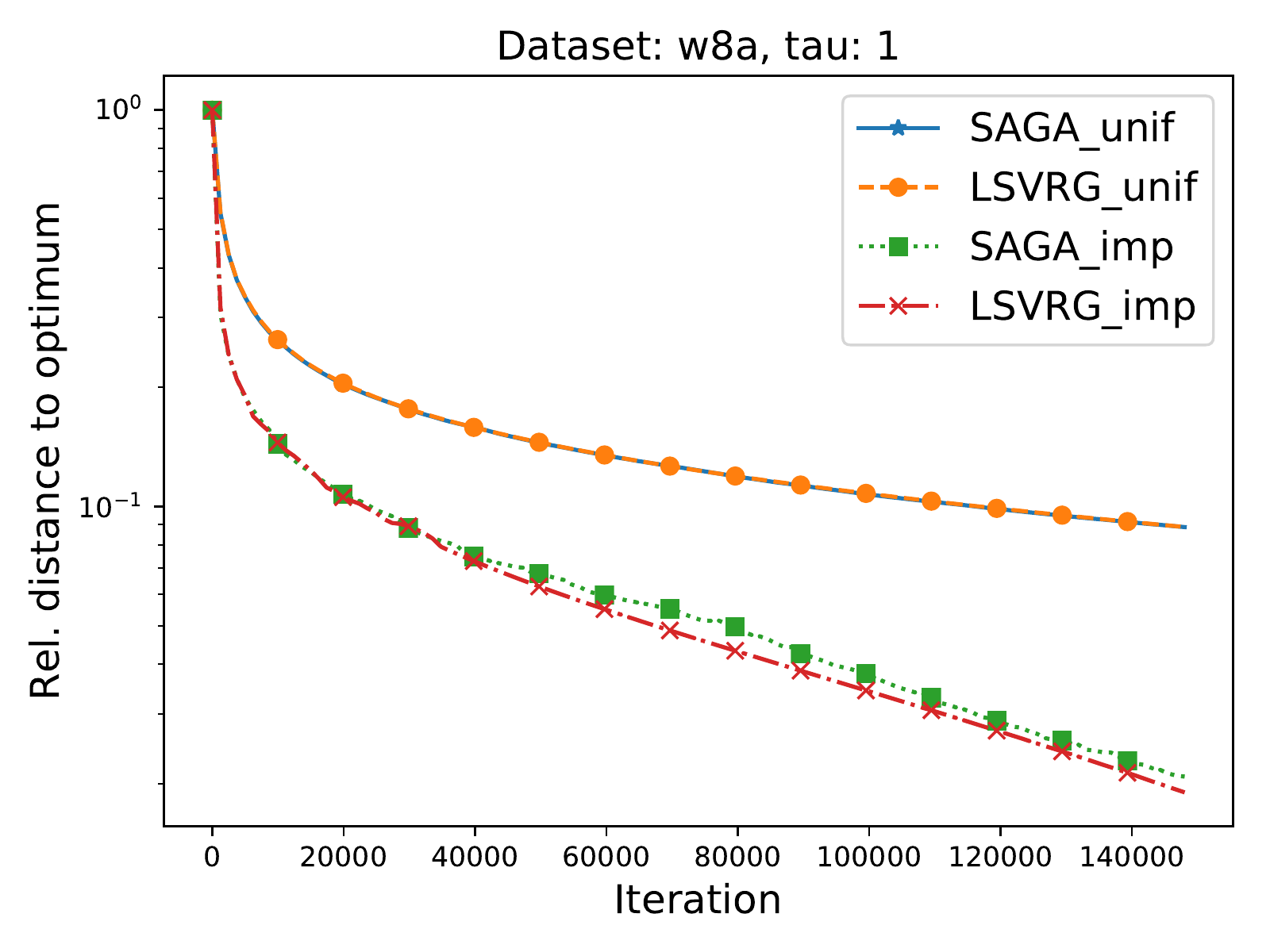}
        %\caption{ Residual vs. iteration  }\label{fig:bl_ex_flops}
\end{minipage}%
\begin{minipage}{0.3\textwidth}
  \centering
\includegraphics[width =  \textwidth ]{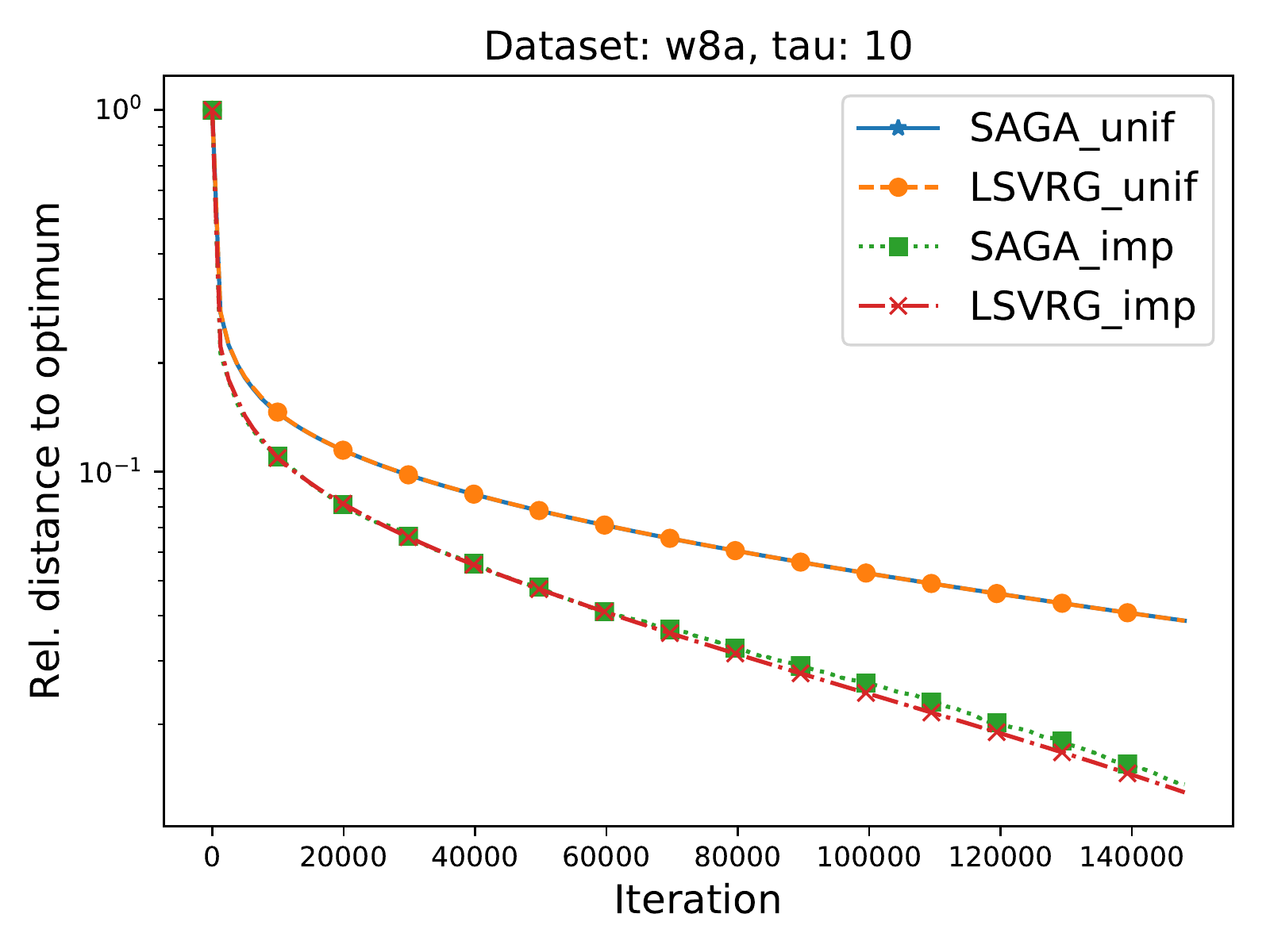}
        %\caption{ Residual vs. iteration  }\label{fig:bl_ex_flops}
\end{minipage}%
\begin{minipage}{0.3\textwidth}
  \centering
\includegraphics[width =  \textwidth ]{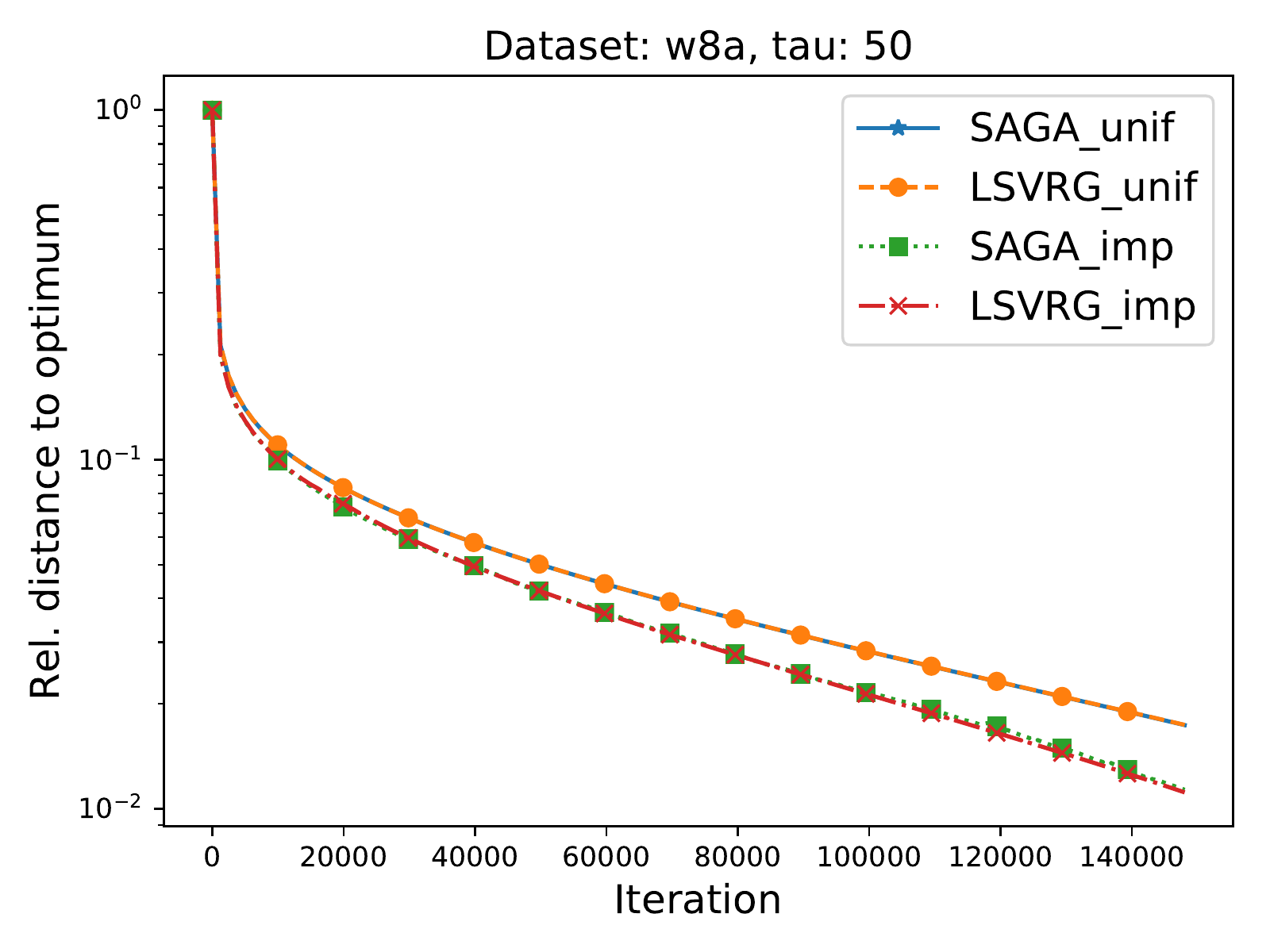}
        %\caption{ Residual vs. iteration  }\label{fig:bl_ex_flops}
\end{minipage}%s
\caption{Comparison of {\tt LSVRG} \& {\tt SAGA} with importance and uniform sampling.} 
\label{fig:LSVRG_main}
\end{figure}

In all cases, {\tt LSVRG} with importance sampling  was the fastest method, while uniform {\tt LSVRG} and {\tt SAGA} performed almost identically. The gain from importance sampling is noticable for small $\tau$. For larger $\tau$, importance sampling is less significantly superior. Note that this behavior is predicted by theory. However, our experiments indicate the superiority of {\tt LSVRG} to  {\tt SAGA} in the importance sampling setup. In particular, proposed stepsize $\gamma$ (see Appendix~\ref{sec:extra_lsvrg}) is often too large for {\tt SAGA}. Note that the optimal stepsize requires the prior knowledge of the quasi strong convexity constant\footnote{Or more generally, knowledge of the strong growth constant is required; see Appendix~\ref{sec:sg}.} $\sigma$, which is in our case unknown (see the importance serial sampling proposed in~\cite{gower2018stochastic}). While one can still estimate it as the $L2$ regularization constant,  this would be a weak estimate and yield suboptimal performance.


\clearpage
\vskip 0.2in
\bibliography{../literature}
\clearpage 
\appendix

\part*{Appendix \\ \Large  One Method to Rule Them All: Variance Reduction for Data, Parameters and Many New Methods }

\section{Table of Contents}

For easier navigation through the paper and the appendices, we include a table of contents.

\tableofcontents

\clearpage
\section{Summary of Complexity Results}

We provide a comprehensive table for faster navigation through special cases and their iteration complexities. In particular, for each special case of {\tt GJS}, we provide the leading complexity term  (i.e., a $\log \nicefrac{1}{\varepsilon}$ factor is omitted in all results) and a reference to the corresponding corollary where this result is established. We also indicate how the operator $\cB$ appearing in the Lyapunov function is picked (this is not needed to run the method; it is only used in the analysis). All details can be found later in the Appendix.

\begin{table}[!h]
\begin{center}
%\begin{adjustbox}{angle=90}
\footnotesize
\begin{tabular}{|c|c|c|c|c|}
\hline
 \multicolumn{2}{|c|}{\bf Algorithm} & \multicolumn{3}{|c|}{\bf Theory} \\
\hline
 \#  &  Name  &  Cor.\ of Thm~\ref{thm:main} &  $\cB \mX $ & Leading complexity term (i.e., $\log \tfrac{1}{\varepsilon}$ factor omitted) \\
\hline
\hline
 %%%%%%%%%%%%%%
 %%%%%%%%%%%%%%
 %%%%%%%%%%%%%%
 \ref{alg:SAGA} &  {\tt SAGA} &  Corollary~\ref{cor:saga} & $\beta \mX$ & $ \nR + \frac{4 m}{\sigma}$ \\
 %%%%%%%%%%%%%%
 %%%%%%%%%%%%%%
 %%%%%%%%%%%%%%
\hline
 %%%%%%%%%%%%%%
 %%%%%%%%%%%%%%
 %%%%%%%%%%%%%%
\ref{alg:SAGA_AS_ESO} & {\tt SAGA}  & Corollary~\ref{cor:saga_as2}  & $\mX\diag(b)$& $\max  \limits_j  \left(   \frac{1}{\pRj} +   \frac{1}{  \pRj} \frac{4 v_j }{\sigma n}  \right)$ \\
  %%%%%%%%%%%%%%
 %%%%%%%%%%%%%%
 %%%%%%%%%%%%%%
\hline
 %%%%%%%%%%%%%%
 %%%%%%%%%%%%%%
 %%%%%%%%%%%%%%
   \ref{alg:SEGA} & {\tt SEGA} &  Corollary~\ref{cor:sega} &  $\beta \mX$ & $ \dL+  \dL        \frac{4  m}{\sigma}  $ \\
  %%%%%%%%%%%%%%
 %%%%%%%%%%%%%%
 %%%%%%%%%%%%%%
\hline
 %%%%%%%%%%%%%%
 %%%%%%%%%%%%%%
 %%%%%%%%%%%%%%
 \ref{alg:SEGAAS} & {\tt SEGA} &  Corollary~\ref{cor:sega_is_11} &  $\diag(b)\mX$ & $  \max  \limits_i\left(    \frac{1}{\pLi} +  \frac{1}{\pLi} \frac{4 m_i}{\sigma} \right) $ \\
  %%%%%%%%%%%%%%
 %%%%%%%%%%%%%%
 %%%%%%%%%%%%%%
\hline
 %%%%%%%%%%%%%%
 %%%%%%%%%%%%%%
 %%%%%%%%%%%%%%
  \ref{alg:SVRCD} & {\tt SVRCD} & Corollary~\ref{cor:svrcd} & $\beta \mX$& $\frac{1}{\probx} + \max  \limits_i \frac{1}{\pLi}\frac{4 m_i}{\sigma}   $ \\
 %%%%%%%%%%%%%%
 %%%%%%%%%%%%%%
 %%%%%%%%%%%%%%
\hline
 %%%%%%%%%%%%%%
 %%%%%%%%%%%%%%
 %%%%%%%%%%%%%%
  \ref{alg:SGD_AS} & {\tt SGD-star} & Corollary~\ref{cor:sgd}  & $0$&  $\max  \limits_j  \frac{1}{ \pRj}\frac{v_j}{\sigma  n} $ \\
  %%%%%%%%%%%%%%
 %%%%%%%%%%%%%%
 %%%%%%%%%%%%%%
\hline
 %%%%%%%%%%%%%%
 %%%%%%%%%%%%%%
 %%%%%%%%%%%%%%
  \ref{alg:LSVRG-AS} & {\tt LSVRG}  & Corollary~\ref{cor:lsvrg_as} & $\beta \mX$& $\frac{1}{ \probx } +  \max  \limits_j \frac{1}{ \pRj} \frac{4v_j }{\sigma n }  $ \\
   %%%%%%%%%%%%%%
 %%%%%%%%%%%%%%
 %%%%%%%%%%%%%%
\hline
 %%%%%%%%%%%%%%
 %%%%%%%%%%%%%%
 %%%%%%%%%%%%%%
\ref{alg:B2} & {\tt B2} & Corollary~\ref{cor:B2} &  $\beta \mX$ & $ \frac{1}{\probx} + \frac{1}{\proby}\frac{4 m}{\sigma}  $ \\
  %%%%%%%%%%%%%%
 %%%%%%%%%%%%%%
 %%%%%%%%%%%%%%
\hline
 %%%%%%%%%%%%%%
 %%%%%%%%%%%%%%
 %%%%%%%%%%%%%%
  \ref{alg:invsvrg}  & {\tt LSVRG-inv}  & Corollary~\ref{cor:inverse_svrg} & $ \mX \diag(b)$ & $\max \limits_j  \frac{1}{\pRj}   + \frac{1}{\proby}\frac{4 m}{\sigma}  $ \\
 %%%%%%%%%%%%%%
 %%%%%%%%%%%%%%
 %%%%%%%%%%%%%%
\hline
 %%%%%%%%%%%%%%
 %%%%%%%%%%%%%%
 %%%%%%%%%%%%%%
\ref{alg:B_sega} & {\tt SVRCD-inv} & Corollary~\ref{cor:SVRCD-inv} &  $\diag(b) \mX$&  $ \max \limits_i \frac{1}{\pLi}  + \frac{1}{\proby}\frac{4 m}{\sigma}    $ \\
 %%%%%%%%%%%%%%
 %%%%%%%%%%%%%%
 %%%%%%%%%%%%%%
\hline
 %%%%%%%%%%%%%%
 %%%%%%%%%%%%%%
 %%%%%%%%%%%%%%
 \ref{alg:RL} & {\tt RL} &  Corollary~\ref{cor:RL} &  $\mX \diag(b)$ & $\max \limits_{i,j} \left( \frac{1}{\pRj} + \frac{1}{\pLi} \frac{4 m_i^j}{\sigma }  \right)$  \\
 %%%%%%%%%%%%%%
 %%%%%%%%%%%%%%
 %%%%%%%%%%%%%%
 \hline
 %%%%%%%%%%%%%%
 %%%%%%%%%%%%%%
 %%%%%%%%%%%%%%
 \ref{alg:LR} & {\tt LR} &  Corollary~\ref{cor:LR} &  $\diag(b) $ \mX  &  $\max \limits_{i,j} \left(  \frac{1}{\pLi}  + \frac{1}{\pRj} \frac{4 v_j}{\sigma }\right) $ \\
 %%%%%%%%%%%%%%
 %%%%%%%%%%%%%%
 %%%%%%%%%%%%%%
 \hline
 %%%%%%%%%%%%%%
 %%%%%%%%%%%%%%
 %%%%%%%%%%%%%%
\ref{alg:saega} & {\tt SAEGA} & Corollary~\ref{cor:saega} &  $\mB \circ \mX$& $ \max \limits_{i,j}   \left(\frac{1}{\pLi \qRj}  +  \frac{1}{\pLi \qRj} \frac{4 m^j_i}{\sigma n }   \right)$ \\
  %%%%%%%%%%%%%%
 %%%%%%%%%%%%%%
 %%%%%%%%%%%%%%
 \hline
  %%%%%%%%%%%%%%
 %%%%%%%%%%%%%%
 %%%%%%%%%%%%%%
  \ref{alg:svrcdg} & {\tt SVRCDG} & Corollary~\ref{cor:svrcdg} &  $\beta  \mX$& $ \frac{1}{\probx} + \max \limits_{i,j}   \frac{1}{\pLi \qRj}  \frac{4 m^j_i}{\sigma n}  $ \\
  %%%%%%%%%%%%%%
 %%%%%%%%%%%%%%
 %%%%%%%%%%%%%%
 \hline
  %%%%%%%%%%%%%%
 %%%%%%%%%%%%%%
 %%%%%%%%%%%%%%
 \ref{alg:isaega} & {\tt ISAEGA} & Corollary~\ref{cor:isaega}  & $\mB \circ \mX$& $  \max \limits_{j\in N_\tR, i,\tR}   \left(\frac{1}{\ptLi \qtRj}  + \left( 1+ \frac{1}{ n\ptLi \qtRj} \right) \frac{4 m_i^j}{\sigma}  \right)$ \\
     %%%%%%%%%%%%%%
 %%%%%%%%%%%%%%
 %%%%%%%%%%%%%%
\hline
  %%%%%%%%%%%%%%
 %%%%%%%%%%%%%%
 %%%%%%%%%%%%%%
 \ref{alg:isega} & {\tt ISEGA} & Corollary~\ref{cor:isega}  & $\mB \circ \mX$& $  \max \limits_{j\in N_\tR, i,\tR}   \left(\frac{1}{\ptLi |\NRt|}  + \left( 1+ \frac{1}{ n\ptLi |\NRt|} \right) \frac{4 m_i^j}{\sigma}  \right)$ \\
     %%%%%%%%%%%%%%
 %%%%%%%%%%%%%%
 %%%%%%%%%%%%%%
\hline
 %%%%%%%%%%%%%%
 %%%%%%%%%%%%%%
 %%%%%%%%%%%%%%
  \ref{alg:jacsketch} & {\tt JS}  & Corollary~\ref{cor:jacsketh} & $\beta\mX\mB$& $ \frac{ 4n^{-1}\ugly \sigma^{-1} \lambda_{\max}\left(  \mB^\top \E{\mR} \mB  \right) + \lambda_{\max} \left(\mB^\top \mB\right) }{\lambda_{\min} \left(\mB^\top \E{\mR} \mB \right) }  $ \\
\hline
\end{tabular}
%\end{adjustbox}
\end{center}
\caption{Iteration complexity of selected special cases of {\tt GJS} (Algorithm~\ref{alg:SketchJac}). Whenever $m$ appears in a result, we assume that $\mM_j = m \mI_d$ for all $j$ (i.e., $f_j$ is $m$-smooth).  Whenever $m_i$ appears in a result, we assume that $f$ is $\mM$-smooth with $\mM= \diag(m_1,\dots,m_d)$. Whenever $m_i^j$ appears in a result, we assume that $\mM_j=\diag(m_1^j,\dots,m_d^j)$. Quantities $\pLi$ for $i \in [d]$, $\pRj$ for $j\in [n]$, $\probx$ and $\proby$ are probabilities defining the algorithms. }
\label{tbl:all_special_cases_theory}
\end{table}

 \clearpage
 \section{Table of Frequently Used Notation} \label{sec:notation_table}

Due to the generality of Algorithm~\ref{alg:SketchJac}, which gives rise to a large number existing and new methods in particular cases, we appreciate that this paper is rather notation-heavy -- and this is still the case after us having spent a considerable amount of time simplifying and optimizing the notation. In an attempt to make the paper more easy to read, here we include a table of the most frequently used notation. We recommend the reader to consult this table while studying our results.

{
%  \footnotesize
 \begin{longtable}{| p{.20\textwidth} | p{.80\textwidth}| } 
  \caption{Frequently used notation.}  \label{tbl:notation}\\
  \hline
  \multicolumn{2}{|c|}{\bf Functions} \\
   \hline
   $f_j :\R^d \to \R$ & a differentiable convex function  \\
   $f :\R^d \to \R$ & $f(x) = \tfrac{1}{n}\sum_{j=1}^n f_j(x)$ \\
      $\nabla f :\R^d \to \R^d$ & gradient of $f$\\
   $\nabla_i f :\R^d \to \R$ & $i$-th partial derivative of $f$ \\
   $\psi$  & regularizer $\R^d \to \R\cup \{+\infty\}$ (a proper closed convex function)  \\
   $x^* \in \R^d$ & unique minimizer of $f+\psi$  \\ 
   $\sigma$ & (a positive) quasi strong  convexity constant of $f$ (see \eqref{eq:strconv3}) \\
  \hline
 \multicolumn{2}{|c|}{\bf Sets} \\
 \hline
 $[n]$ & the set $\{1,2,\dots,n\}$ \\
  $[d]$ & the set $\{1,2,\dots,d\}$ \\
  $R$ &  random subset (``sampling'') of $[n]$ \\
   $R^k$ &  random subset (``sampling'') of $[n]$ drawn at iteration $k$ \\
  $L$ & a random subset (``sampling'') of $[d]$ \\
   $L^k$ &  random subset (``sampling'') of $[d]$ drawn at iteration $k$ \\ 
  $\pRj$ & probability that $j\in R$ \\
    $\pLi$ & probability that $i\in L$ \\   
 \hline
  \multicolumn{2}{|c|}{\bf Spaces $\R^n$ and $\R^d$} \\
   \hline
  $\eR \in \R^n$ & vector of all ones in $\R^n$ \\
  $\eL \in \R^d$ & vector of all ones in $\R^d$ \\
    $\eRj \in \R^n$ & $j$-th standard unit basis vector in $\R^n$ \\
    $\eLi \in \R^n$ & $i$-th standard unit basis vector in $\R^d$ \\   
    $x^k \in \R^d$  & the $k$-th iterate produced by Algorithm~\ref{alg:SketchJac} \\
    $\pL \in \R^d$ & the vector $(\pL_1, \dots,  \pL_d)$ \\
        $\ptL \in \R^d$ & the vector $(\ptL_1, \dots,  \ptL_d)$ \\
    $\pR \in \R^n$ & the vector $(\pR_1, \dots,  \pR_n)$ \\
      $\ptR \in \R^n$ & the vector $(\ptR_1, \dots,  \ptR_n)$ \\
     $\qR \in \R^n$ & the vector $(\qR_1, \dots,  \qR_n)$ \\
         $\qtR \in \R^n$ & the vector $(\qtR_1, \dots,  \qtR_n)$  \\
        $v \in \R^n$ & any vector for which~\eqref{eq:ESO_saga} holds  \\
    $\langle x, y \rangle$ & standard Euclidean inner product \\
     $\|x\|$ &  standard Euclidean norm of vector $x$: $\|x\| = \langle x, x \rangle^{\nicefrac{1}{2}}$\\
          $x^{-1}$ &  elementwise inverse of $x$ \\
     $g^k$ & estimator of the gradient $\nabla f(x^k)$ produced by Algorithm~\ref{alg:SketchJac} \\
  \hline    
 \multicolumn{2}{|c|}{\bf Matrices in $\R^{d\times d}$, $\R^{d\times n}$ and $\R^{n\times n}$} \\
 \hline 
 $\mI_d \in \R^{d\times d}$ & $d\times d$ identity matrix \\
  $\mI_n \in \R^{n\times n}$ & $n\times n$ identity matrix \\
  $\mG(x) \in \R^{d\times n}$ & the Jacobian matrix, i.e., $\mG(x) = [\nabla f_1(x), \dots, \nabla f_n(x)]$ \\
  $\mJ^k \in \R^{d\times n}$ & estimator of the Jacobian produced by Algorithm~\ref{alg:SketchJac} \\
 $\mM_j \in \R^{d\times d}$ & smoothness matrix of $f_j$ (if $\mM_j = m^j \mI_d$, then this specializes to $m^j$-smoothness)\\
 $\mR \in \R^{n\times n}$ & a random matrix we use to multiply $\mJ$ or $\mG$ from the right \\
 $\mR_{R} \in \R^{n\times n}$ & the random matrix  $\mR_R \eqdef \sum_{j\in R} \eRj \eRj^\top$  \\
 $\mL  \in \R^{d\times d}$ & a random matrix we use to multiply $\mJ$ or $\mG$ from the left \\
  $\mL_{L} \in \R^{d\times d}$ & the random matrix  $\mL_L \eqdef \sum_{i\in L} \eLi \eLi^\top$  \\
 $\langle \mX, \mY \rangle$ & trance inner product of matrices $\mX$ and $\mY$: $\langle \mX, \mY \rangle \eqdef \Tr{\mX^\top \mY}$ \\
  $\|\mX\| $ & Frobenius norm of matrix $\mX$: $\|\mX\| = \langle \mX, \mX\rangle^{\nicefrac{1}{2}} $\\
  $\mX \circ \mY$ & Hadamard product: $(\mX \circ \mY)_{ij} = \mX_{ij} \mY_{ij}$  \\
    $\mX \otimes \mY$ & Kronecker product  \\
  $\diag(x)$ & diagonal matrix with vector $x$ on the diagonal \\
    $\PR \in \R^{n\times n}$ & Matrix defined by $\PR_{jj'}=\Prob{j\in R, j'\in R}$\\
        $\PtR \in \R^{n\times n}$ & Matrix defined by $\PtR_{jj'}=\Prob{j\in R_\tR, j'\in R_\tR}$\\
  \hline
 \multicolumn{2}{|c|}{\bf Linear operators $\R^{d\times n} \to \R^{d\times n}$} \\
 \hline
  $\cA$ & a generic  linear operator \\ 
  $\cA^*$ & the adjoint of $\cA$: $\langle \cA \mX, \mY \rangle \equiv \langle  \mX, \cA^*\mY \rangle $ for all $\mX,\mY\in \R^{d\times n}$\\ 
   $\cA^\dagger$ & the Moore Penrose pseudoinverse of $\cA$ \\ 
   $\Range{\cA}$ & image (range space) of $\cA$: $\Range{\cA} \eqdef \{\cA \mX \;:\; \mX\in \R^{d\times n}\}$ \\
   $\Range{\cA}^\top$ & orthogonal complement of $\Range{\cA}$ \\   
      $\Null{\cA}$ & kernel (null space) of $\cA$: $\Null{\cA} \eqdef \{ \mX \in \R^{d\times n} \;:\; \cA \mX = 0\}$ \\
 $\cI$ & identity operator: $\cI \mX \equiv \mX$\\ 
  $\cU$ & any unbiased operator: $\E{\cU \mX} \equiv \mX$, i.e., $\E{\cU} \equiv \cI$\\ 
 $\cS$ &  any random projection operator \\
 $\cM$  & operator defined via $(\cM \mX)_{:j} = \mM_j \mX_{:j}$ \\
 $\cB$   & (a technical) operator used to define the Lyapunov function \eqref{eq:Lyapunov}   \\
  $\cR$   & (a technical) operator such that $\mJ^k - \mG(x^*)\in \Range{\cR}$  \\

 \hline  
 \multicolumn{2}{|c|}{\bf Miscellaneous} \\
 \hline
 $\alpha$ & stepsize used in  Algorithm~\ref{alg:SketchJac} \\
  $\Gamma$  & Random operator  $\Gamma : \R^{d\times n} \to \R^d$ defined by $\Gamma \mX = \cU \mX \eR$ \\
  $\prox(x)$ & the proximal operator of $\psi$: $\prox(x) \eqdef \argmin_{u\in \R^d} \{ \alpha\psi(u) + \frac{1}{2}\|u-x\|^2\}$ \\
 \hline  
 \end{longtable}
 }

 \clearpage
 \section{Additional Experiments \label{sec:extra_exp}}

\subsection{{\tt SEGA} and {\tt SVRCD} with importance sampling \label{sec:sega_exp}}
In Sections~\ref{sec:sega_is_v1} and~\ref{sec:svrcd_is2} we develop an arbitrary (and thus importance in special case) sampling for {\tt SEGA}, as well as new method {\tt SVRCD} with arbitrary sampling.  In this experiment, we compare them to its natural competitors -- basic {\tt SEGA} from~\cite{hanzely2018sega} and proximal gradient descent. 

Consider artificial quadratic minimization with regularizer $\psi$ being an indicator of the unit ball\footnote{In such case, proixmal operator of $\psi$ becomes a projection onto the unit ball.}:
\[ f(x) = x^\top \mM x  - b^\top x, \quad \psi(x) = \begin{cases} x& 0\leq 1 \\ \infty & \| x\| > 1\end{cases} .
\]
 Specific choices of $\mM,b$ are given by by Table~\ref{tbl:quadratics}. As both {\tt SEGA} and {\tt SVRCD} (from Section~\ref{sec:sega_is_v1} and~\ref{sec:svrcd_is2}) require a diagonal smoothness matrix, we shall further consider vector $m$ such that the upped bound $\mM\preceq \diag(m)$ holds. As the choice of $m$ is not unique, we shall choose the one which minimizes $\sum_{i=1}^d m_i$ for importance sampling and $m = \lambda_{\max}(\mM) \eL$  for uniform. Further, stepsize $\gamma = \frac{1}{4\sum_{i=1}^d m_i}$ was chosen in each case.
 Figure~\ref{fig:sega_cmp} shows the results of this experiment. As theory suggests, importance sampling for both {\tt SEGA} and {\tt SVRCD} outperform both plain {\tt SEGA} and proximal gradient always. The performance difference depends on the data; the closer $\mM$ is to a diagonal matrix with non-uniform elements, the larger stronger is the effect of importance sampling.

 \begin{table}[!h]
\begin{center}
\begin{tabular}{|c|c|c|}
\hline
Type & $\mM $ & $b$ \\
 \hline
 1   & $\diag\left(1.3^{[d]}\right)$ & $\gamma u$  \\
 \hline
  2   & $\diag((d,1,1,\dots, 1))$ & $\gamma u$ \\
\hline
3   & $\diag\left(1.1^{[d]}\right)+\mN \mN^\top \frac{1.1^d}{1000 d}$, $\mN~\sim N(0,\mI)$ & $\gamma  u$\\
\hline
4   & $\mN \mN^\top$, $\mN~\sim N(0, \mI)$ & $\gamma u$ \\
\hline
\end{tabular}
\end{center}
\caption{Four types of quadratic problems. We choose $u \sim N(0,\mI_d)$, and  $\gamma$ to be such that $\|\gamma \mM^{-1} u\| = \nicefrac32$. Notation $c^{[d]}$ stands for a vector $(c, c^2, \dots c^d)$.}
\label{tbl:quadratics}
\end{table}

\begin{figure}[!h]
\centering
\begin{minipage}{0.25\textwidth}
  \centering
\includegraphics[width =  \textwidth ]{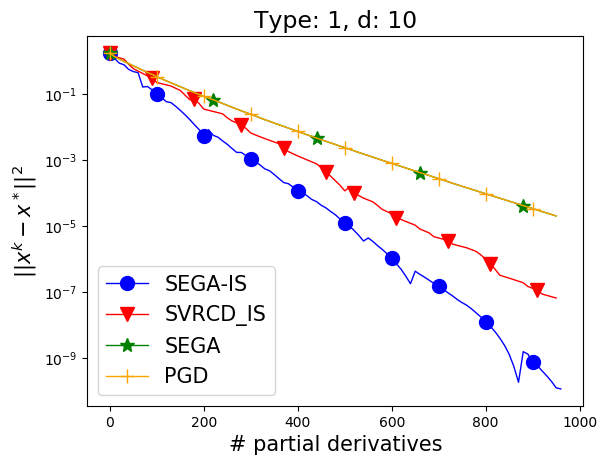}
        %\caption{ Residual vs. iteration  }\label{fig:bl_ex_flops}
\end{minipage}%
\begin{minipage}{0.25\textwidth}
  \centering
\includegraphics[width =  \textwidth ]{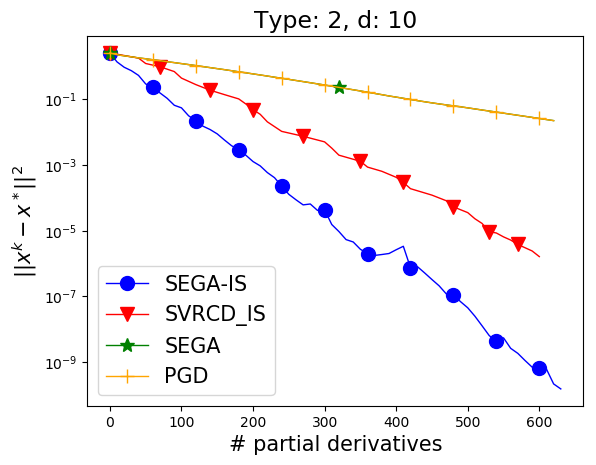}
        %\caption{ Residual vs. iteration  }\label{fig:bl_ex_flops}
\end{minipage}%
\begin{minipage}{0.25\textwidth}
  \centering
\includegraphics[width =  \textwidth ]{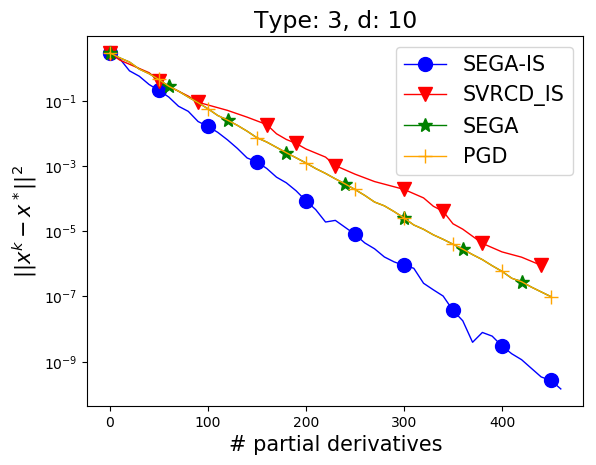}
        %\caption{ Residual vs. iteration  }\label{fig:bl_ex_flops}
\end{minipage}%
\begin{minipage}{0.25\textwidth}
  \centering
\includegraphics[width =  \textwidth ]{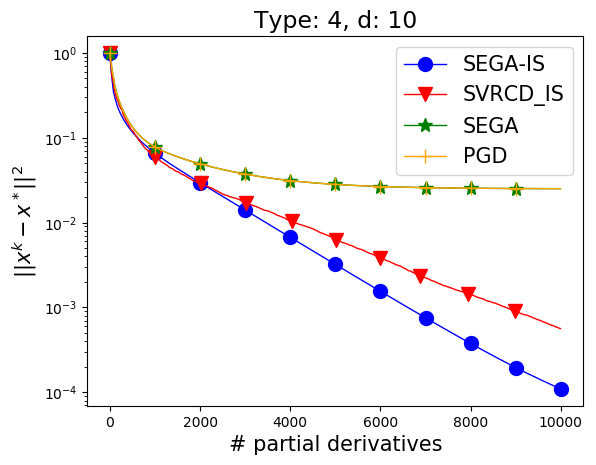}
        %\caption{ Residual vs. iteration  }\label{fig:bl_ex_flops}
\end{minipage}%
\\
\begin{minipage}{0.25\textwidth}
  \centering
\includegraphics[width =  \textwidth ]{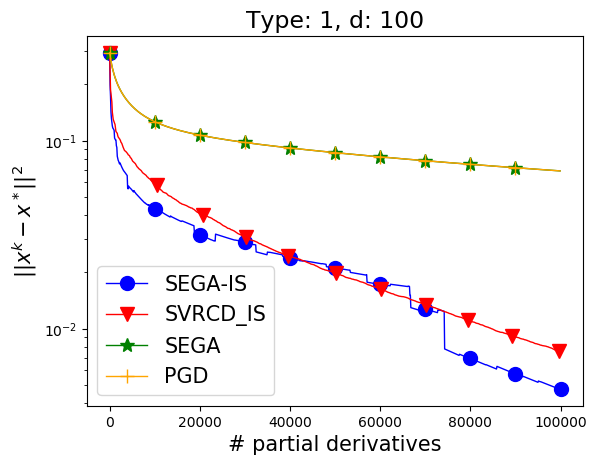}
        %\caption{ Residual vs. iteration  }\label{fig:bl_ex_flops}
\end{minipage}%
\begin{minipage}{0.25\textwidth}
  \centering
\includegraphics[width =  \textwidth ]{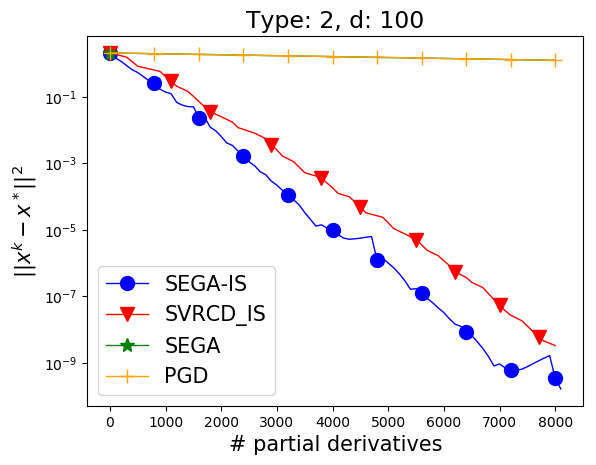}
        %\caption{ Residual vs. iteration  }\label{fig:bl_ex_flops}
\end{minipage}%
\begin{minipage}{0.25\textwidth}
  \centering
\includegraphics[width =  \textwidth ]{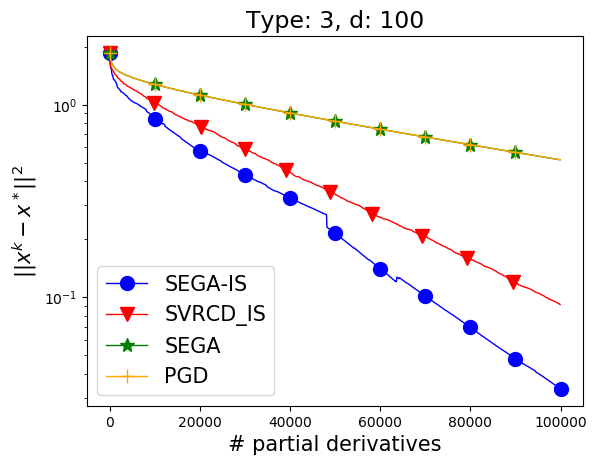}
        %\caption{ Residual vs. iteration  }\label{fig:bl_ex_flops}
\end{minipage}%
\begin{minipage}{0.25\textwidth}
  \centering
\includegraphics[width =  \textwidth ]{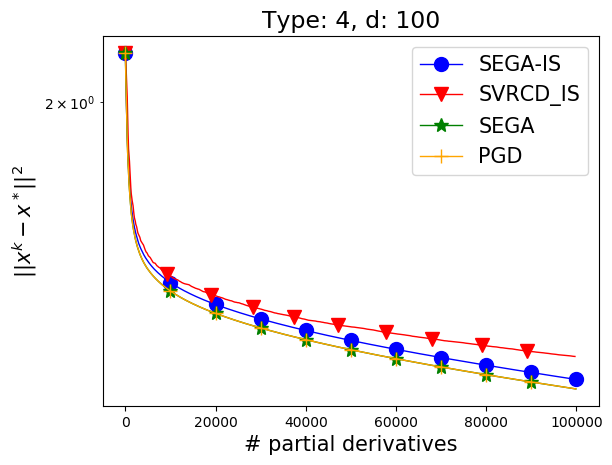}
        %\caption{ Residual vs. iteration  }\label{fig:bl_ex_flops}
\end{minipage}%
\\
\begin{minipage}{0.25\textwidth}
  \centering
\includegraphics[width =  \textwidth ]{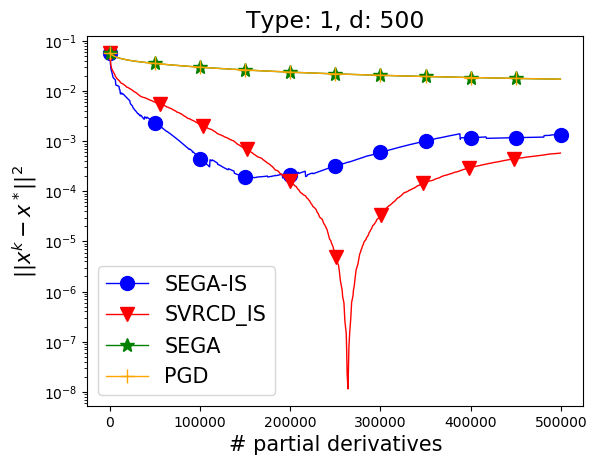}
        %\caption{ Residual vs. iteration  }\label{fig:bl_ex_flops}
\end{minipage}%
\begin{minipage}{0.25\textwidth}
  \centering
\includegraphics[width =  \textwidth ]{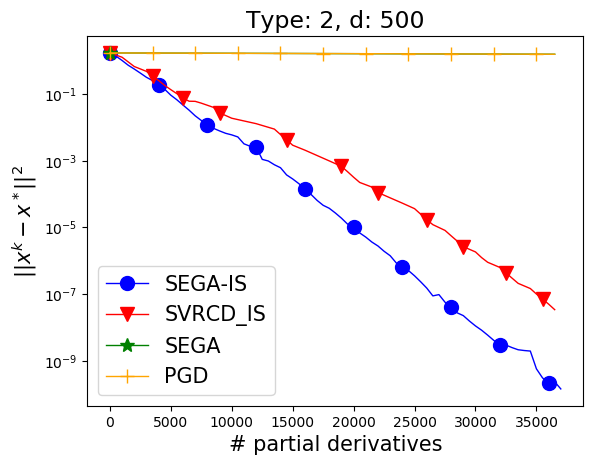}
        %\caption{ Residual vs. iteration  }\label{fig:bl_ex_flops}
\end{minipage}%
\begin{minipage}{0.25\textwidth}
  \centering
\includegraphics[width =  \textwidth ]{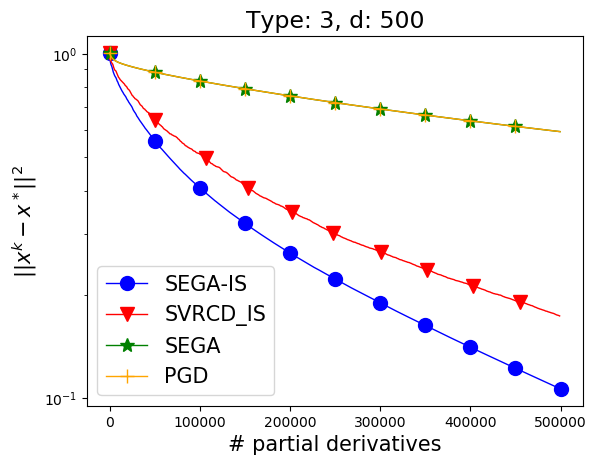}
        %\caption{ Residual vs. iteration  }\label{fig:bl_ex_flops}
\end{minipage}%
\begin{minipage}{0.25\textwidth}
  \centering
\includegraphics[width =  \textwidth ]{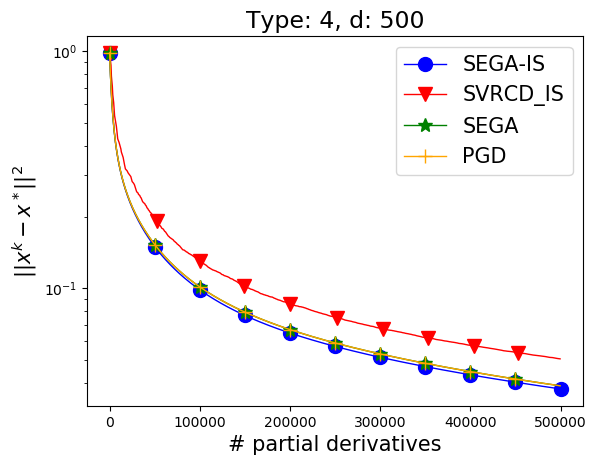}
        %\caption{ Residual vs. iteration  }\label{fig:bl_ex_flops}
\end{minipage}%
\caption{Comparison of {\tt SEGA-AS}, {\tt SVRCD-AS}, {\tt SEGA} and proximal gradient on 4 quadratic problems given by Table~\ref{tbl:quadratics}. {\tt SEGA-AS}, {\tt SVRCD-AS} and {\tt SEGA} compute single partial derivative each iteration ({\tt SVRCD} computes all of them with probability $\probx$), {\tt SEGA-AS}, {\tt SVRCD-AS} with probabilities proportional to diagonal of $\mM$.  }
\label{fig:sega_cmp}
\end{figure}

\subsection{{\tt SVRCD}: Effect of $\probx$ \label{sec:svrcd_exp}}
In this experiment we demonstrate very broad range of $\probx$ can be chosen to still attain almost best possible rate for {\tt SVRCD} for problems from Table~\ref{tbl:quadratics} and $m,\gamma$ as described in Section~\ref{sec:sega_exp} Results can be found in Figure~\ref{fig:svrcd}. They indeed show that in many cases, varying $\probx$ from $\frac1n$ down to $\frac{2\lambda_{\min}(\mM)}{\sum_{i=1}^d m_i}$ does not influences the complexity significantly. However, too small $\probx$ leads to significantly slower convergence. Note that those findings are in accord with Corollary~\ref{cor:svrcd}. Similar results were shown in~\cite{LSVRG} for {\tt LSVRG}. 

\begin{figure}[!h]
\centering
\begin{minipage}{0.25\textwidth}
  \centering
\includegraphics[width =  \textwidth ]{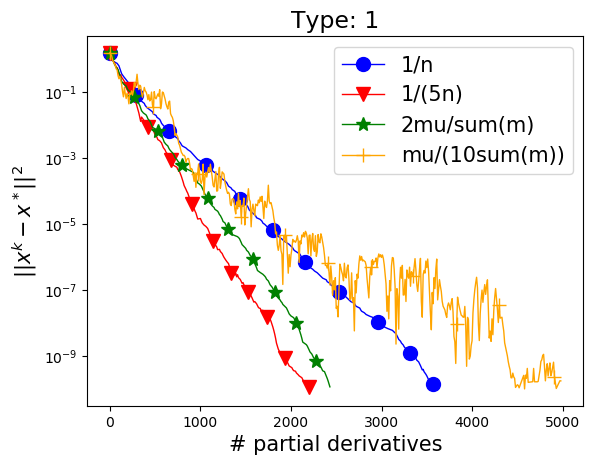}
        %\caption{ Residual vs. iteration  }\label{fig:bl_ex_flops}
\end{minipage}%
\begin{minipage}{0.25\textwidth}
  \centering
\includegraphics[width =  \textwidth ]{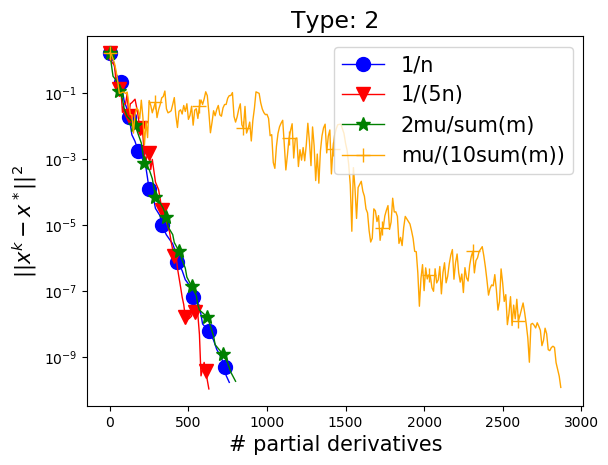}
        %\caption{ Residual vs. iteration  }\label{fig:bl_ex_flops}
\end{minipage}%
\begin{minipage}{0.25\textwidth}
  \centering
\includegraphics[width =  \textwidth ]{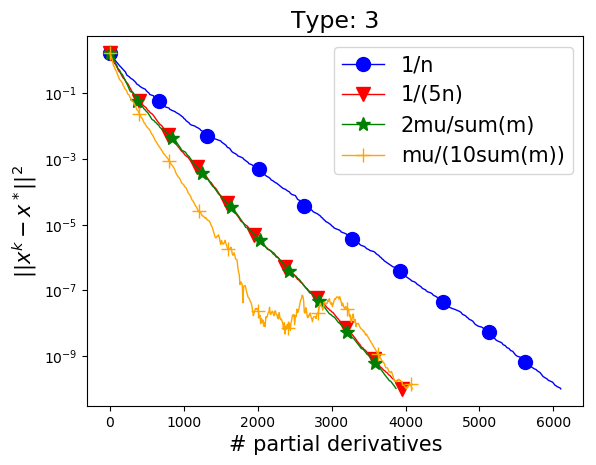}
        %\caption{ Residual vs. iteration  }\label{fig:bl_ex_flops}
\end{minipage}%
\begin{minipage}{0.25\textwidth}
  \centering
\includegraphics[width =  \textwidth ]{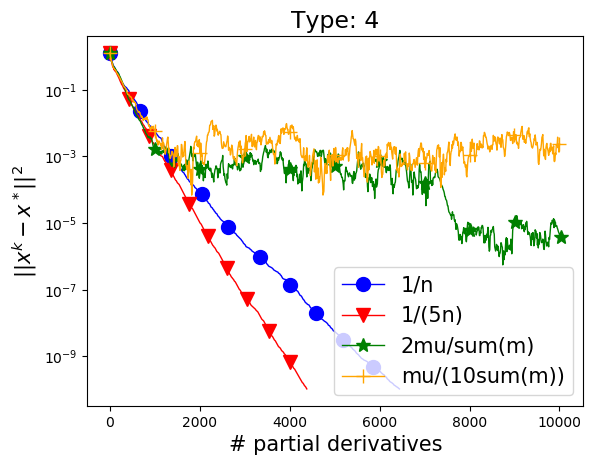}
        %\caption{ Residual vs. iteration  }\label{fig:bl_ex_flops}
\end{minipage}%
\\
\begin{minipage}{0.25\textwidth}
  \centering
\includegraphics[width =  \textwidth ]{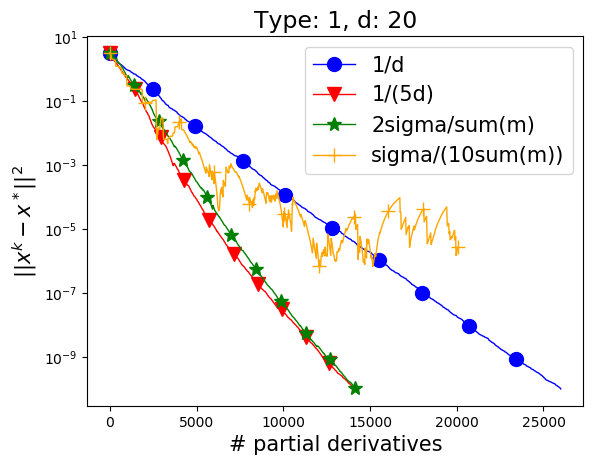}
        %\caption{ Residual vs. iteration  }\label{fig:bl_ex_flops}
\end{minipage}%
\begin{minipage}{0.25\textwidth}
  \centering
\includegraphics[width =  \textwidth ]{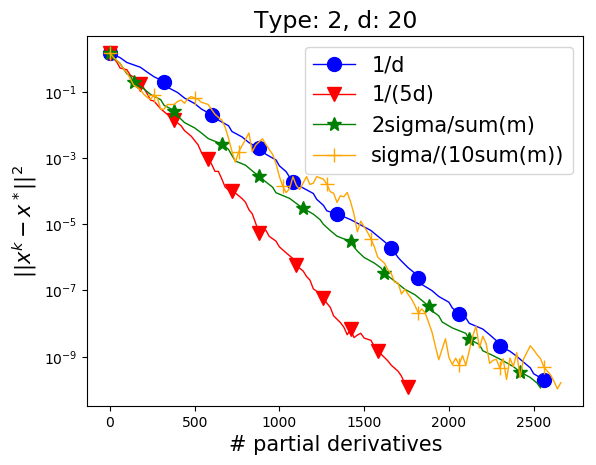}
        %\caption{ Residual vs. iteration  }\label{fig:bl_ex_flops}
\end{minipage}%
\begin{minipage}{0.25\textwidth}
  \centering
\includegraphics[width =  \textwidth ]{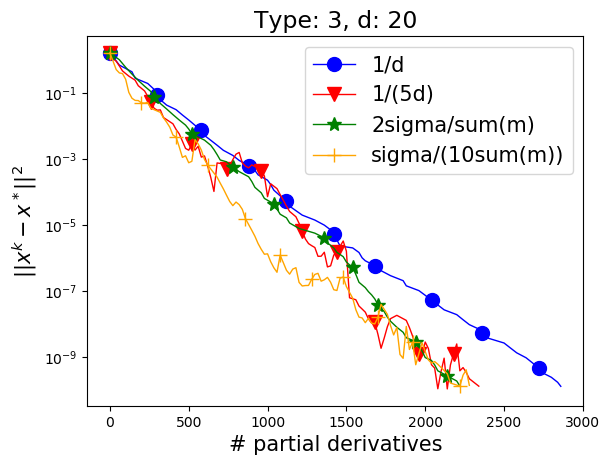}
        %\caption{ Residual vs. iteration  }\label{fig:bl_ex_flops}
\end{minipage}%
\begin{minipage}{0.25\textwidth}
  \centering
\includegraphics[width =  \textwidth ]{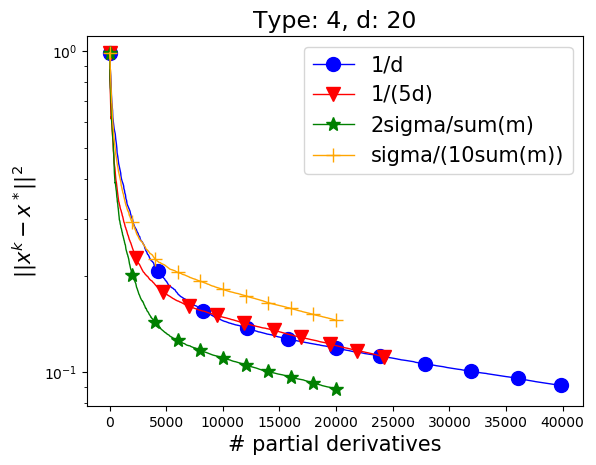}
        %\caption{ Residual vs. iteration  }\label{fig:bl_ex_flops}
\end{minipage}%
\\
\begin{minipage}{0.25\textwidth}
  \centering
\includegraphics[width =  \textwidth ]{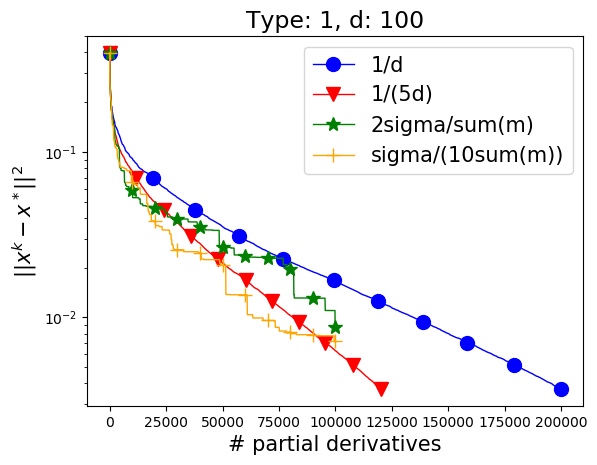}
        %\caption{ Residual vs. iteration  }\label{fig:bl_ex_flops}
\end{minipage}%
\begin{minipage}{0.25\textwidth}
  \centering
\includegraphics[width =  \textwidth ]{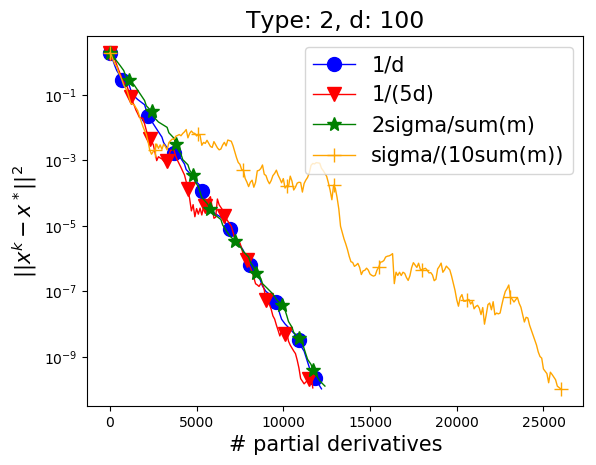}
        %\caption{ Residual vs. iteration  }\label{fig:bl_ex_flops}
\end{minipage}%
\begin{minipage}{0.25\textwidth}
  \centering
\includegraphics[width =  \textwidth ]{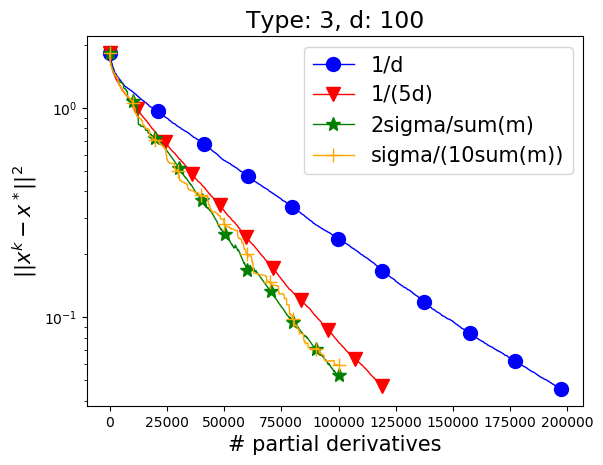}
        %\caption{ Residual vs. iteration  }\label{fig:bl_ex_flops}
\end{minipage}%
\begin{minipage}{0.25\textwidth}
  \centering
\includegraphics[width =  \textwidth ]{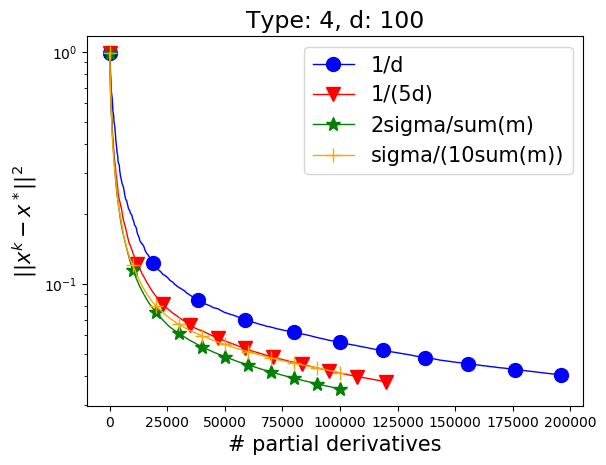}
        %\caption{ Residual vs. iteration  }\label{fig:bl_ex_flops}
\end{minipage}%
\\
\caption{The effect of $\probx$ on convergence rate of {\tt SVRCD} on quadratic problems from Table~\ref{tbl:quadratics}. In every case, probabilities were chosen proportionally to the diagonal of $\mM$ and only a single partial derivative is evaluated in $\cS$.}
\label{fig:svrcd}
\end{figure}

\subsection{{\tt ISAEGA} \label{sec:ISAEGA_exp}}

In this section we test a simple version of {\tt ISAEGA} (Algorithm~\ref{alg:isaega})\footnote{The full description of {\tt ISAEGA}, together with convergence guarantees are provided in Section~\ref{sec:ISAEGA}}. As mentioned, {\tt ISAEGA} is an algorithm for distributed optimization which, at each iteration, computes a subset of partial derivatives of stochastic gradient on each machine, and constructs corresponding Jacobian estimate and stochastic gradient.

For simplicity, we consider only the simple version which assumes $\mM_j = m \mI_d$ for all $j$ (i.e. we do not do importance sampling), and we suppose that $|R_\tR |=1$ always for all $\tR$ (i.e. each machine always looks at a single function from the local finite sum). Further, we consider $\psi(x)=0$. Corollary~\ref{cor:isaega} shows that, if the condition number of the problem is not too small, {\tt ISAEGA} with $|L_{\tR} | \approx \frac1\TR$ (where $\TR$ is a number of parallel units) enjoys, up to small constant factor, same rate as {\tt SAGA} (which is, under a convenient smoothness, the same rate as the convergence rate of gradient descent). Thus, {\tt ISAEGA} scales linearly in terms of partial derivative complexity in parallel setup. In other words, given that we have twice more workers, each of them can afford to evaluate twice less partial derivatives\footnote{Practical implications of the method are further explained in~\cite{mishchenko201999}.}. The experiments we propose aim to verify this claim. 

We consider l2 regularized logistic regression (for the binary classification). In particular, 
\[
\forall j: \quad f_j(x) \eqdef  \log \left(1+\exp\left(\mA_{j,:}x\cdot  y_i\right) \right)+\frac{\lambda}{2} \| x\|^2,
\]
where $\mA\in \R^{n\times d}$ is a data matrix, $y\in \{-1,1\}^{n}$ is a vector of labels and $\lambda\geq 0$ is the regularization parameter.  
Both $\mA,y$ are provided from LibSVM~\cite{chang2011libsvm} datasets: {\tt a1a}, {\tt a9a}, w1a, {\tt w8a}, {\tt gisette}, {\tt madelon}, {\tt phishing} and {\tt mushrooms}. Further,  $\mA$ was normalized such that $\| \mA_{j,:}\|^2=1$. 
Next, it is known that $f_j$ is $(\frac14+ \lambda)$-smooth, convex, while $f$ is $\lambda$-strongly convex. Therefore, as a stepsize for all versions of ${\tt ISAEGA}$, we set $\gamma = \frac{1}{6\lambda + \frac{3}{2}}$ (this is an approximation of theoretical stepsize). 

In each experiment, we compare 4 different setups for {\tt ISEAGA} -- given by 4 different values of $\TR$. Given a value of $\TR$, we set $|L_\tR|=\frac1\TR$ for all $\tR$. Further, we always sample $L_\tR$ uniformly. The results are presented in Figure~\ref{fig:ISAEGA}. Indeed, we observe the almost perfect parallel linear scaling.

For completeness, we provide dataset sized in Table~\ref{tbl:libsvm}. 

 \begin{table}[!h]
\begin{center}
\begin{tabular}{|c|c|c|}
\hline
Name & $n $ & $d$ \\
 \hline
  \hline
{\tt a1a}   & $1605$ & $123$  \\
  \hline
{\tt a9a}   & $32561$ & $123$  \\
  \hline
{\tt w1a}   & $2477$ & $300$  \\
  \hline
{\tt w8a}   & $49749$ & $300$  \\
  \hline
{\tt gisette}   & $6000$ & $5000$  \\
  \hline
{\tt madelon}   & $2000$ & $500$  \\
  \hline
{\tt phishing}   & $11055$ & $68$  \\
  \hline
{\tt mushrooms}   & $8124$ & $112$  \\
\hline
\end{tabular}
\end{center}
\caption{Table of LibSVM data used for our experiments. }
\label{tbl:libsvm}
\end{table}

\begin{figure}[!h]
\centering
\begin{minipage}{0.24\textwidth}
  \centering
\includegraphics[width =  \textwidth ]{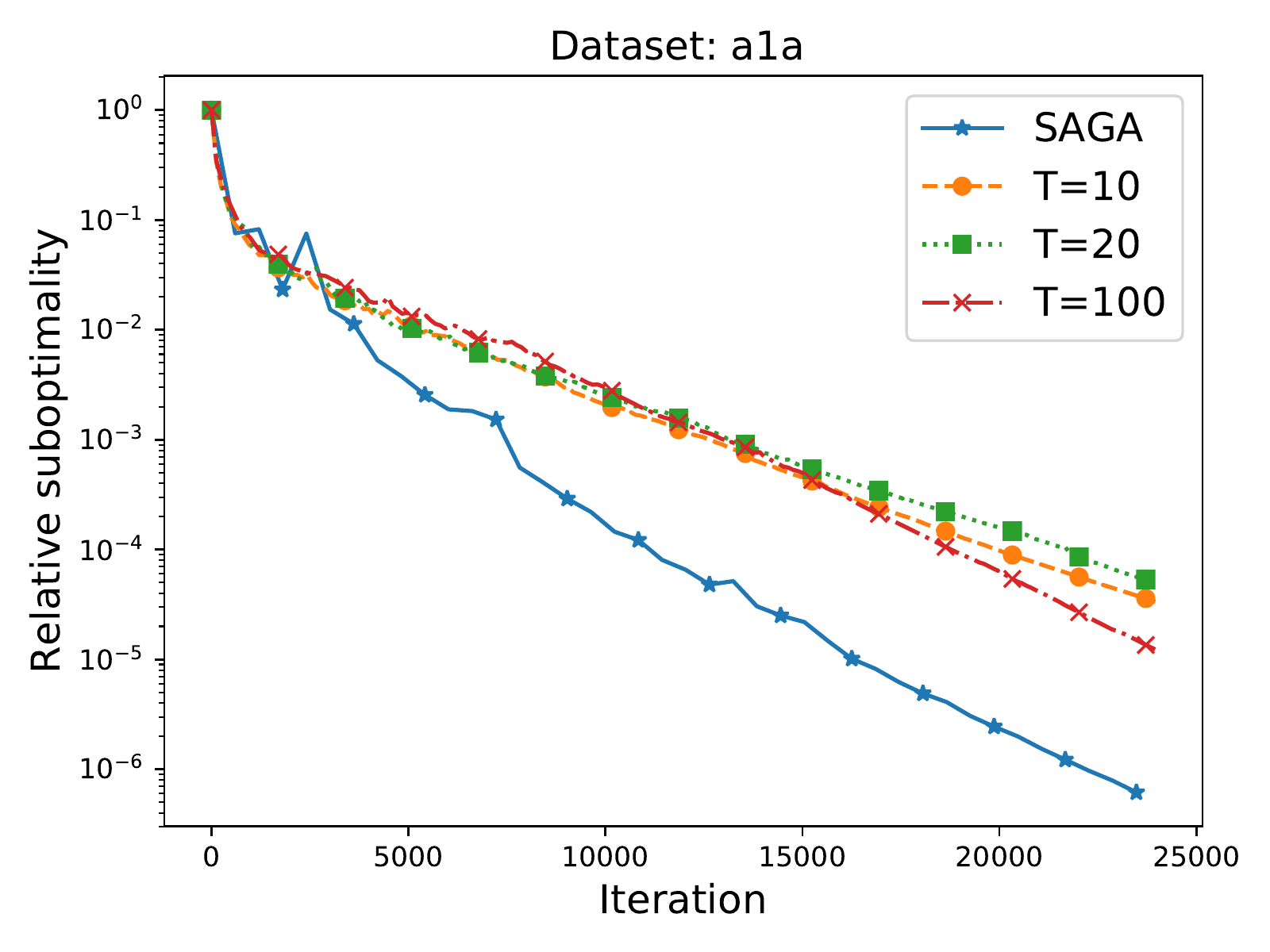}
        %\caption{ Residual vs. iteration  }\label{fig:bl_ex_flops}
\end{minipage}%
\begin{minipage}{0.24\textwidth}
  \centering
\includegraphics[width =  \textwidth ]{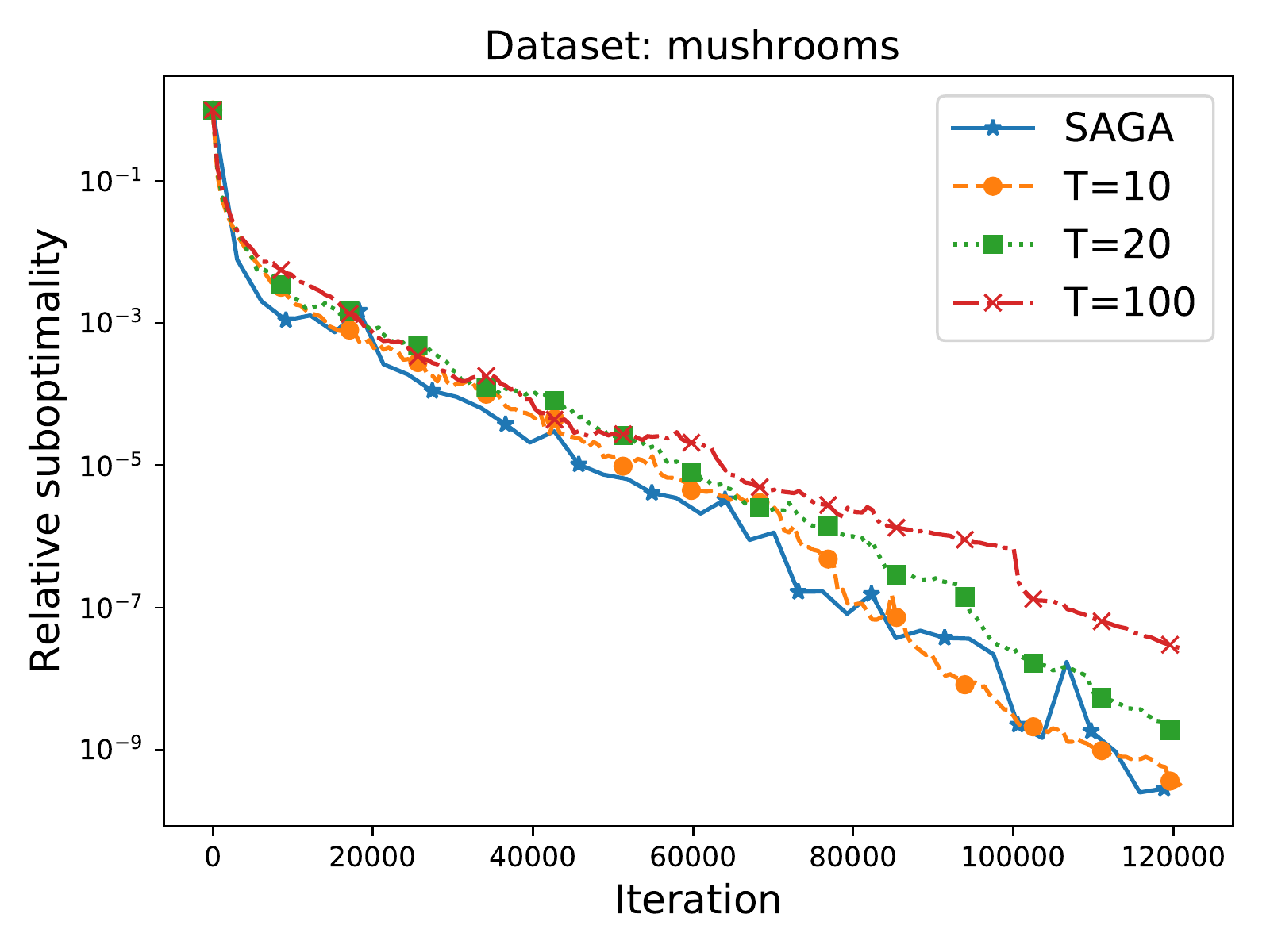}
        %\caption{ Residual vs. iteration  }\label{fig:bl_ex_flops}
\end{minipage}%
\begin{minipage}{0.24\textwidth}
  \centering
\includegraphics[width =  \textwidth ]{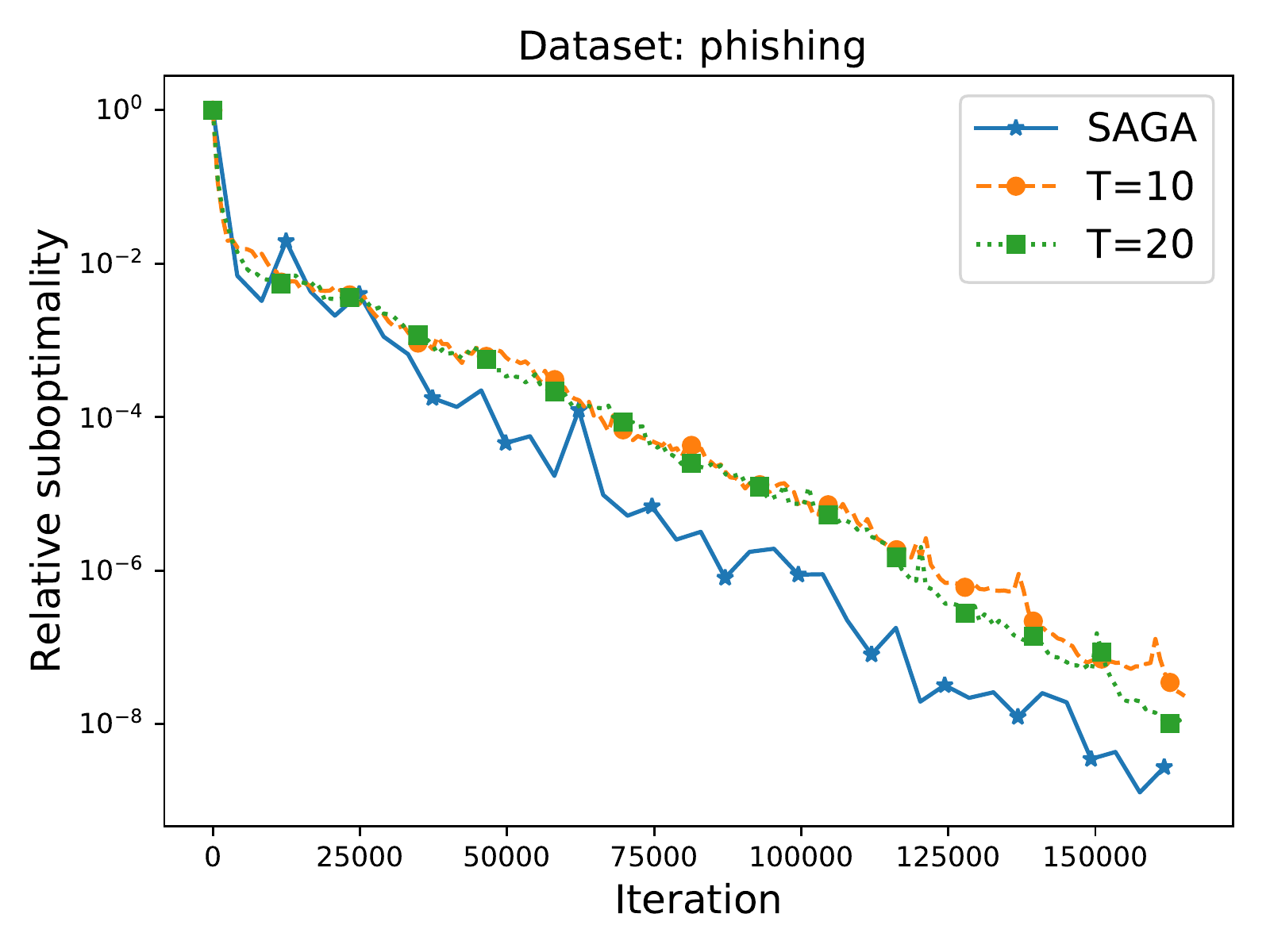}
        %\caption{ Residual vs. iteration  }\label{fig:bl_ex_flops}
\end{minipage}%
\begin{minipage}{0.24\textwidth}
  \centering
\includegraphics[width =  \textwidth ]{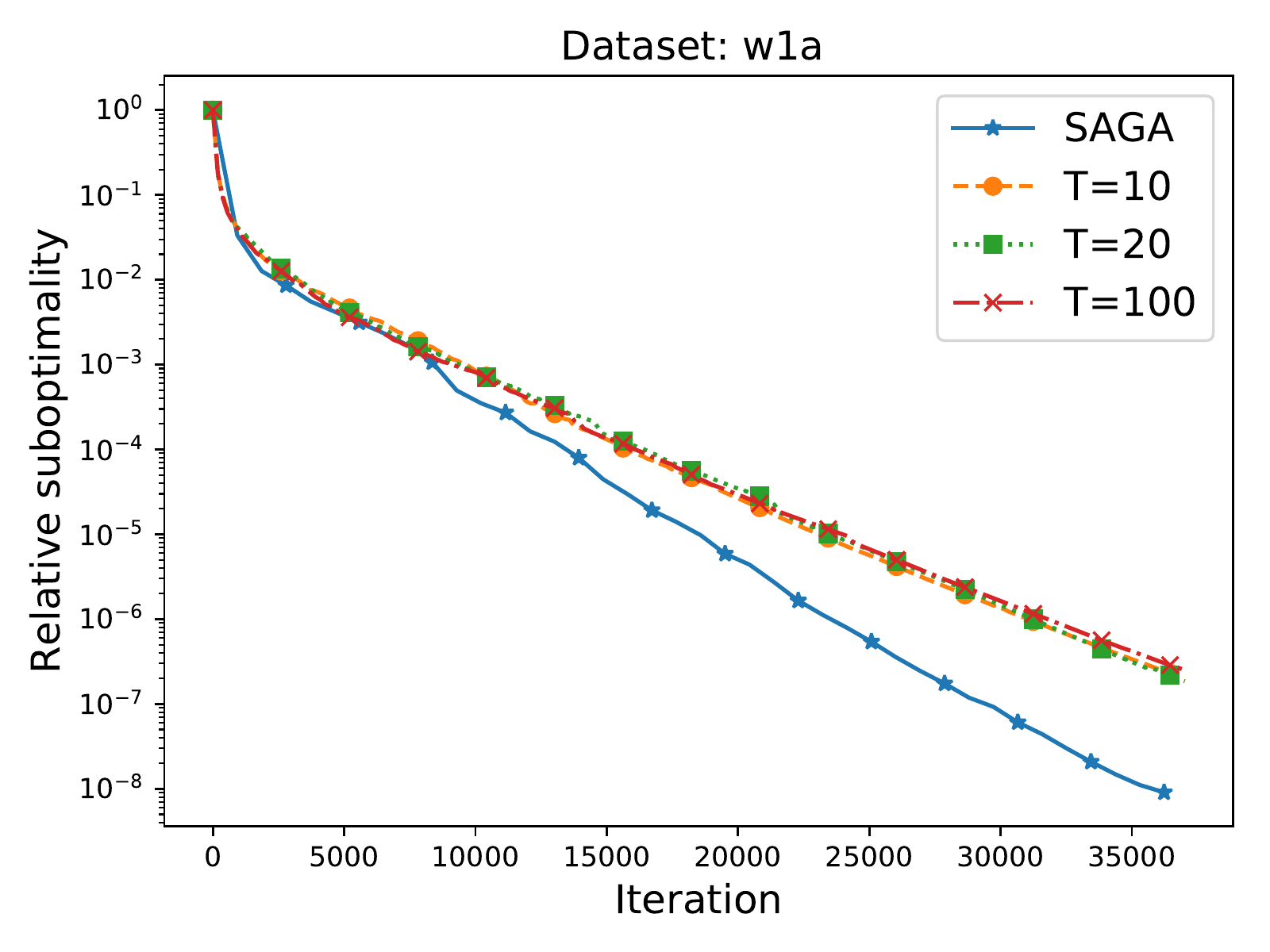}
        %\caption{ Residual vs. iteration  }\label{fig:bl_ex_flops}
\end{minipage}%
\\
\begin{minipage}{0.24\textwidth}
  \centering
\includegraphics[width =  \textwidth ]{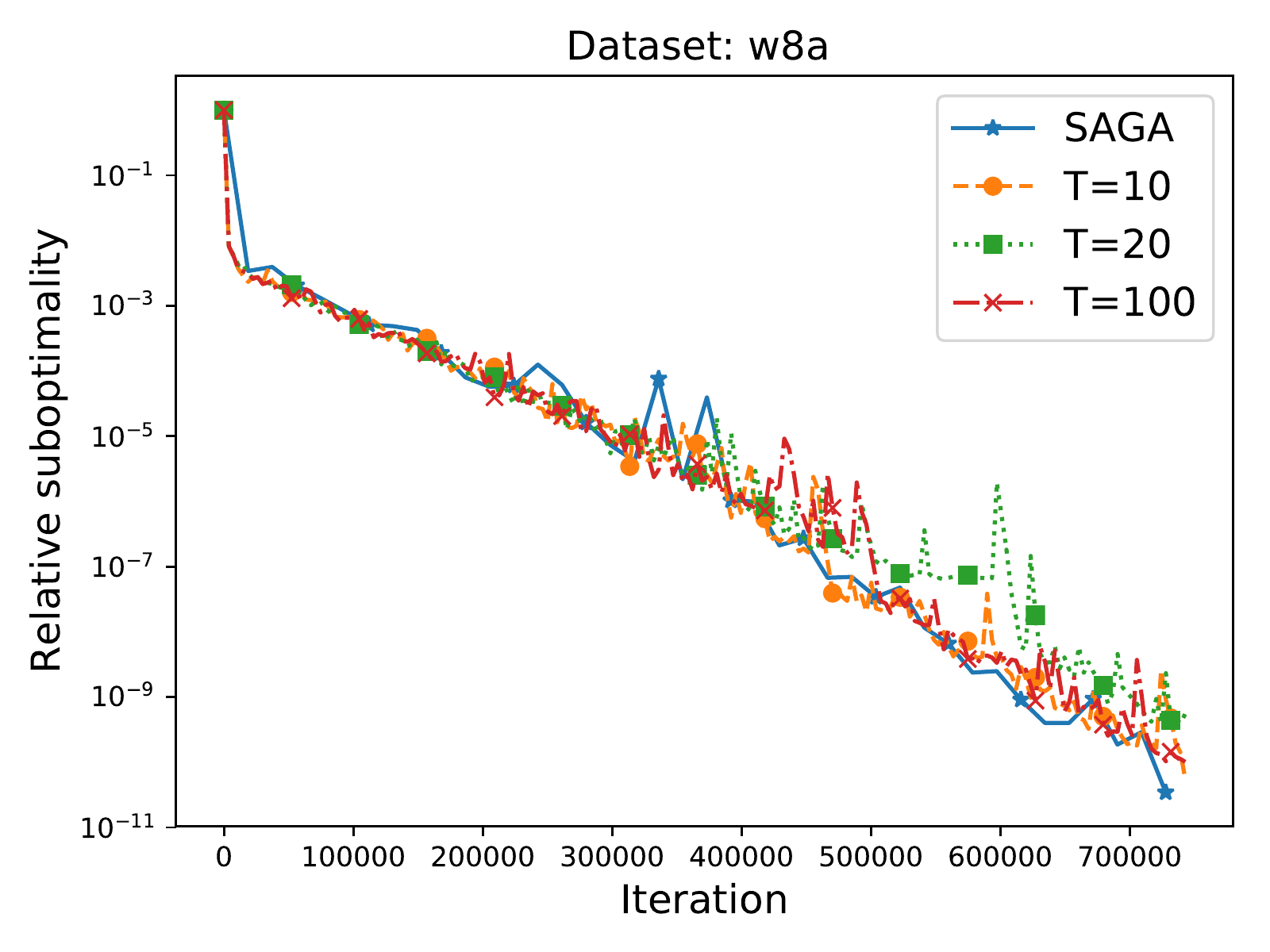}
        %\caption{ Residual vs. iteration  }\label{fig:bl_ex_flops}
\end{minipage}%
\begin{minipage}{0.24\textwidth}
  \centering
\includegraphics[width =  \textwidth ]{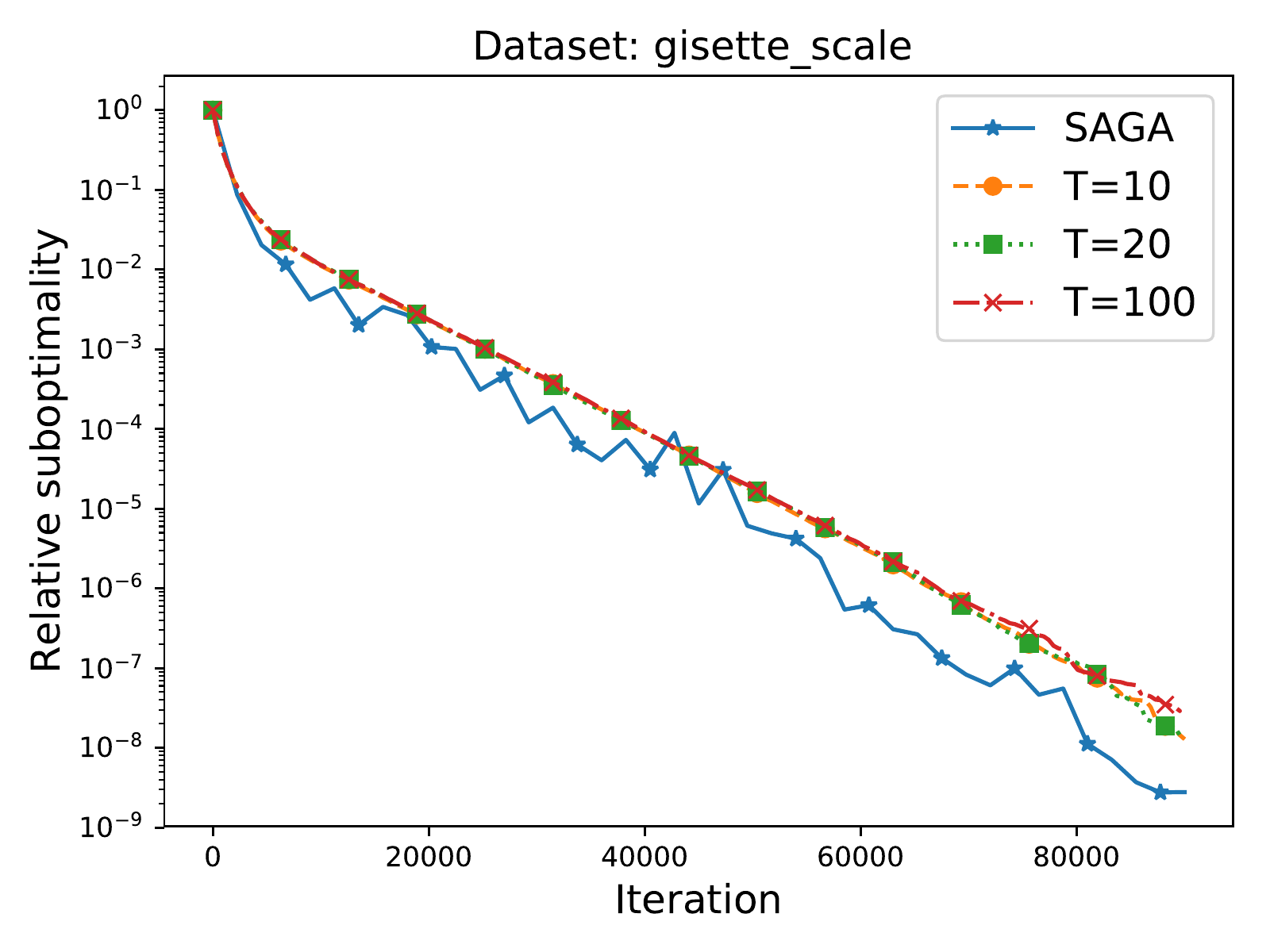}
        %\caption{ Residual vs. iteration  }\label{fig:bl_ex_flops}
\end{minipage}%
\begin{minipage}{0.24\textwidth}
  \centering
\includegraphics[width =  \textwidth ]{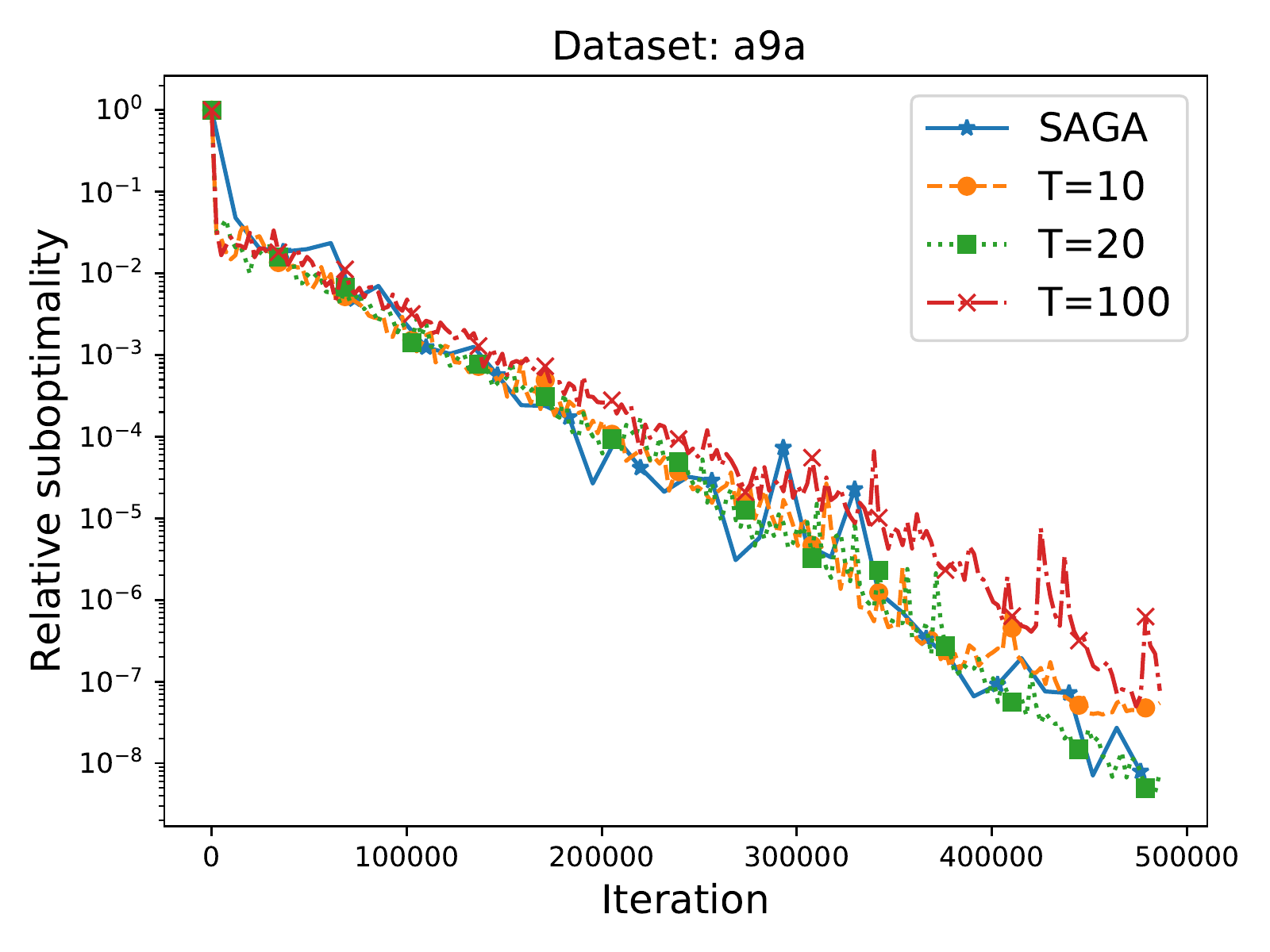}
        %\caption{ Residual vs. iteration  }\label{fig:bl_ex_flops}
\end{minipage}%
\begin{minipage}{0.24\textwidth}
  \centering
\includegraphics[width =  \textwidth ]{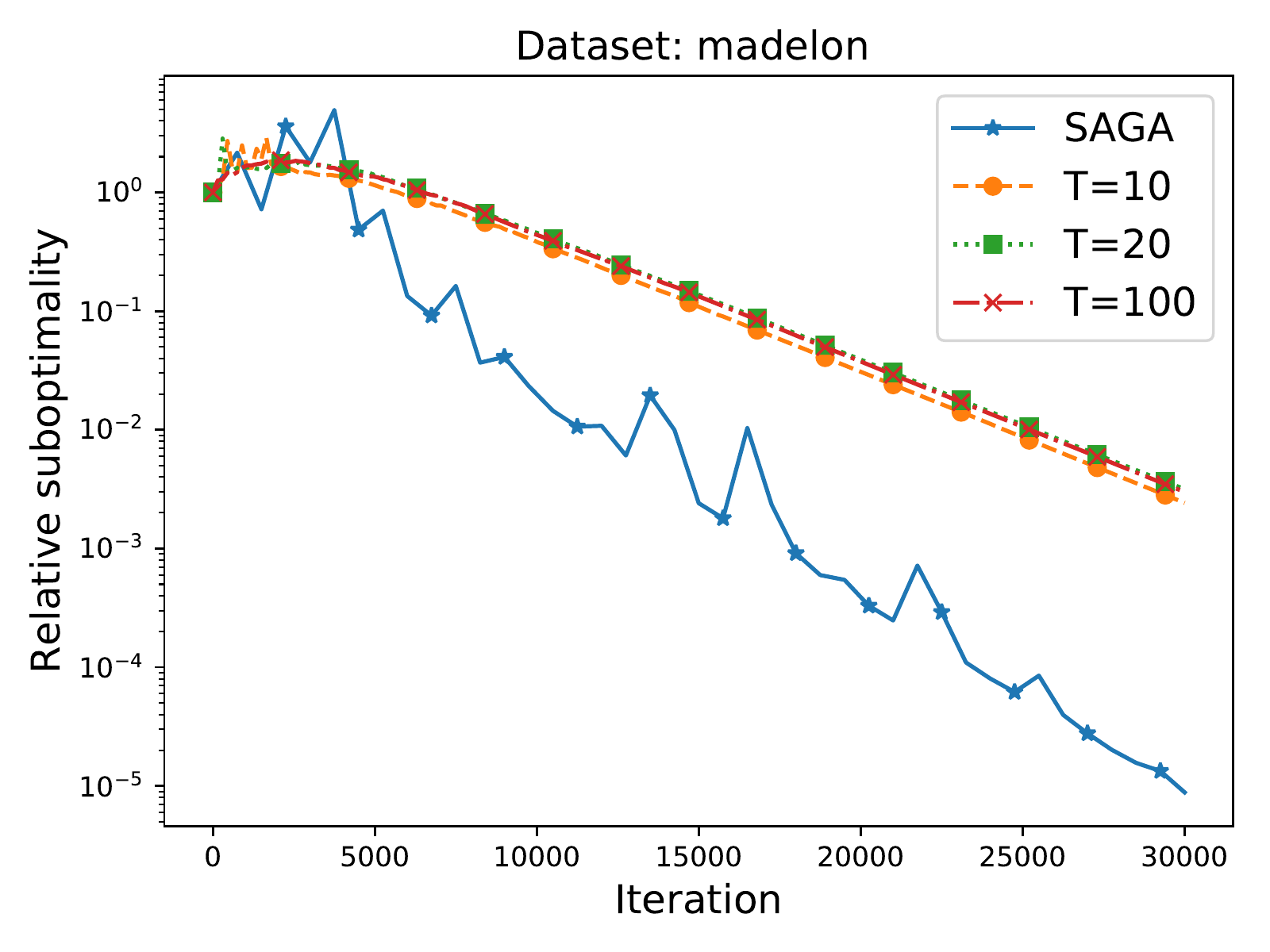}
        %\caption{ Residual vs. iteration  }\label{fig:bl_ex_flops}
\end{minipage}%
\caption{{\tt ISAEGA} applied on LIBSVM~\cite{chang2011libsvm} datasts with $\lambda = 4\cdot 10^{-5}$. Axis $y$ stands for relative suboptimality, i.e. $\frac{f(x^k)-f(x^*)}{f(x^k)-f(x^0)}$.}
\label{fig:ISAEGA}
\end{figure}

\subsection{ {\tt LSVRG} with importance sampling \label{sec:extra_lsvrg}}
As mentioned, one of the contributions of this work is {\tt LSVRG} with arbitrary sampling. In this section, we demonstrate that designing a good sampling can yield a significant speedup in practice. We consider logistic regression problem on LibSVM~\cite{chang2011libsvm} data, as described in Section~\ref{sec:ISAEGA_exp}. However, since LibSVM data are normalized, we pre-multiply each row of the data matrix by a random scaling factor. In particular, the scaling factors are proportional to $l^2$ where $l$ is sampled uniformly from $[1000]$ such that the Frobenius norm of the data matrix is $n$. For the sake of simplicity, consider case $\lambda=0$.

\paragraph{Choice vector $v$.} 
Note that since $\mM_j = \mA_{j:}^\top \mA_{j:}$, the following claim must hold: \emph{ Consider fixed $v$. Then if~\eqref{eq:ESO_saga} holds for any set of vector $\{h_j\}_{j=1}^n$ such that $h_j$ is parallel to $\mA_{j:}$, then ~\eqref{eq:ESO_saga} holds for any set of vector $\{h_j\}_{j=1}^n$.} Thus, we can set $h_j = c_j\mA_{j:}^\top/\| \mA_{j:}\|$ without loss of generality. Thus, $\mM_j^{\frac12} h_j= c_j \mA_{j:}^\top$, and~\eqref{eq:ESO_saga} becomes equivalent to $\PR \circ \left(\mA^\top \mA\right)\preceq \diag(p\circ v)$ where $\PR_{jj'} = \Prob{j\in R, j' \in R}$. Note that this is exactly \emph{expected separable overapproximation (ESO)} for coordinate descent~\cite{ESO}. Thus we choose vector $v$ to be proportional to $\pR$ such that $\PR \circ \left(\mA^\top \mA\right)\preceq \diag(p\circ v)$ holds (as proposed in~\cite{AccMbCd}). In order to compute the scaling constant, one needs to evaluate maximum eigenvalue of PSD $n\times n$ matrix, which is of $\cO(n^2)$ cost. We do so in the experiments. Note that there is a suboptimal, but cheeaper way to obtain $v$ described in~\cite{qian2019saga}. Lastly, if $\lambda>0$, we set $v$ such that $\PR \circ \left(\mA^\top \mA +\lambda \mI \right)\preceq \diag(p\circ v)$.

\paragraph{Choice of probabilities.}
In order to be fair, we only compare methods where $\E{|R|}=\tau$. For the case $\tau =1$, we consider a sampling such that $|R|=1$ according to a given probability vector $\pR$. For uniform sampling, we have $\pR=n^{-1}\eR$, while for importance sampling, we set $\pR_j= \frac{\lambda_{\max}(\mM_j)}{\sum_{j'=1}^n \lambda_{\max}(\mM_{j'})}$. In the case $\tau >1$, we consider independent sampling from~\cite{AccMbCd}. In particular, $\Prob{j\in R} = \pRj$ with $\sum \pRj = \tau$ and binary random variables $(j\in R)$ are jointly independent. For uniform sampling we have $\pR = \tau n^{-1} \eR$. For importance sampling, probability vector $\pR$ is chosen such that $p_j =  \frac{\lambda_{\max}(\mM_j)}{\varrho+ \lambda_{\max}(\mM_{j}}$, where $\varrho$ is such that $\sum \pRj = \tau$. The mentioned sampling was proven to be superior over uniform minibatching in~\cite{AccMbCd}. Next, stepsize $\gamma = \frac{1}{6}\min_j\frac{n\pRj}{v_j}$ was chosen for all methods.

Lastly, $\probx = \frac{1}{2n}$ was chosen for {\tt LSVRG}. The results are presented in Figures~\ref{fig:LSVRG1} and~\ref{fig:LSVRG2} (a subset of the results was already presented in Figure~\ref{fig:LSVRG_main}).

In all cases, {\tt LSVRG} with importance sampling was the fastest method. As provided theory suggests, it outperformed methods with importance sampling especially significantly for small $\tau$; and the larger $\tau$, the smaller the effect of importance sampling is. However, our experiments indicate the superiority of {\tt LSVRG} to  {\tt SAGA} in the importance sampling setup. In particular, stepsize $\gamma = \frac{1}{6}\min_j\frac{n\pRj}{v_j}$ is often too large for {\tt SAGA}. Note that both optimal stepsize and optimal probabilities require the prior knowledge of the quasi strong convexity constant $\sigma$\footnote{Or more generally, strong growth constant, see Appendix~\ref{sec:sg}} which is, in our case unknown (see the importance serial sampling proposed in~\cite{gower2018stochastic}, and {\tt SAGA} is more sensitive to that choice. One can still estimate it as $\lambda$, however, this would yield suboptimal performance as well.

\begin{figure}[!h]
\centering
\begin{minipage}{0.3\textwidth}
  \centering
\includegraphics[width =  \textwidth ]{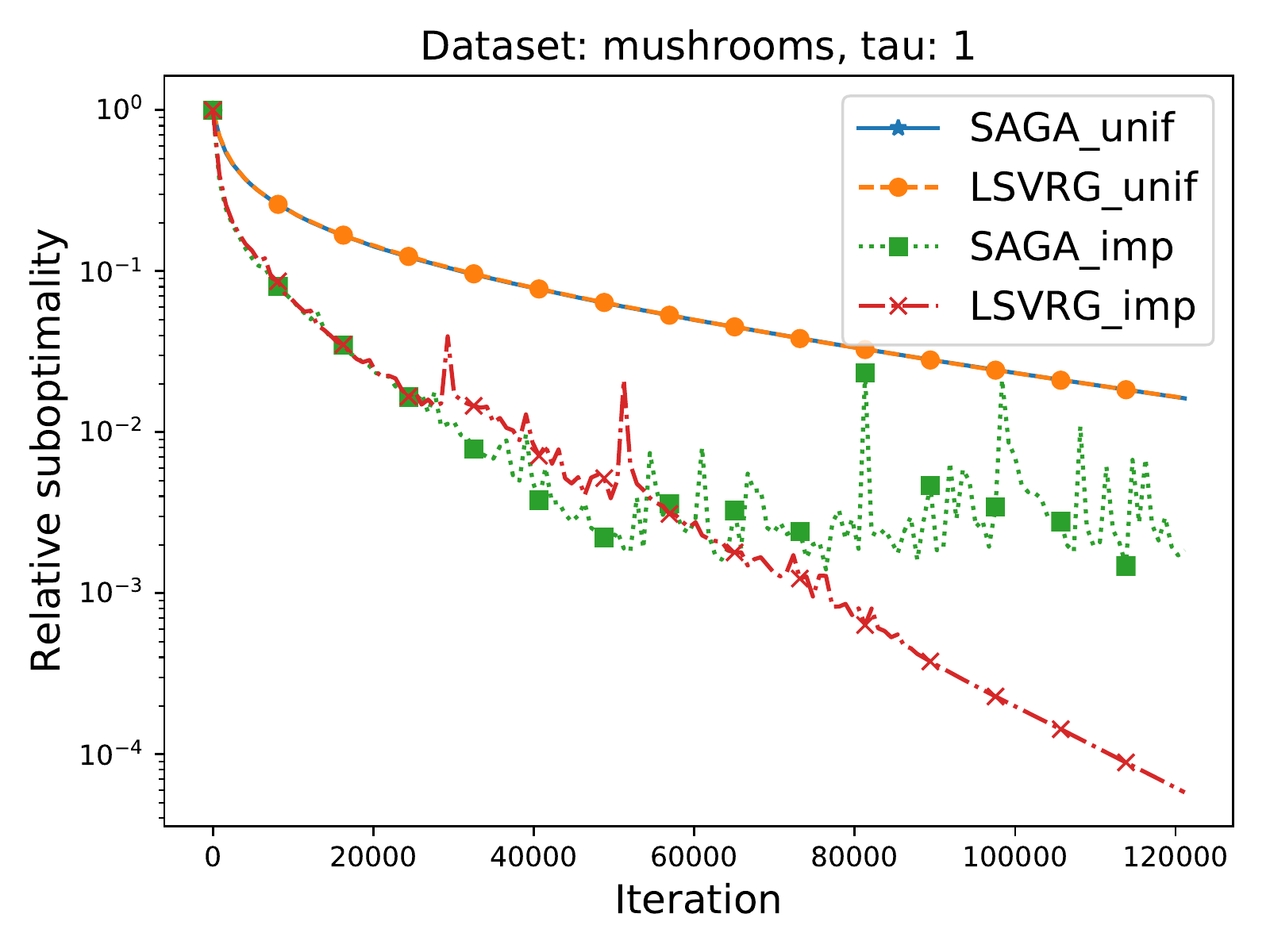}
        %\caption{ Residual vs. iteration  }\label{fig:bl_ex_flops}
\end{minipage}%
\begin{minipage}{0.3\textwidth}
  \centering
\includegraphics[width =  \textwidth ]{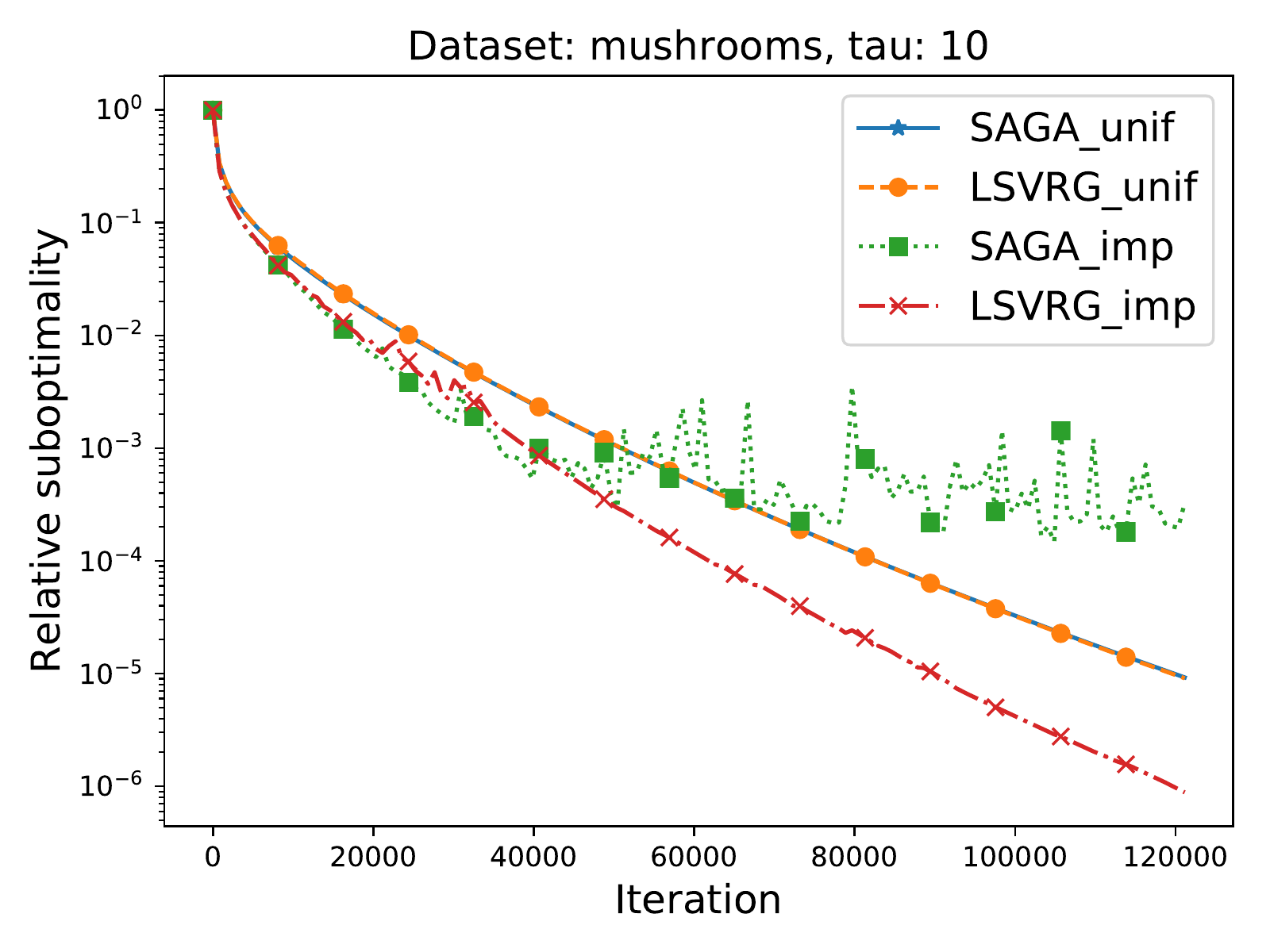}
        %\caption{ Residual vs. iteration  }\label{fig:bl_ex_flops}
\end{minipage}%
\begin{minipage}{0.3\textwidth}
  \centering
\includegraphics[width =  \textwidth ]{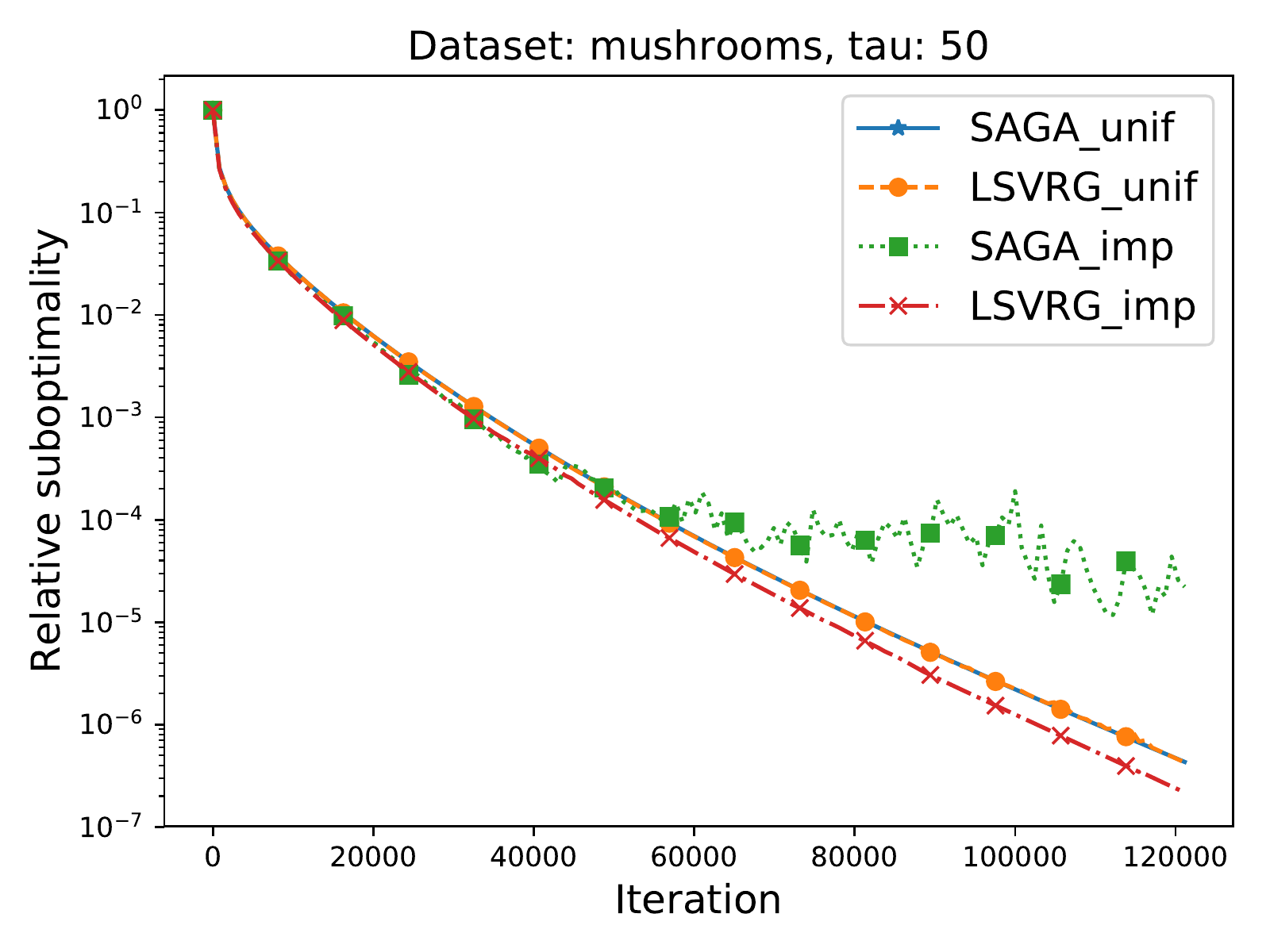}
        %\caption{ Residual vs. iteration  }\label{fig:bl_ex_flops}
\end{minipage}%
\\
\begin{minipage}{0.3\textwidth}
  \centering
\includegraphics[width =  \textwidth ]{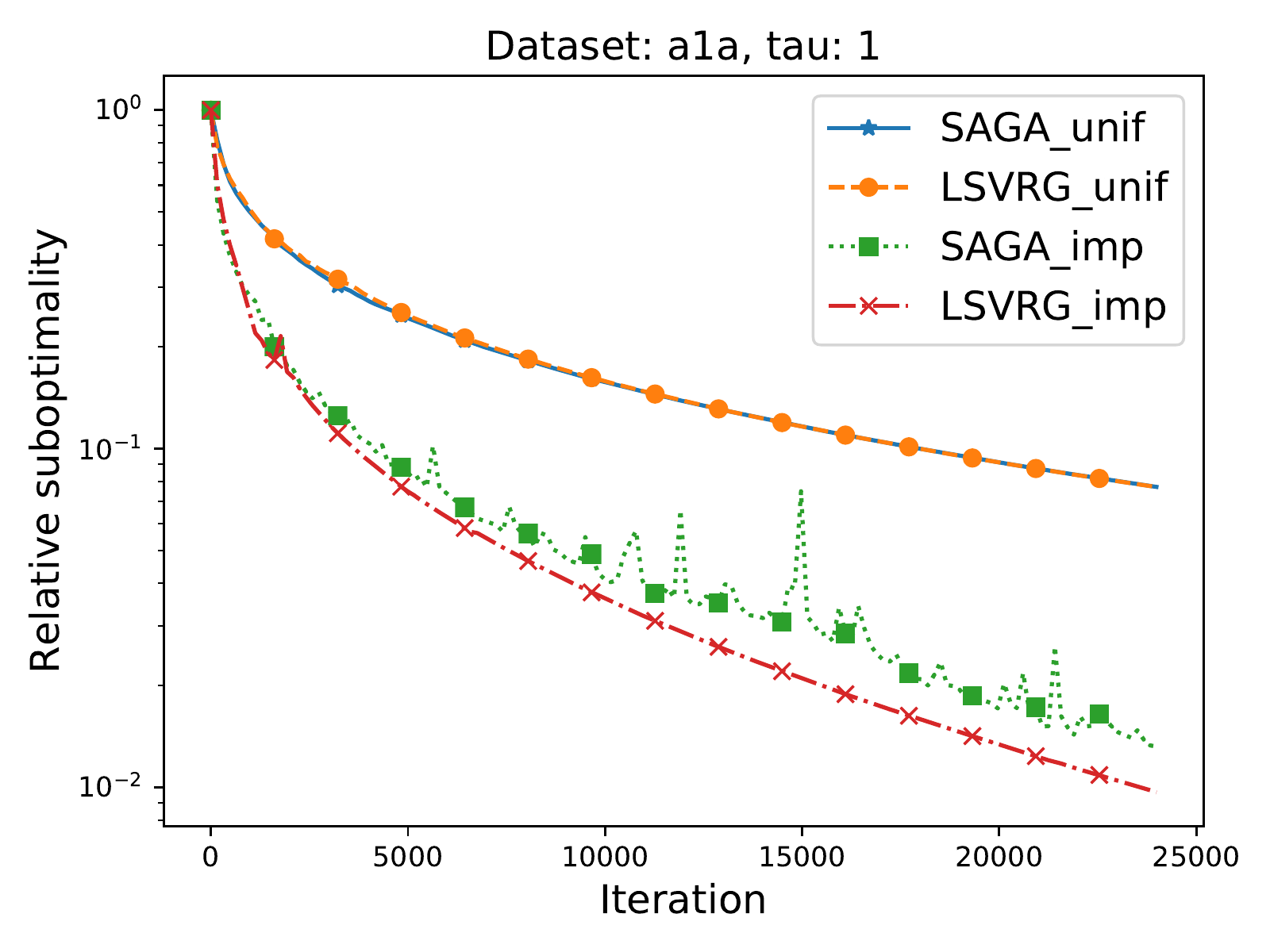}
        %\caption{ Residual vs. iteration  }\label{fig:bl_ex_flops}
\end{minipage}%
\begin{minipage}{0.3\textwidth}
  \centering
\includegraphics[width =  \textwidth ]{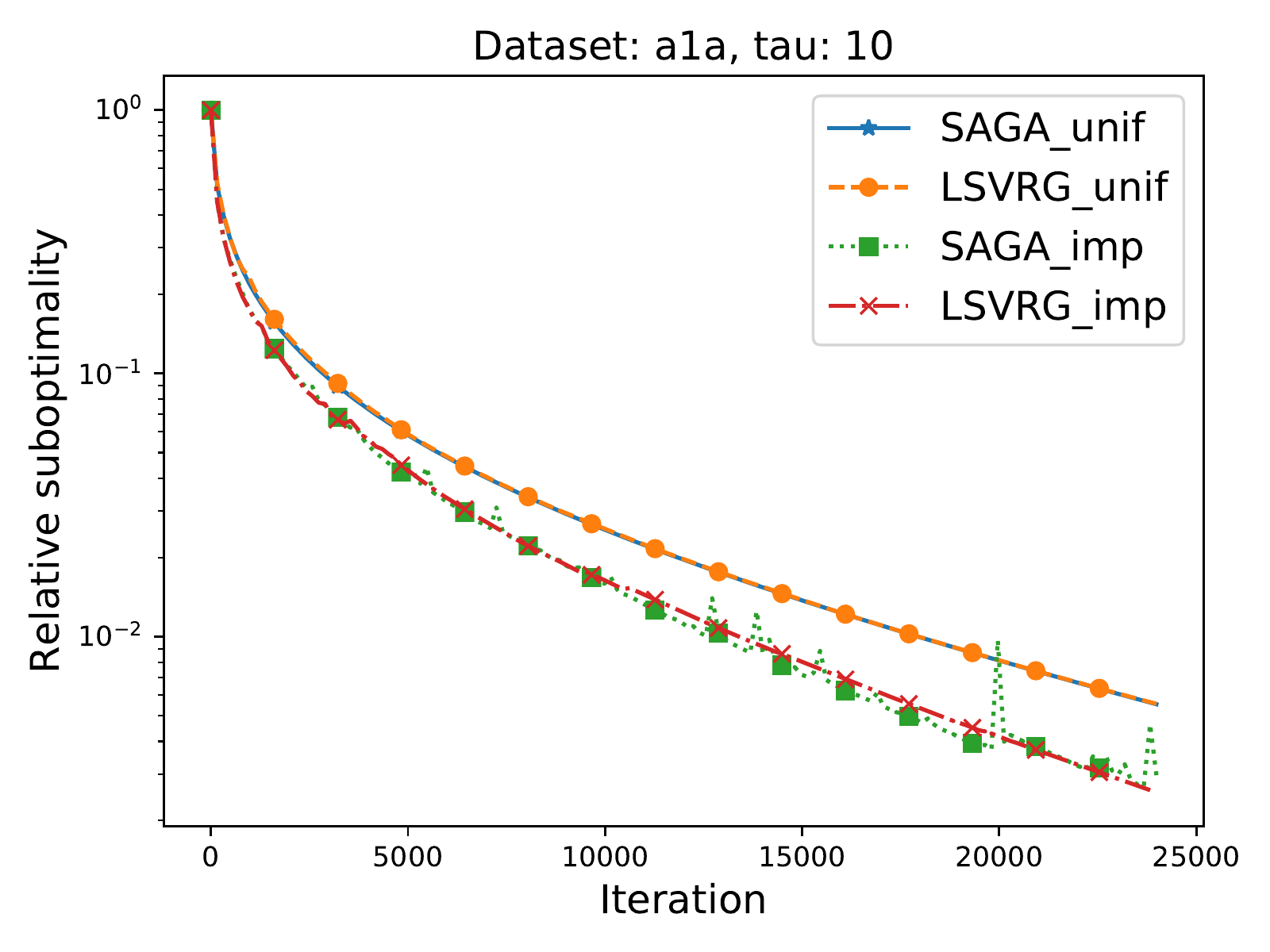}
        %\caption{ Residual vs. iteration  }\label{fig:bl_ex_flops}
\end{minipage}%
\begin{minipage}{0.3\textwidth}
  \centering
\includegraphics[width =  \textwidth ]{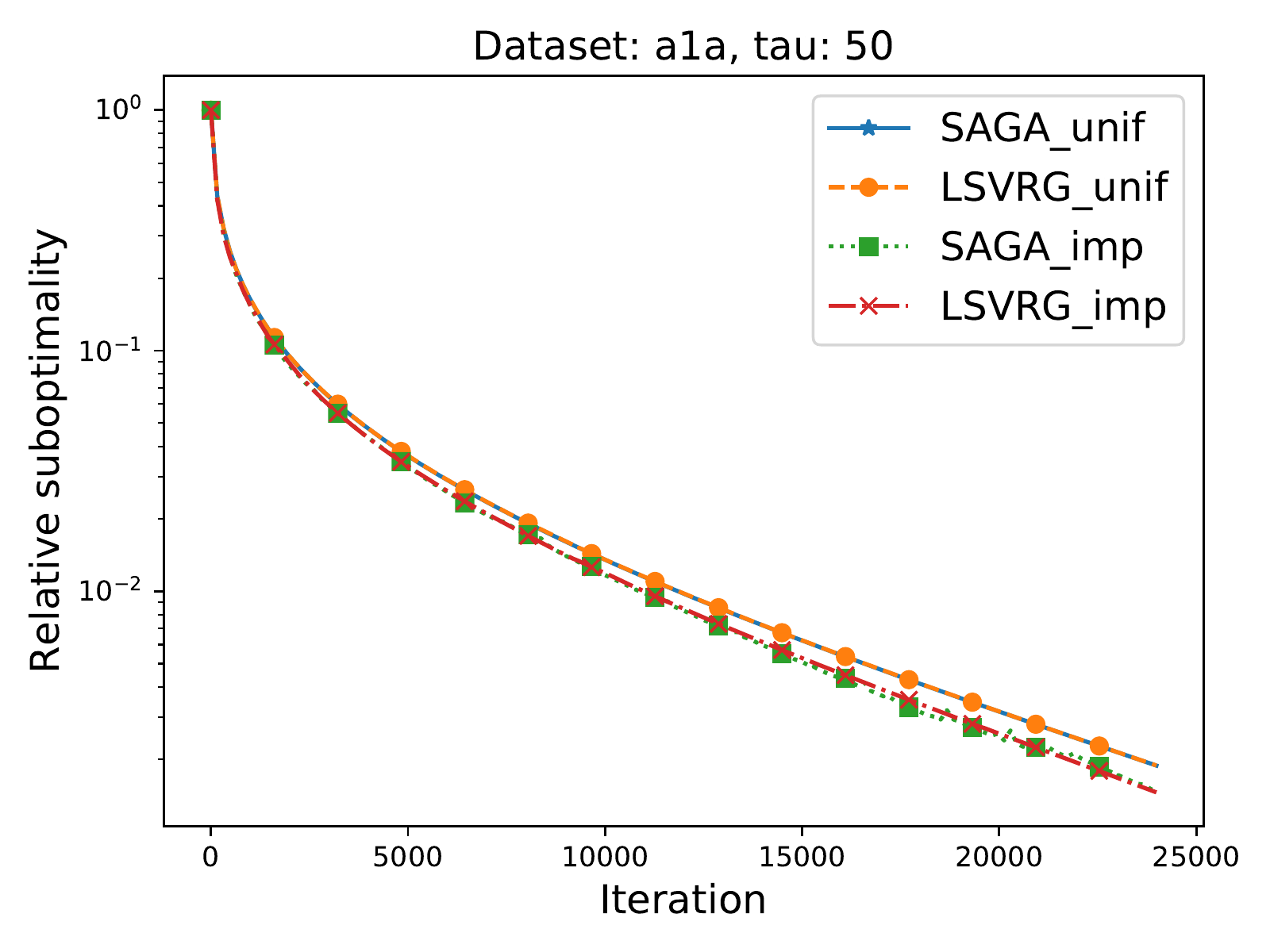}
        %\caption{ Residual vs. iteration  }\label{fig:bl_ex_flops}
\end{minipage}%
\\
\begin{minipage}{0.3\textwidth}
  \centering
\includegraphics[width =  \textwidth ]{LSVRGDatasetw1atau1meth3.pdf}
        %\caption{ Residual vs. iteration  }\label{fig:bl_ex_flops}
\end{minipage}%
\begin{minipage}{0.3\textwidth}
  \centering
\includegraphics[width =  \textwidth ]{LSVRGDatasetw1atau10meth3.pdf}
        %\caption{ Residual vs. iteration  }\label{fig:bl_ex_flops}
\end{minipage}%
\begin{minipage}{0.3\textwidth}
  \centering
\includegraphics[width =  \textwidth ]{LSVRGDatasetw1atau50meth3.pdf}
        %\caption{ Residual vs. iteration  }\label{fig:bl_ex_flops}
\end{minipage}%
\\
\begin{minipage}{0.3\textwidth}
  \centering
\includegraphics[width =  \textwidth ]{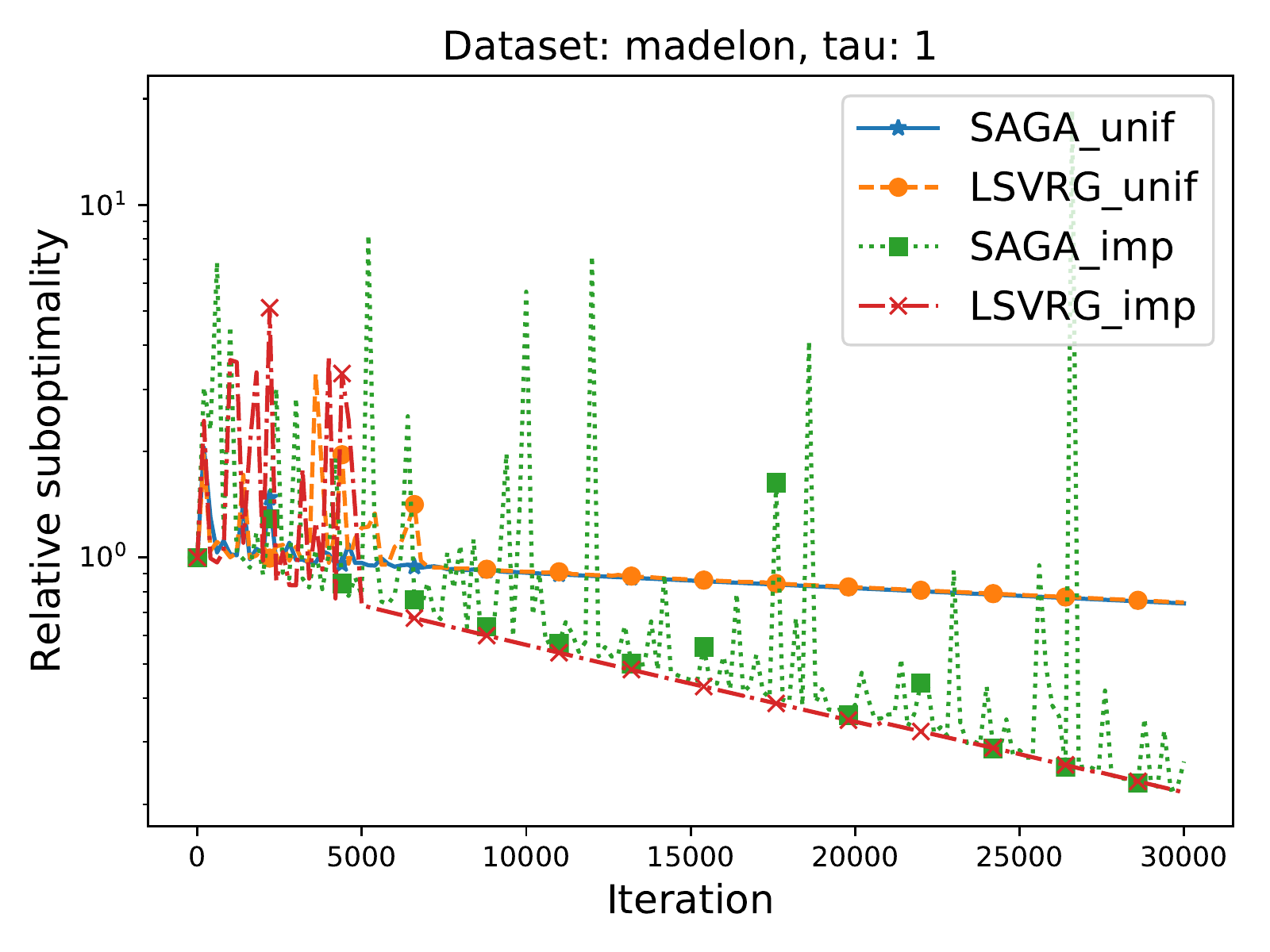}
        %\caption{ Residual vs. iteration  }\label{fig:bl_ex_flops}
\end{minipage}%
\begin{minipage}{0.3\textwidth}
  \centering
\includegraphics[width =  \textwidth ]{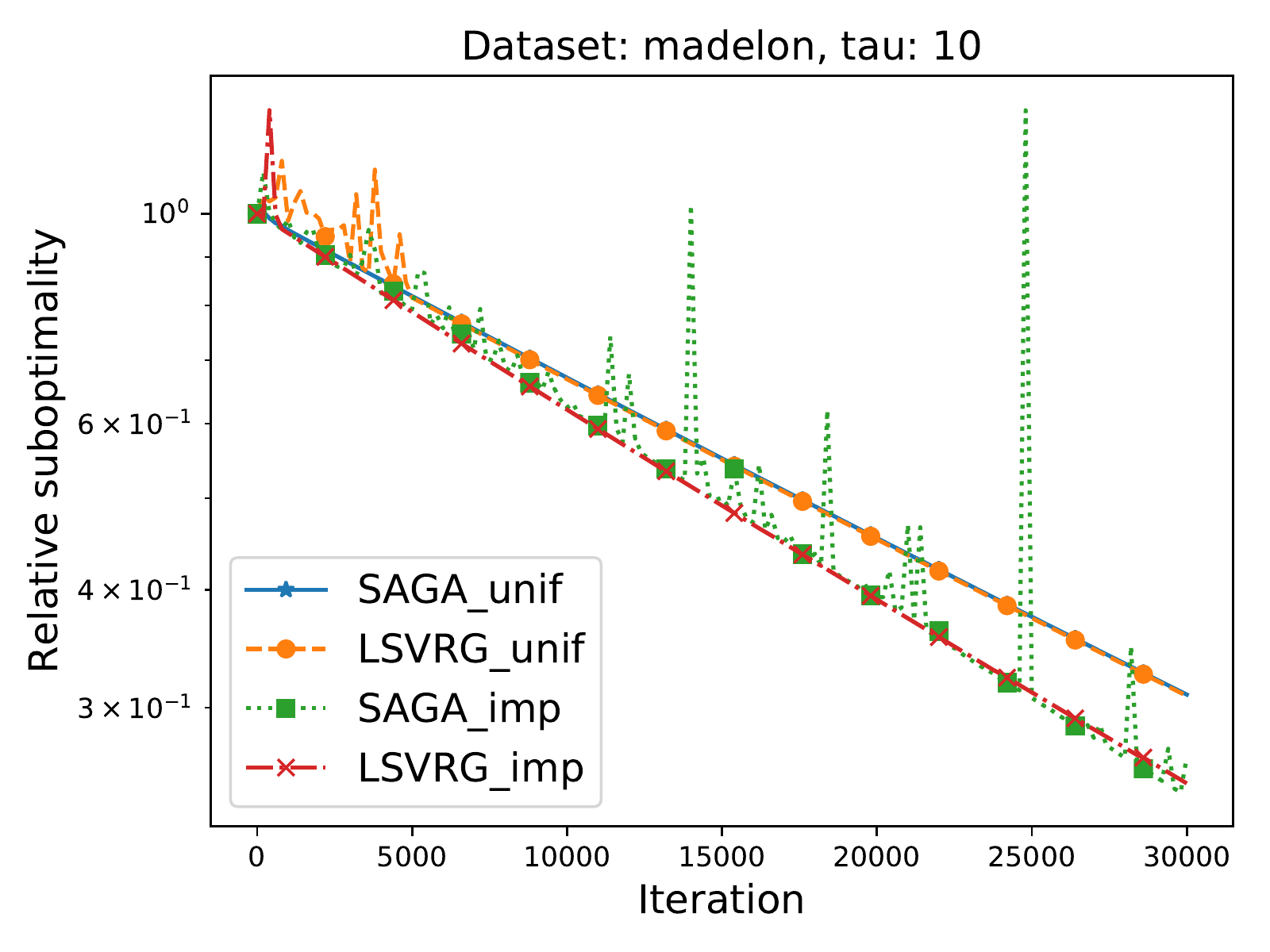}
        %\caption{ Residual vs. iteration  }\label{fig:bl_ex_flops}
\end{minipage}%
\begin{minipage}{0.3\textwidth}
  \centering
\includegraphics[width =  \textwidth ]{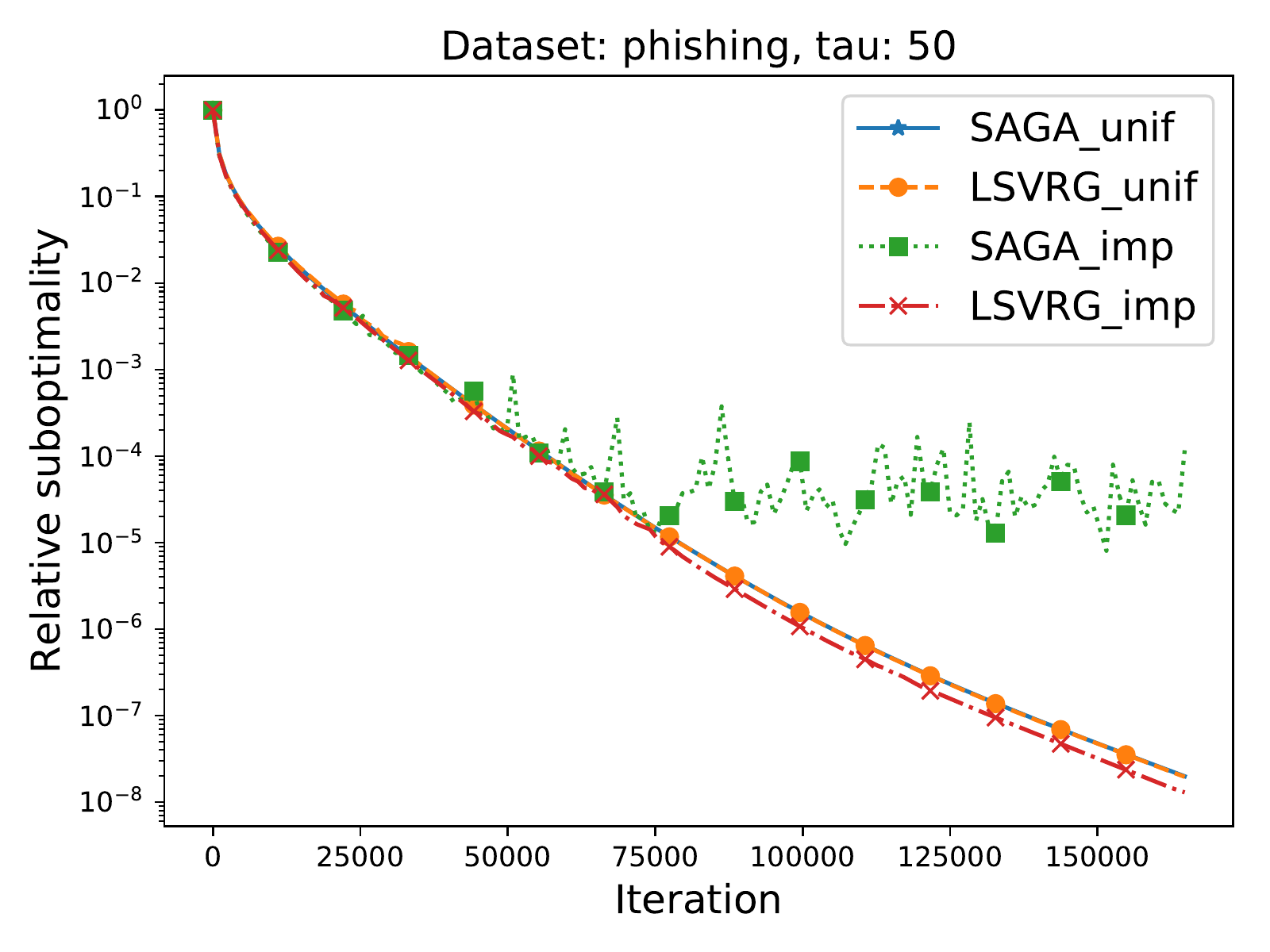}
        %\caption{ Residual vs. iteration  }\label{fig:bl_ex_flops}
\end{minipage}%
\\
\begin{minipage}{0.3\textwidth}
  \centering
\includegraphics[width =  \textwidth ]{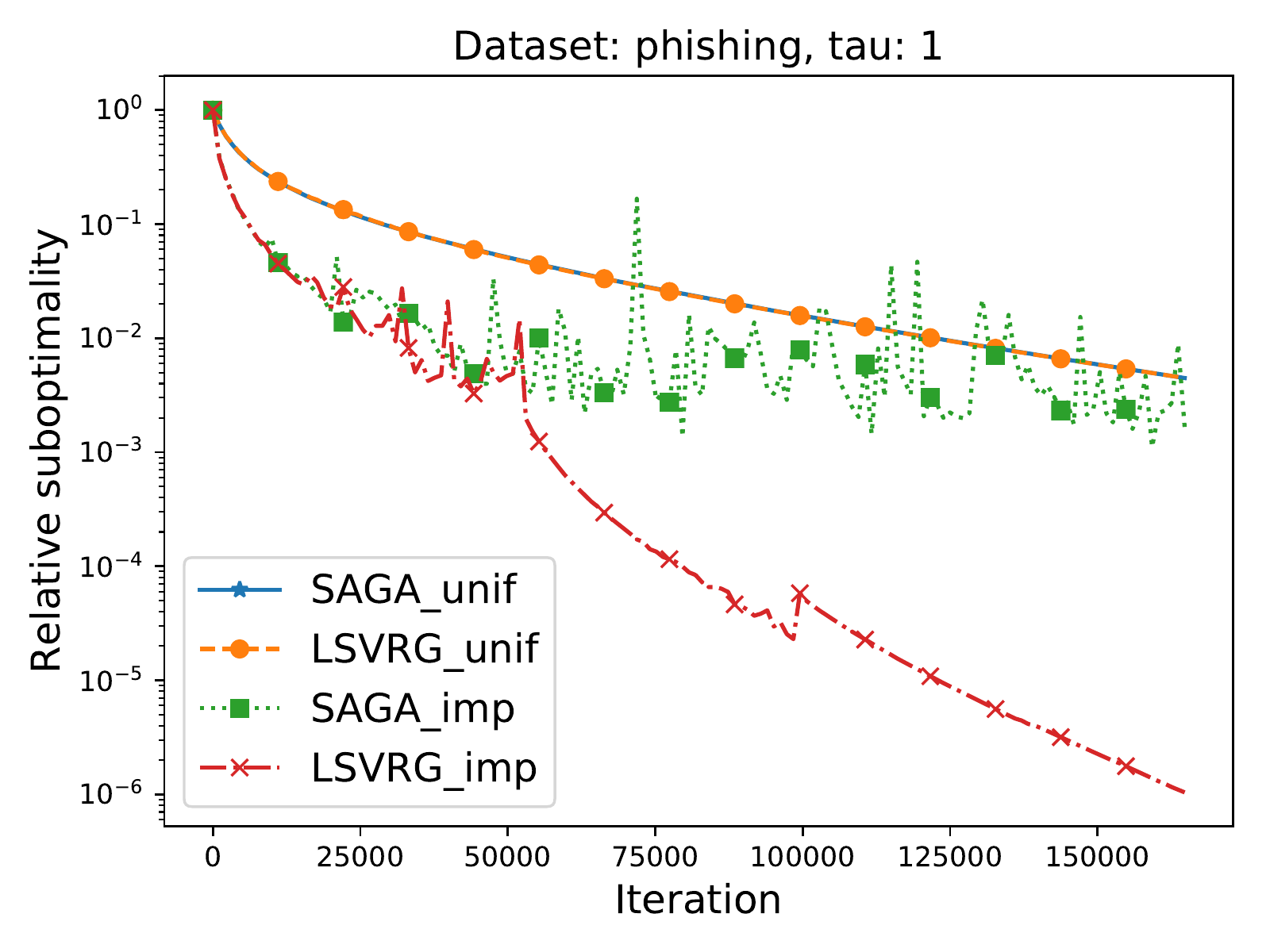}
        %\caption{ Residual vs. iteration  }\label{fig:bl_ex_flops}
\end{minipage}%
\begin{minipage}{0.3\textwidth}
  \centering
\includegraphics[width =  \textwidth ]{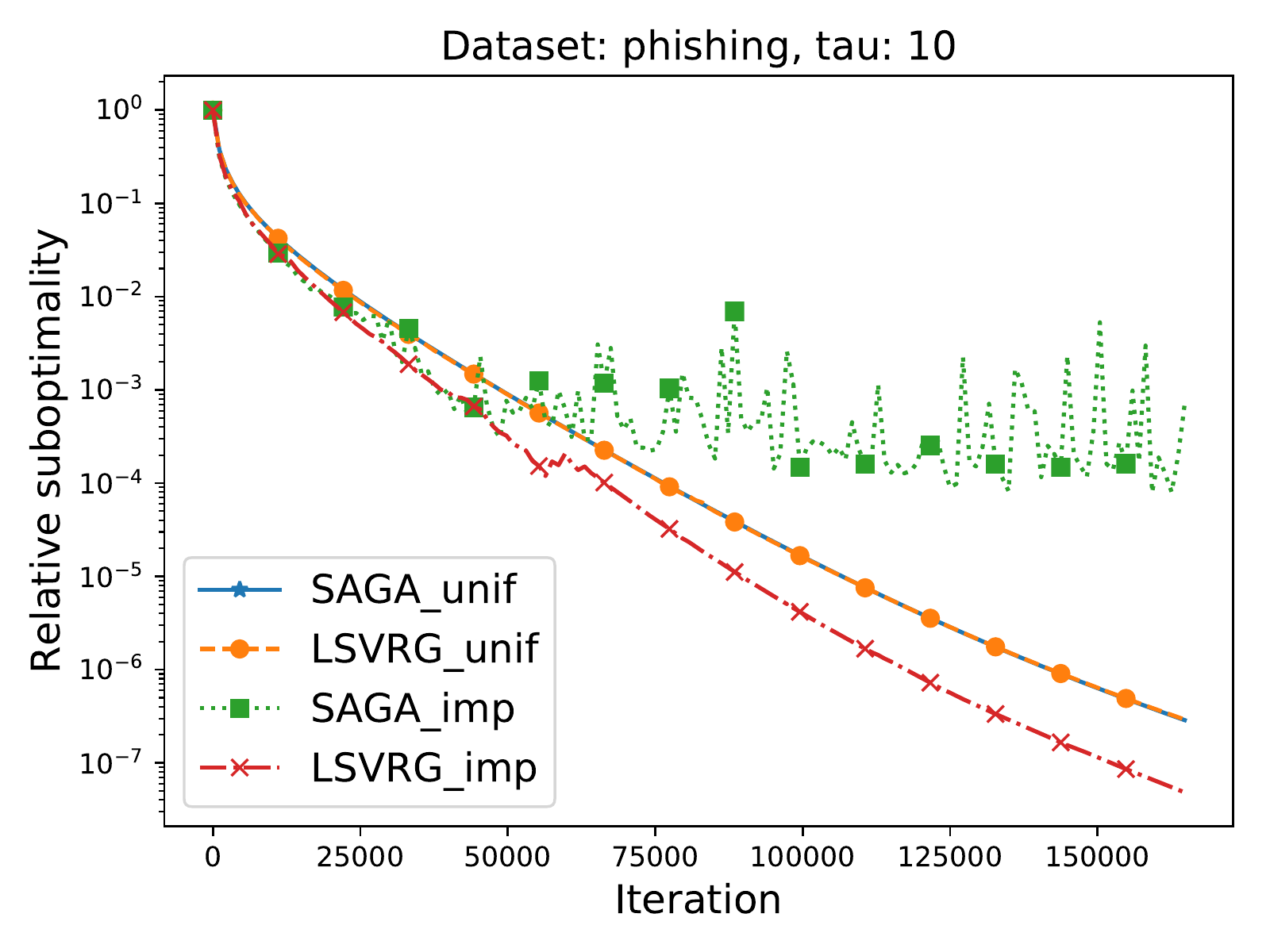}
        %\caption{ Residual vs. iteration  }\label{fig:bl_ex_flops}
\end{minipage}%
\begin{minipage}{0.3\textwidth}
  \centering
\includegraphics[width =  \textwidth ]{LSVRGDatasetphishingtau50meth3.pdf}
        %\caption{ Residual vs. iteration  }\label{fig:bl_ex_flops}
\end{minipage}%
\\
\begin{minipage}{0.3\textwidth}
  \centering
\includegraphics[width =  \textwidth ]{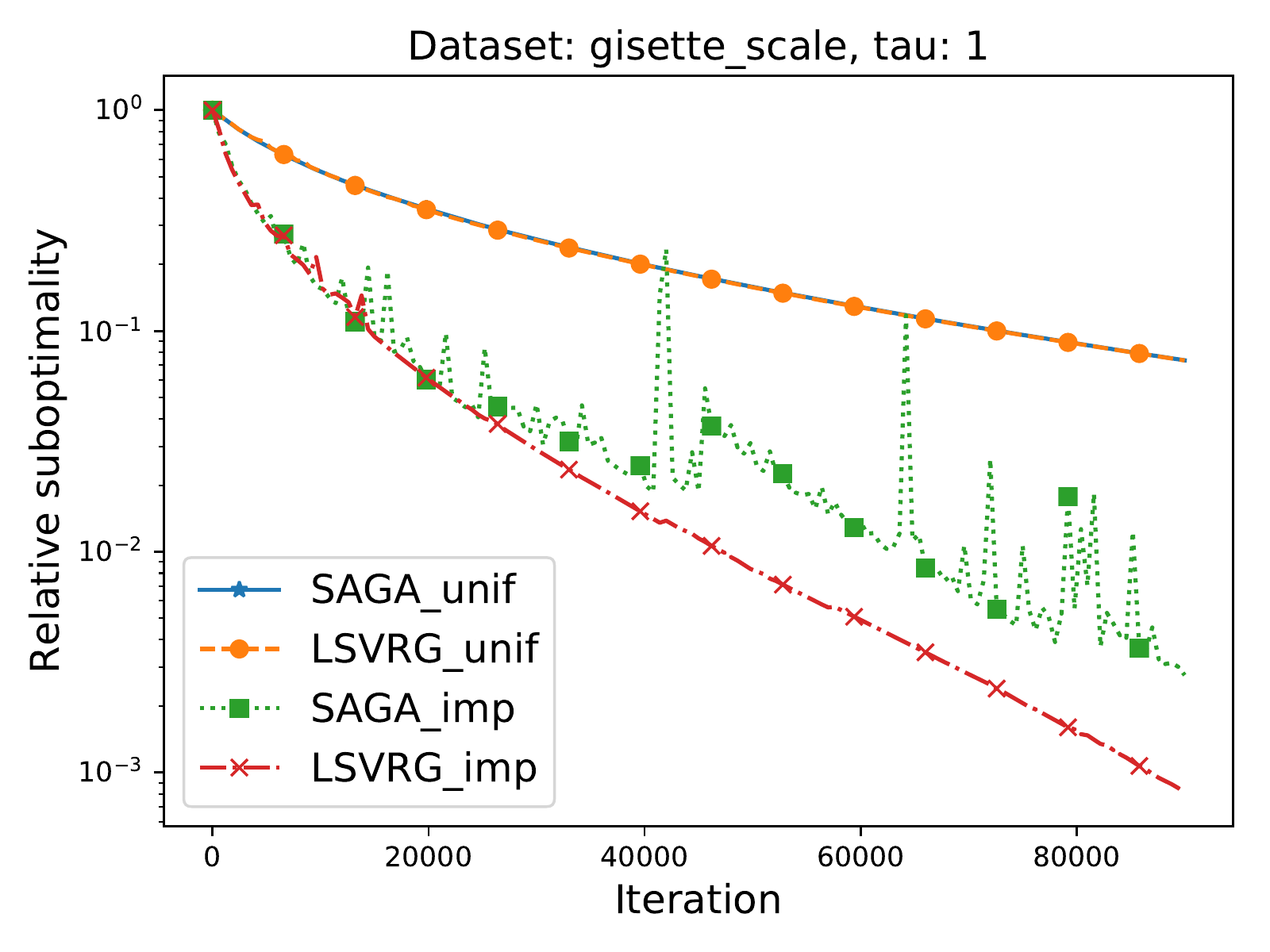}
        %\caption{ Residual vs. iteration  }\label{fig:bl_ex_flops}
\end{minipage}%
\begin{minipage}{0.3\textwidth}
  \centering
\includegraphics[width =  \textwidth ]{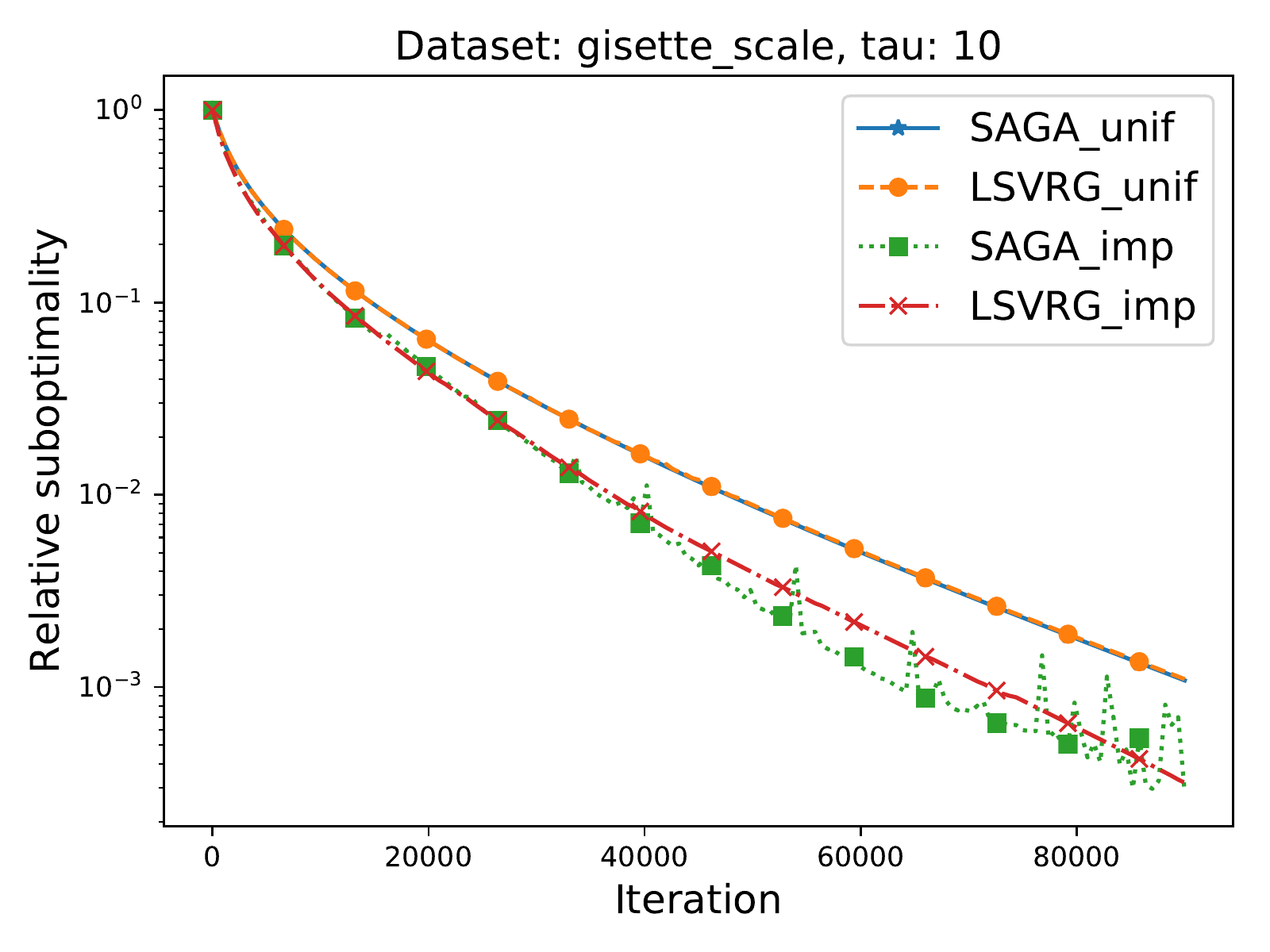}
        %\caption{ Residual vs. iteration  }\label{fig:bl_ex_flops}
\end{minipage}%
\begin{minipage}{0.3\textwidth}
  \centering
\includegraphics[width =  \textwidth ]{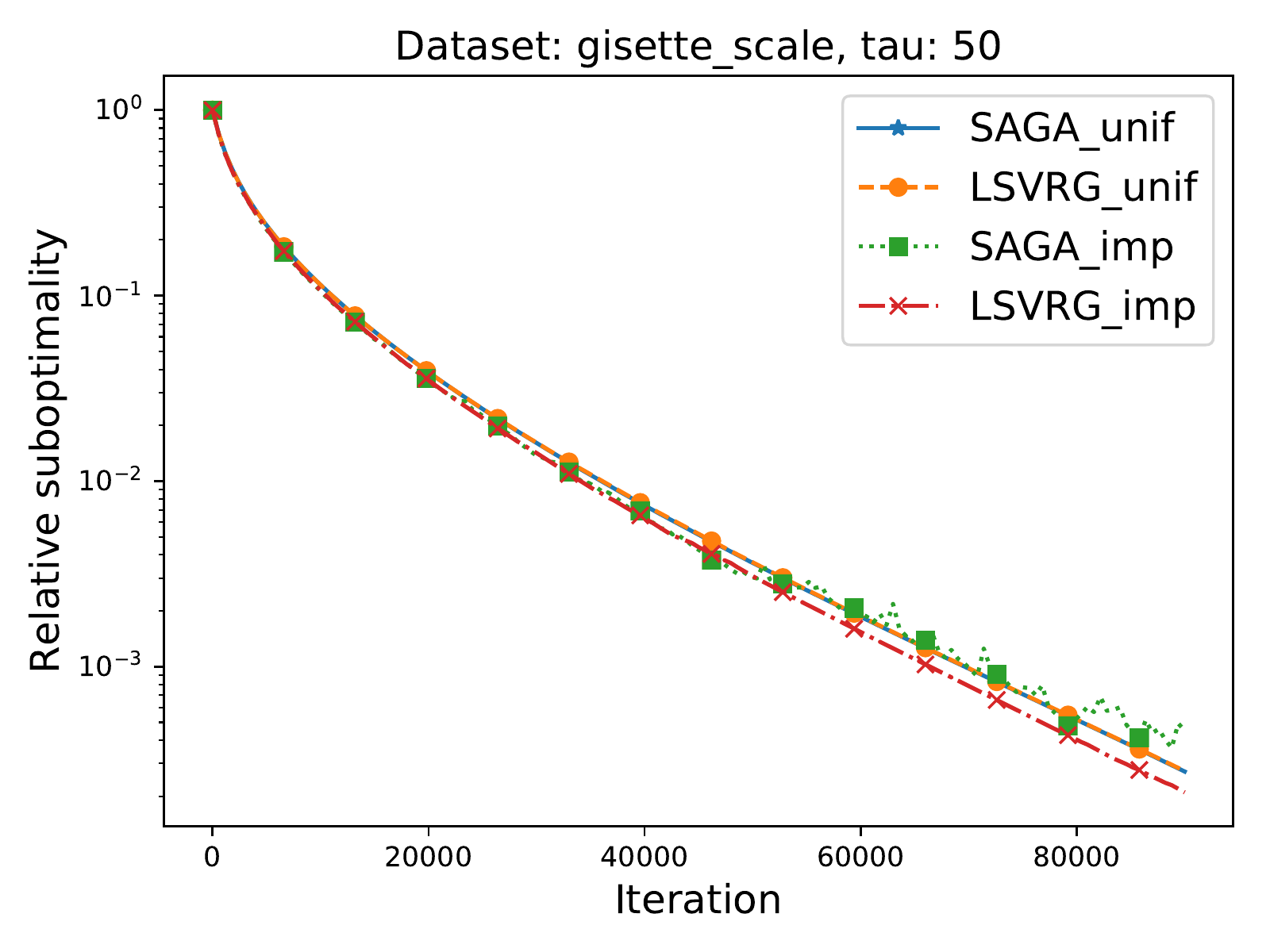}
        %\caption{ Residual vs. iteration  }\label{fig:bl_ex_flops}
\end{minipage}%s
\caption{{\tt LSVRG} applied on LIBSVM~\cite{chang2011libsvm} datasets with $\lambda = 10^{-5}$. Axis $y$ stands for relative suboptimality, i.e. $\frac{f(x^k)-f(x^*)}{f(x^k)-f(x^0)}$.}
\label{fig:LSVRG1}
\end{figure}

\begin{figure}[!h]
\centering
\begin{minipage}{0.3\textwidth}
  \centering
\includegraphics[width =  \textwidth ]{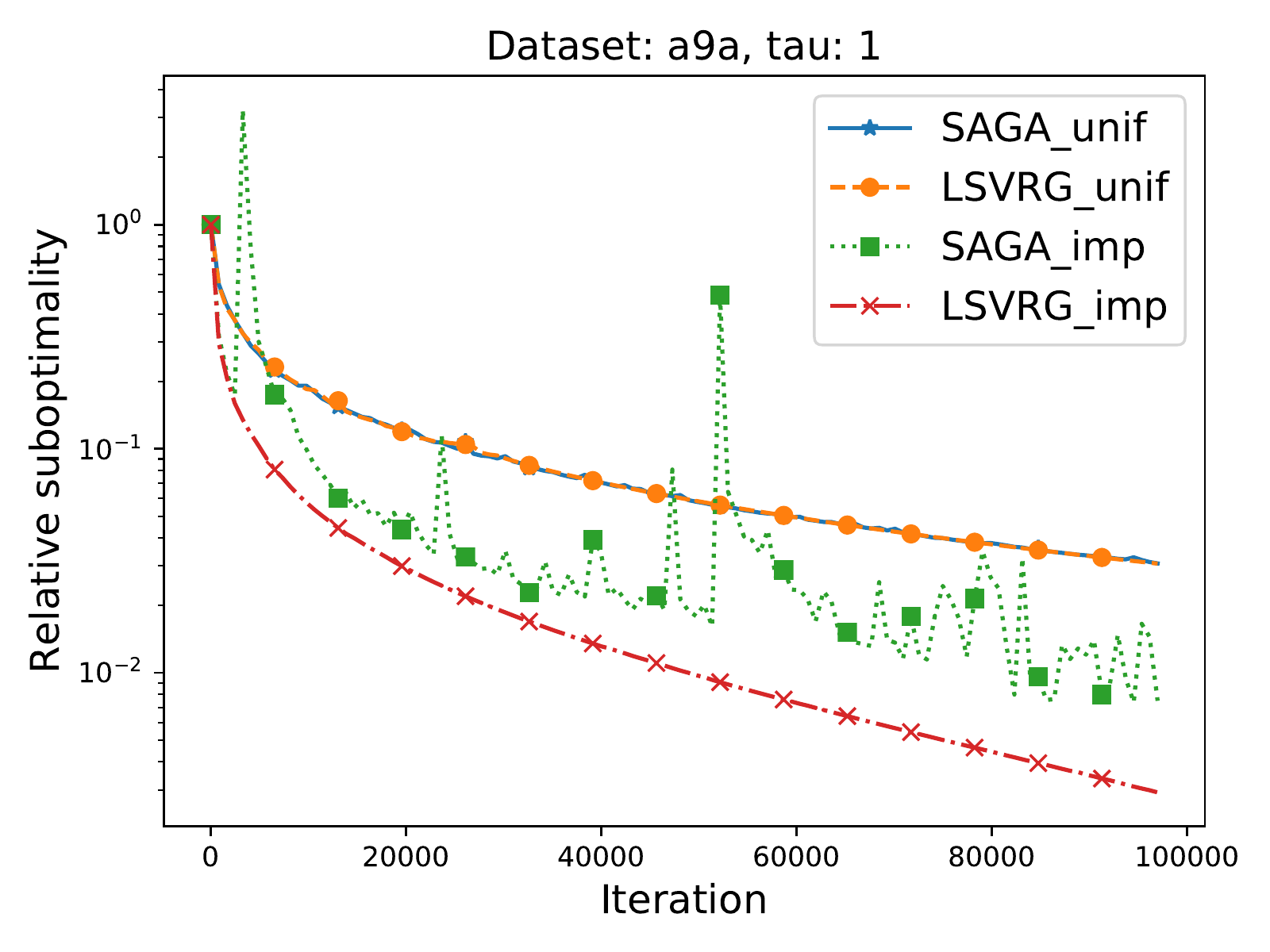}
        %\caption{ Residual vs. iteration  }\label{fig:bl_ex_flops}
\end{minipage}%
\begin{minipage}{0.3\textwidth}
  \centering
\includegraphics[width =  \textwidth ]{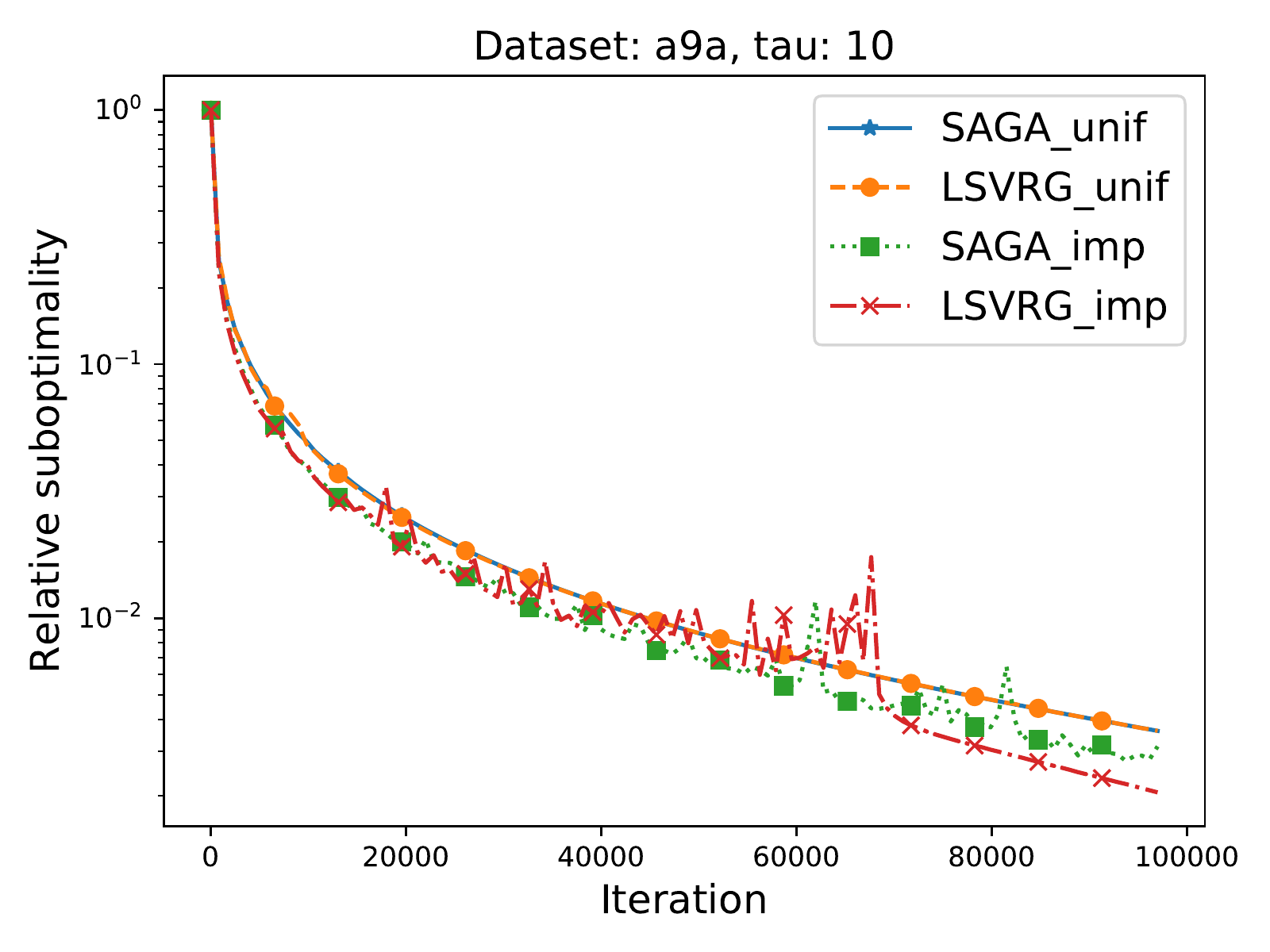}
        %\caption{ Residual vs. iteration  }\label{fig:bl_ex_flops}
\end{minipage}%
\begin{minipage}{0.3\textwidth}
  \centering
\includegraphics[width =  \textwidth ]{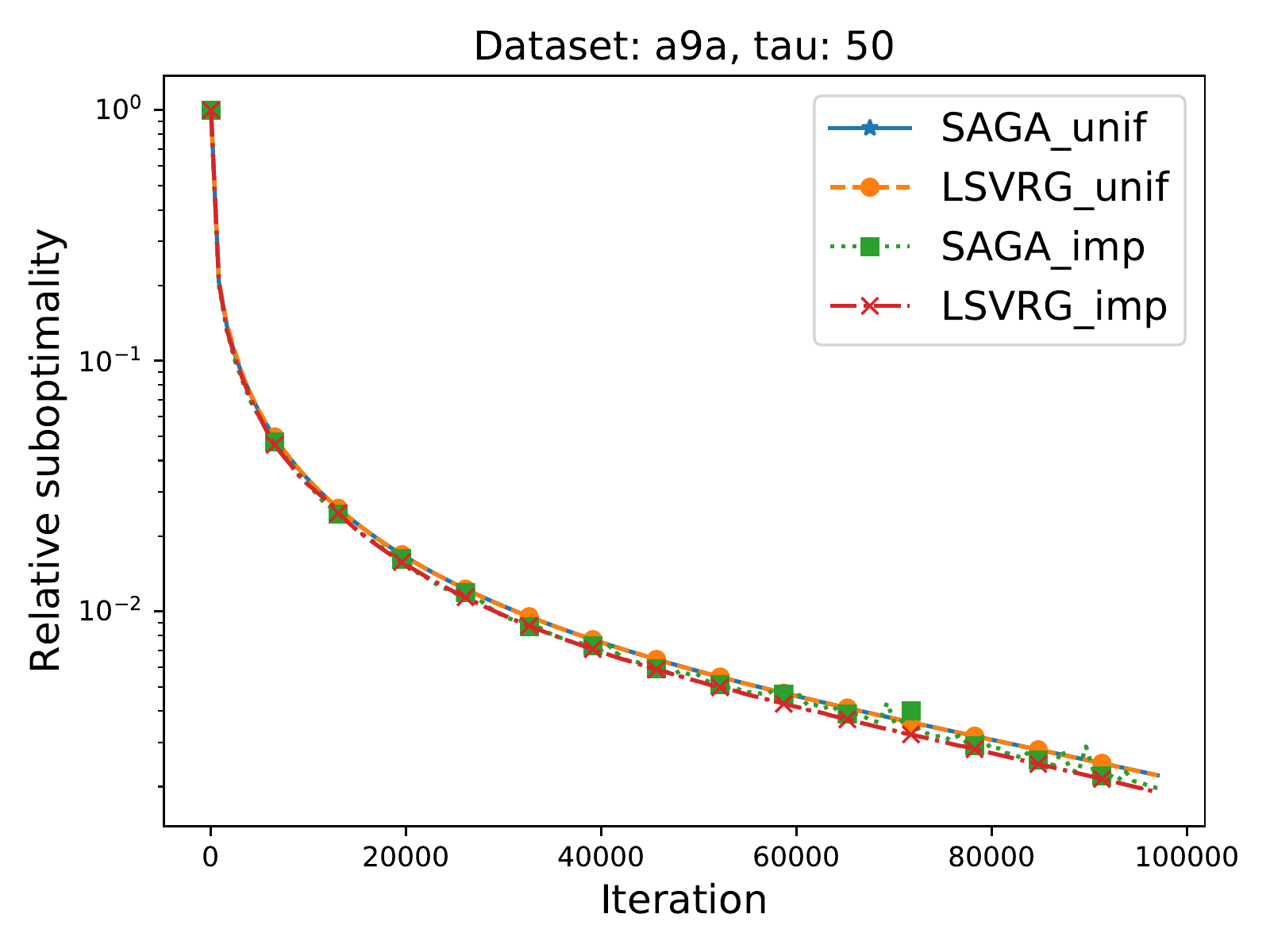}
        %\caption{ Residual vs. iteration  }\label{fig:bl_ex_flops}
\end{minipage}%s
\\
\begin{minipage}{0.3\textwidth}
  \centering
\includegraphics[width =  \textwidth ]{LSVRGDatasetw8atau1meth3.pdf}
        %\caption{ Residual vs. iteration  }\label{fig:bl_ex_flops}
\end{minipage}%
\begin{minipage}{0.3\textwidth}
  \centering
\includegraphics[width =  \textwidth ]{LSVRGDatasetw8atau10meth3.pdf}
        %\caption{ Residual vs. iteration  }\label{fig:bl_ex_flops}
\end{minipage}%
\begin{minipage}{0.3\textwidth}
  \centering
\includegraphics[width =  \textwidth ]{LSVRGDatasetw8atau50meth3.pdf}
        %\caption{ Residual vs. iteration  }\label{fig:bl_ex_flops}
\end{minipage}%s
\caption{{\tt LSVRG} applied on LIBSVM~\cite{chang2011libsvm} datasets. For {\tt a9a},  $\lambda = 0$ and $\probx = \frac1n$ was chosen; for {\tt w8a}, $\lambda = 10^{-8}$ and $\probx = \frac3n$ was chosen. Axis $y$ stands for relative suboptimality, i.e. $\frac{f(x^k)-f(x^*)}{f(x^k)-f(x^0)}$.
}
\label{fig:LSVRG2}
\end{figure}
 
\clearpage

\section{Several Lemmas}
 
\subsection{Existence lemma} 
\begin{lemma} \label{lem:existence}
Suppose that $\cX \in \Range{\cM}$. Denote $\piop(\mX)\eqdef \cU\mX\eR$. Suppose that $ \E{\left( \piop\cM^{\frac12} \right)^*\piop \cM^{\frac12} }$ exists and $\lambda_{\min}\left(\E{\cS}\right)>0$. Then, there are $\alpha>0$ and $\cB$ such that~\eqref{eq:small_step} and~\eqref{eq:small_step2} hold. Moreover, inequalities \eqref{eq:small_step}, \eqref{eq:small_step2} hold for $\alpha =0, \cB=0$ without any extra assumptions. 

\end{lemma}

\begin{proof}
Consider only $\alpha, \cB$ such that that $\alpha< \lambda_{\min}\left(\E{\cS}\right) \sigma^{-1}$, $\lambda_{\min}\left(\cB^*\cB\right)>0$, $\lambda_{\max}\left(\cB^*\cB\right)<\infty$.
Let $\mY  = {\cM^\dagger}^{\frac12} \mX$. Thus we have $  \E{ \norm{ \cU  \mX \eR }^2 } \leq   \|\mY\|^2\lambda_{\max}  \E{\left( \piop\cM^{\frac12} \right)^*\piop \cM^{\frac12} }$. 

 Thus 
\begin{eqnarray*}
  (1-\alpha \sigma) \NORMG{ \cB\mY } - \NORMG{  \left(\cI - \E{\cS} \right)^{\frac12}\cB  \mY} 
  &=& \<(\cB\mY)^\top, ( \E{\cS} - \alpha\sigma \cI) \cB\mY > \\
  &\geq& \left(\lambda_{\min}\left(\E{\cS}\right) - \alpha \sigma\cI\right)\| \cB\mY\|^2 \\
  & \geq&   \left(\lambda_{\min}\left(\E{\cS}\right) - \alpha \sigma\cI \right)\lambda_{\min}\left(\cB^*\cB\right)\| \mY\|^2\,.
\end{eqnarray*}
Therefore, to have~\eqref{eq:small_step}, it suffices to set
\[
\alpha \leq  \frac{ \lambda_{\min}\left(\E{\cS}\right) \lambda_{\min}\left(\cB^*\cB\right)}{  \sigma \lambda_{\min}\left(\cB^*\cB\right) + \frac{2}{n^2}  \lambda_{\max}  \left[\E{\left( \piop\cM^{\frac12} \right)^*\piop \cM^{\frac12} }\right]}.
\]
Similarly, to satisfy~\eqref{eq:small_step2}, it suffices to have
\[
\frac{2\alpha}{n} \lambda_{\max}\left(\E{\left( \piop\cM^{\frac12} \right)^*\piop \cM^{\frac12} } \right) +  n\lambda_{\min}\left(\E{\cS}\right) \lambda_{\max}\left(\cB^*\cB\right) \leq 1.
\]
A valid choice to satisfy the above is for example $\alpha, \cB$ such that
\[
\lambda_{\max}\left(\cB^*\cB\right) \leq \frac{1}{2n\lambda_{\min} \left(\E{\cS}\right) }, \quad \alpha \leq  \frac{1}{\frac{1}{n} \lambda_{\max}\left(\E{\piop \left(\cM^{\frac12} \right)^*\piop \cM^{\frac12} }\right) }.
\]

\end{proof}
 
 \subsection{Smoothness lemmas}
 
Let $h:\R^d\to \R$ be a differentiable and convex function. The  Bregman distance of $x$ and $y$ with respect to $h$ is defined by
\begin{equation} \label{eq:b987gf98f}
D_h(x,y) \eqdef h(x) - h(y) - \<\nabla h(y),x-y >.
\end{equation}

 \begin{lemma}[Lemma A.1 from~\cite{hanzely2018sega}]\label{lem:smooth22}
 Suppose that function $h:\R^d\to \R$ is convex and $\mM$-smooth, where $\mM\succeq 0$. Then 
\begin{equation}\label{eq:smooth_lemma}
D_h(x,y) \geq  \frac12 \norm{\nabla h(y)-\nabla h(x)}^2_{\mM^{\dagger}}, \quad \forall x,y \in \R^d.
\end{equation}
Further, 
\begin{equation} \label{eq:smooth_dotprod}
\<\nabla h(x) - \nabla h(y), x-y> \geq \| \nabla h(x) - \nabla h(y)\|^2_{\mM^{\dagger}}.
\end{equation}
 \end{lemma}

\begin{proof}
Fix $y$ and consider the function $\phi(x) \eqdef h(x) - \<\nabla h(y), x >$. Clearly, $\phi$ is $\mM$-smooth, and hence
\begin{equation}\label{eq:n98hf8gf}
\phi(x+d) \leq \phi(x) + \langle \nabla \phi(x),d\rangle + \frac{1}{2}\|d\|_{\mM}^2, \quad \forall x,d\in \R^d.
\end{equation}
 Moreover, since $h$ is convex, $\phi$ is convex, non-negative and is minimized at  $y$. Letting $t=\nabla h(x)-\nabla h(y)$, this implies that
\begin{eqnarray*}
\phi(y) &\leq & \phi \left(x - \mM^{\dagger}t \right) \\
&\overset{\eqref{eq:n98hf8gf}}{\leq} & \phi(x) -  \langle \nabla \phi(x),\mM^{\dagger}t\rangle + \frac{1}{2}\|\mM^{\dagger}t\|_{\mM}^2 \\
& = & \phi(x) -  \langle t ,\mM^{\dagger}t\rangle + \frac{1}{2}\|\mM^{\dagger}t\|_{\mM}^2 \\
&= & \phi(x) - \frac12\|t  \|^2_{\mM^{\dagger}},
\end{eqnarray*}
which is equivalent to~\eqref{eq:smooth_lemma}. In the last step we have used the identities $(\mM^\dagger)^\top = \left(\mM^\top \right)^\dagger = \mM^\dagger$ and $\mM^\dagger \mM \mM^\dagger = \mM^\dagger$.

To show~\eqref{eq:smooth_dotprod}, it suffices to sum inequality~\ref{eq:smooth_lemma} applied on vector pairs $(x,y)$ and $(y,x)$.

\end{proof}

\begin{lemma} \label{lem:smooth2} Let \eqref{eq:smooth_ass} hold. That is, assume that function $f_j$ are convex and $\mM_j$-smooth. Then
\begin{equation}\label{eq:smooth}
D_{f_j}(x,y) \geq \frac12 \norm{ \nabla f_j(x)-\nabla f_j(y) }^2_{\mM_j^{\dagger}}, \quad \forall x,y\in \R^d .
\end{equation}
If $x-y\in \Null{\mM_j}$, then 
\begin{enumerate} 
\item[(i)]  \begin{equation}\label{eq:linear_on_subspace} f_j(x) = f_j(y) + \langle \nabla f_j(y), x-y\rangle,\end{equation}
\item[(ii)]
\begin{equation} \label{eq:n98g8ff} \nabla f_j(x)-\nabla f_j(y) \in \Null{\mM_j},\end{equation}
\item[(iii)]  
\begin{equation} \label{eq:nb87sgb} \langle \nabla f_j(x) - \nabla f_j(y),x-y\rangle =0.\end{equation}
\end{enumerate}

If, in addition, $f_j$ is bounded below, then $\nabla f_j(x)  \in \Range{\mM_j}$ for all $x$.
\end{lemma}

\begin{proof} 
Inequality \eqref{eq:smooth} follows by applying Lemma~\ref{lem:smooth22} for $h=f_j$ and $\mM=\mM_j$. Identity \eqref{eq:linear_on_subspace} is a direct consequence of  \eqref{eq:smooth_ass}.  Combining \eqref{eq:smooth} and \eqref{eq:linear_on_subspace}, we get $0 \geq \frac12 \norm{ \nabla f_j(x)-\nabla f_j(y) }^2_{\mM_j^{\dagger}}$ , which implies that \begin{equation}\label{eq:nbui8gf7gdnjs}\nabla f_j(x)-\nabla f_j(y) \in \Null{\mM_j^\dagger} = \Null{\mM_j^\top} =\Null{\mM_j}  ,\end{equation} 
recovering \eqref{eq:n98g8ff}.  By adding two copies of  \eqref{eq:linear_on_subspace} (with the roles of $x$ and $y$ exchanged), we get \eqref{eq:nb87sgb}. Finally, if $f_j$ is bounded below, then in view of \eqref{eq:linear_on_subspace} there exists $c\in \R$ such that,
\[c \leq   \inf_{x \in  y +\Null{\mM_j}} f_j(x)  \overset{\eqref{eq:linear_on_subspace} }{=} \inf_{x \in  y +\Null{\mM_j}} f_j(y) +  \langle \nabla f_j(y), x-y\rangle  .\]
This implies that $\nabla f_j(y) \in \Range{\mM_j^\top} = \Range{\mM_j} $.

%On the other hand, fix any $y$ and define $\phi(x)\eqdef f_j(x) - f_j(y) - \langle \nabla f_j(y),x-y \rangle $. Since $f_j$ is convex, $\phi$ is also convex: \[\phi(y) \geq \phi(x) + \langle \nabla \phi(x),y-x\rangle.\] Notice that in view of \eqref{eq:linear_on_subspace}, $\phi(x)=0$ for  all $x\in y+ \Null{\mM_j}$. Hence,  \[0 \geq  \langle \nabla \phi(x),y-x\rangle = \langle \nabla f_j(x) - \nabla f_j(y), y-x \rangle\]

\end{proof}

\begin{lemma} \label{lem:smooth3}
Assume $f$ is twice continuously differentiable. Then $\mG(x)-\mG(y)\in \Range{\cM}$ for all $x,y\in \R^d$. 
\end{lemma}

\begin{proof}
For $\mG(x)-\mG(y)\in \Range{\cM}$, it suffices to show that $\nabla f_j(x) - \nabla f_j(y)\in \Range{\mM_j}$.
Without loss of generality, suppose that $f(z,w)$ (for $x  = [z,w]$) is such that $f(z,\cdot)$ is linear (for fixed $z$; from~\eqref{eq:linear_on_subspace}) and $f(\cdot,w)$ is $\mM'$ smooth for full rank $\mM'$. Note that

\[
0\preceq \nabla^2f(x) = \begin{pmatrix}
 \nabla_{ww}^2f(w,z) &  \nabla_{wz}^2f(w,z) \\
  \nabla_{zw}^2f(w,z) &  \nabla_{zz}^2f(w,z)
\end{pmatrix} = 
\begin{pmatrix}
 \nabla_{ww}^2f(w,z) &  \nabla_{wz}^2f(w,z) \\
  \nabla_{zw}^2f(w,z) &  0
\end{pmatrix} .
\]
Since every submatrix of the above must be positive definite, it is easy to see that we must have both $ \nabla_{wz}^2f(w,z)= 0$, $ \nabla_{zw}^2f(w,z) = 0$. This, however, means that $f(w,z)$ is separable in $z,w$. Therefore indeed  $\nabla f_j(x) - \nabla f_j(y)\in \Range{\mM_j}$ for all $x,y\in \R^d$ and all $j\in [n]$.

\end{proof}

\subsection{Projection lemma}

In the next lemma, we establish some basic properties of the interaction of the random projection matrices $\cS$ and $\cI - \cS$ with various matrices, operators, and norms.

\begin{lemma} \label{lem:nb98gd8fdx}
Let $\cS$ be a random projection operator and $\cA$ any deterministic linear operator commuting with $\cS$, i.e.,  $\cA \cS = \cS \cA$. Further, let $\mX,\mY \in \R^{d\times n}$ and define $\mZ = (\cI-\cS) \mX + \cS \mY$. Then
\begin{itemize}
\item[(i)] $\cA \mZ = (\cI-\cS) \cA \mX + \cS \cA \mY $,
\item[(ii)] $\norm{\cA \mZ}^2 = \norm{(\cI-\cS) \cA \mX}^2  + \norm{\cS \cA \mY}^2 $,
\item[(iii)] $\E{\norm{\cA \mZ}^2} = \norm{(\cI-\E{\cS})^{1/2} \cA \mX}^2  + \norm{\E{\cS}^{1/2} \cA \mY}^2 $, where the expectation is with respect to $\cS$.
\end{itemize}
\end{lemma}
\begin{proof} Part (i) follows by noting that $\cA$ commutes with $\cI-\cS$. Part (ii) follows from (i) by expanding the square, and noticing that  $(\cI-\cS)\cS = 0$. Part (iii) follows from (ii) after using the definition of the Frobenius norm, i.e., $\|\mM\|^2 =\Tr{\mM^\top \mM}$, the identities $(\cI -\cS)^2  = \cI -\cS$, $\cS^2 =\cS$,  and taking expectation on both sides. 
\end{proof}

\subsection{Decomposition lemma}

In the next lemma, we give a  bound on the expected squared distance of the gradient estimator $g^k$ from $\nabla f(x^*)$.

\begin{lemma}\label{lem:g_lemma}
For all $k\geq 0$ we have
\begin{equation}\label{eq:g_lemma}
\E{ \norm{ g^k - \nabla f(x^*)}^2} \leq    \frac{2}{n^2} \E{  \norm{ \cU \left(\mG(x^k) - \mG(x^*) \right) \eR}^2  } + \frac{2}{n^2}  \E{ \norm{ \cU \left(\mJ^k - \mG(x^*) \right) \eR}^2}.
\end{equation}
\end{lemma}
\begin{proof}
In view of \eqref{eq:ni98hffs} and since $\nabla f(x^*) = \frac{1}{n}\mG(x^*) \eR$, we have
\begin{equation}\label{eq:nb87fvdbs8s}
g^k-\nabla f(x^*) = \underbrace{\frac1n \cU \left(\mG(x^k)- \mG(x^*) \right) \eR}_{a}  +  \underbrace{\frac1n \left(\mJ^k - \mG(x^*) \right) \eR - \frac1n \cU \left( \mJ^k - \mG(x^*) \right) \eR}_{b} .
\end{equation}
Applying the bound $\norm{a+b}^2 \leq 2\norm{a}^2 + 2\norm{b}^2$ to \eqref{eq:nb87fvdbs8s} and taking expectations, we get
\begin{eqnarray*}
\E{ \norm{ g^k -\nabla f(x^*) }^2} &\leq &
 \E{ \frac{2}{n^2} \norm{  \cU \left(\mG(x^k)- \mG(x^*) \right) \eR  }^2} \\
 && \qquad + \E{ \frac{2}{n^2} \norm{   \left(\mJ^k-\mG(x^*) \right) \eR -   \cU \left(\mJ^k - \mG(x^*) \right) \eR }^2  }\\
&=&
 \frac{2}{n^2} \E{  \norm{ \cU \left(\mG(x^k) - \mG(x^*) \right) \eR }^2  } \\
 && \qquad + \frac{2}{n^2} \E{  \norm{ \left(\cI- \cU \right) \left(\mJ^k - \mG(x^*) \right) \eR }^2 }.
\end{eqnarray*}

It remains to note that
\begin{eqnarray*}
\E{  \norm{ \left(\cI-  \cU) (\mJ^k - \mG(x^*) \right) \eR }^2 }
 &=&
\E{ \norm{ \cU \left(\mJ^k - \mG(x^*) \right) \eR  }^2 } - \norm{ \left( \mJ^k - \mG(x^*) \right) \eR }^2
  \\
  &\leq& 
  \E{ \norm{ \cU \left(\mJ^k - \mG(x^*) \right) \eR }^2 } .
\end{eqnarray*}

\end{proof}

\section{Proof of Theorem~\ref{thm:main}}

For simplicity of notation, in this proof, all expectations are conditional on $x^k$, i.e., the expectation is taken with respect to the randomness of $g^k$.

Since
\begin{equation}\label{eq:prox_opt}
x^* = \prox(x^* - \alpha \nabla f(x^*)),
\end{equation}
and since the prox operator is non-expansive, we have
\begin{eqnarray}
\E{\norm{x^{k+1} -x^*}^2 } &\overset{ \eqref{eq:prox_opt}}{=} &
 \E{\norm{\prox(x^k-\alpha g^k) - \prox(x^*-\alpha \nabla f(x^*))  }^2}  
 \notag \\
 &\leq &
 \E{\norm{x^k-  x^* -\alpha(   g^{k} - \nabla f(x^*) ) }^2}  
 \notag\\
& \overset{\eqref{eq:unbiased_xx}}{=}& 
 \norm{x^k  -x^*}^2 -2\alpha \<  \nabla f(x^k)- \nabla f(x^*) , x^k  -x^*> \notag \\ 
 && \qquad  + \alpha^2\E{\norm{ g^{k}- \nabla f(x^*) }^2}
 \notag  \\ 
&\overset{\eqref{eq:strconv3}+ \eqref{eq:b987gf98f}}{\leq} & 
   (1-\alpha\sigma)\norm{x^k  -x^*}^2 +\alpha^2\E{\norm{  g^{k} -\nabla f(x^*) }^2}  \notag \\
   && \qquad -2\alpha D_f(x^k,x^*).
 \label{eq:convstepsub1XX}
\end{eqnarray}

Since $f(x)=\frac{1}{n}\sum_{j=1}^n f_j(x)$, in view of \eqref{eq:b987gf98f} and \eqref{eq:smooth} we have
\begin{eqnarray}
D_{f}(x^k,x^*) \quad \overset{\eqref{eq:b987gf98f}}{=} \quad  \frac{1}{n}\sum_{j=1}^n D_{f_j}(x^k,x^*) & \overset{\eqref{eq:smooth}}{\geq} &\frac{1}{2n} \sum_{j=1}^n \norm{\nabla f_j(x^k) -\nabla f_j(x^*)  }^2_{{\mM_j^{\dagger}} }  \notag\\
&=&  
 \frac{1}{2n}  \left\|{\cM^\dagger}^{\frac12}\left(\mG(x^k) -\mG(x^*) \right) \right\|^2. \label{eq:nb98gd8ff}
\end{eqnarray}

By combining \eqref{eq:convstepsub1XX} and \eqref{eq:nb98gd8ff}, we get
\begin{eqnarray*}
\E{\norm{x^{k+1} -x^*}^2 } & \leq &
   (1-\alpha\sigma)\norm{x^k  -x^*}^2 +\alpha^2\E{\norm{g^{k} -\nabla f(x^*) }^2}  \\
   && \qquad -\frac{\alpha}{n}  \norm{ {\cM^\dagger}^{\frac12} \left( \mG(x^k) -\mG(x^*) \right) }^2.
\end{eqnarray*}

Next, applying Lemma~\ref{lem:g_lemma} leads to the estimate
\begin{eqnarray}
\E{\norm{x^{k+1} -x^*}^2 } &\leq &
   (1-\alpha\sigma)\norm{x^k  -x^*}^2 -\frac{\alpha}{n} \norm{ {\cM^\dagger}^{\frac12} \left(\mG(x^k) -\mG(x^*) \right) }^2 \notag 
   \\
   && \qquad  +  \frac{2\alpha^2}{n^2} \E{ \norm{ \cU \left( \mG(x^k) - \mG(x^*) \right) \eR }^2  } \notag \\
   && \qquad + \frac{2\alpha^2}{n^2}  \E{ \norm{ \cU \left( \mJ^k - \mG(x^*) \right) \eR  }^2}   \label{eq:48u34719841234} .
\end{eqnarray}

In view of~\eqref{eq:nio9h8fbds79kjh}, we have $\mJ^{k+1} = (\cI-\cS)\mJ^k + \cS \mG(x^k)$, whence
\begin{equation} \label{eq:h98gf9hh89dsd}
\underbrace{\mJ^{k+1} - \mG(x^*)}_{\mZ} = (\cI-\cS) \underbrace{(\mJ^k -\mG(x^*))}_{\mX} + \cS \underbrace{(\mG(x^k) - \mG(x^*))}_{\mY}.
\end{equation}
Since, by assumption,  both $\cB$ and ${\cM^\dagger}^{\frac12}$ commute with $\cS$, so does their composition $\cA \eqdef \cB {\cM^\dagger}^{\frac12}$. Applying Lemma~\ref{lem:nb98gd8fdx}, we get
 \begin{eqnarray}\label{eq:J_jac_bound}\nonumber
\E{ \NORMG{\cB {\cM^\dagger}^{\frac12} \left(\mJ^{k+1}-\mG(x^*)  \right) }}  &=&  \NORMG{ (\cI - \E{\cS})^{\frac12}  \cB {\cM^\dagger}^{\frac12} \left(\mJ^k-\mG(x^*) \right) }  \\
&& +  \NORMG{\E{\cS}^{\frac12}  \cB {\cM^\dagger}^{\frac12} \left(\mG(x^k)-\mG(x^*) \right) } .
\end{eqnarray}

Adding $\alpha$-multiple of~\eqref{eq:J_jac_bound}  to~\eqref{eq:48u34719841234} yields
\begin{eqnarray*}
&&\E{\norm{x^{k+1} -x^*}^2 } + \alpha\E{ \NORMG{\cB {\cM^\dagger}^{\frac12} \left(\mJ^{k+1}-\mG(x^*)\right)}}
\\
&& \qquad  \leq 
   (1-\alpha\sigma)\norm{x^k  -x^*}^2 + \frac{2\alpha^2}{n^2} \E{  \norm{ \cU \left(\mG(x^k) - \mG(x^*) \right) \eR}^2  }    \\
   && \qquad \qquad  
 + \frac{2\alpha^2}{n^2}  \E{ \norm{ \cU \left(\mJ^k - \mG(x^*) \right) \eR }^2 } +   \alpha\NORMG{ (\cI - \E{\cS})^{\frac12} \cB {\cM^\dagger}^{\frac12} \left(\mJ^k-\mG(x^*) \right)} 
 \\
    && \qquad \qquad   + \alpha \NORMG{\E{\cS}^{\frac12}  \cB {\cM^\dagger}^{\frac12}\left(\mG(x^k)-\mG(x^*) \right)}-\frac{\alpha}{n} \norm{ {\cM^\dagger}^{\frac12}\left(\mG(x^k) -\mG(x^*) \right) }^2
   \\
&& \qquad  \stackrel{\eqref{eq:small_step}}{\leq} 
(1-\alpha\sigma)\norm{x^k  -x^*}^2 + (1-\alpha\sigma)\alpha \NORMG{\cB {\cM^\dagger}^{\frac12} \left(\mJ^{k}-\mG(x^*) \right)} \\
&& \qquad \qquad  
+\frac{2\alpha^2}{n^2} \E{  \norm{ \cU \left(\mG(x^k) - \mG(x^*) \right) \eR }^2  } +  \alpha \NORMG{\E{\cS}^{\frac12}  \cB {\cM^\dagger}^{\frac12} \left(\mG(x^k)-\mG(x^*) \right)}  \\
    && \qquad \qquad  
 -\frac{\alpha}{n} \left\| {\cM^\dagger}^{\frac12} \left(\mG(x^k) -\mG(x^*) \right) \right\|^2
   \\
&& \qquad  \stackrel{\eqref{eq:small_step2}}{\leq} 
(1-\alpha\sigma) \left( \norm{x^k  -x^*}^2 +\alpha \NORMG{\cB {\cM^\dagger}^{\frac12} \left(\mJ^{k}-\mG(x^*) \right)} \right).
\end{eqnarray*}
Above, we have used~\eqref{eq:small_step} with $\mX=  \mJ^k-\mG(x^*) $ and~\eqref{eq:small_step2} with $\mX=  \mG(x^k) -\mG(x^*) $.

\clearpage

\section{Special Cases: {\tt SAGA}-like Methods}

\subsection{Basic variant of {\tt SAGA}~\cite{SAGA} \label{sec:saga_basic}}

Suppose that for all $j$, $f_j$ is $m$-smooth (i.e., $\mM_j = m \mI_d$). To recover basic SAGA~\cite{SAGA}, consider the following choice of random operators $\cS,  \cU$: 
\[
(\forall j) \text{ with probability } \frac1n :  \quad \cS \mX =\mX \eRj \eRj^\top \quad \text{and} \quad   \cU \mX =\mX n \eRj \eRj^\top.
 \]

The resulting algorithm is stated as Algorithm~\ref{alg:SAGA}. Further, as a direct consequence of Theorem~\ref{thm:main}, convergence rate of {\tt SAGA} (Algorithm~\ref{alg:SAGA}) is presented in Corollary~\ref{cor:saga}.

\begin{algorithm}[h]
    \caption{{\tt SAGA} \cite{SAGA}}
    \label{alg:SAGA}
    \begin{algorithmic}
        \Require learning rate $\alpha>0$, starting point $x^0\in\R^d$
        \State Set $\psi_j^0 = x^0$ for each $j\in \{1,2,\dots,n\}$
        \For{ $k=0,1,2,\ldots$ }
        \State{Sample  $j \in [n]$ uniformly at random}
        \State{Set $\phi_j^{k+1} = x^k$ and $\phi_i^{k+1} = \phi_i^{k}$ for $i\neq j$}
        \State{$g^k = \nabla f_j(\phi_j^{k+1}) - \nabla f_j(\phi_j^k) + \frac{1}{n}\sum\limits_{i=1}^n\nabla f_i(\phi_i^k)$}
        \State{$x^{k+1} = \prox( x^k - \alpha g^k)$}
        \EndFor
    \end{algorithmic}
\end{algorithm}

\begin{corollary}[Convergence rate of {\tt SAGA}]\label{cor:saga}  Let $\alpha =  \frac{1}{4m + \sigma n}$. Then, iteration complexity of Algorithm~\ref{alg:SAGA} (proximal {\tt SAGA}) is $\left( 4\frac{m}{\sigma}+ n \right) \log\frac1\epsilon$. 
\end{corollary}

\subsection{{\tt SAGA} with arbitrary sampling \label{sec:saga_as}}

In contrast to Section~\ref{sec:saga_basic}, here we use the general matrix smoothness assumption, i.e., that $f_j$ is $\mM_j$ smooth. We  recover results from~\cite{qian2019saga}. Denote $\pR$ to be probability vector, i.e., $\pR_i = \Prob{i\in R}$ where $R$ is a random subset of $[n]$.

We shall consider the following choice of random operators $\cS,  \cU$: 
\[
(\forall R) \text{ with probability } \pRR :  \quad \cS \mX = \mX \sum_{j\in R}  \eRj \eRj^\top \quad \text{and} \quad   \cU \mX =\mX \sum_{j\in R} \frac{1}{\pRj} \eRj \eRj^\top .
 \]

The resulting algorithm is stated as Algorithm~\ref{alg:SAGA_AS_ESO}.
\begin{algorithm}[h]
    \caption{{\tt SAGA} with arbitrary sampling (a variant of \cite{qian2019saga})}
    \label{alg:SAGA_AS_ESO}
    \begin{algorithmic}
        \Require learning rate $\alpha>0$, starting point $x^0\in\R^d$, random sampling $R \subseteq \{1,2,\dots,n\}$
        \State Set $\phi_j^0 = x^0$ for each $j\in [n]$
        \For{ $k=0,1,2,\ldots$ }
        \State{Sample random  $R^k \subseteq \{1,2,\dots,n\}$}
        \State{Set $\phi_j^{k+1} = \begin{cases} x^k & j\in R^k\\  \phi_j^{k} &  j\not \in R^k \end{cases}$}
        \State{$g^k = \frac{1}{n}\sum\limits_{j=1}^n\nabla f_j(\phi_j^k) +\sum \limits_{j\in R^k} \frac{1}{n \pRj} \left( \nabla f_j(\phi_j^{k+1}) - \nabla f_j(\phi_j^k) \right)$}
        \State{$x^{k+1} = \prox(x^k - \alpha g^k)$}
        \EndFor
    \end{algorithmic}
\end{algorithm}

In order to give tight rates under $\mM$-smoothness, we need to do a  bit more work. First, let $v\in \R^n$ be a vector for which the following inequality {\em expected separable overapproximation} inequality holds
\begin{equation} \label{eq:ESO_saga}
\E{\left\|\sum_{j \in R} \mM^{\frac12}_{j} h_{j}\right\|^{2}} \leq \sum_{j=1}^{n} \pRj v_{j}\left\|h_{j}\right\|^{2}, \qquad \forall h_1,\dots,h_n \in \R^{d} .
\end{equation}
Since the function on the left is a quadratic in $h = (h_1,\dots,h_n) \in \R^{nd}$, this inequality is satisfied for large enough values of $v_j$.
A variant of~\eqref{eq:ESO_saga} was used to obtain the best known rates for coordinate descent with arbitrary sampling~\cite{ALPHA, ESO}.

Further, we shall consider the following assumption:
\begin{assumption}\label{as:pseudoinverse}
Suppose that  for all $k$ 
\begin{equation}\label{eq:g_in_range}
\mG(x^k) - \mG(x^*) = \cM^{\dagger} \cM \left(\mG(x^k) - \mG(x^*)\right)
\end{equation}
 and 
 \begin{equation}\label{eq:j_in_range}
\mJ^k - \mG(x^*) = \cM^{\dagger} \cM \left(\mJ^k - \mG(x^*)\right).
\end{equation}
\end{assumption}
The assumption, although in a slightly less general form, was demonstrated to obtain tightest complexity results for {\tt SAGA}~\cite{qian2019saga}. Note that if for each $j$, $f_j$ corresponds to loss function of a linear model, then~\eqref{eq:g_in_range} and~\eqref{eq:j_in_range} follow for free. Further, Lemmas~\ref{lem:smooth2} and~\ref{lem:smooth3} give some easy-to-interpret sufficient  sufficient conditions, such as lower boundedness of all functions $f_j$ (which happens for any loss function), or twice differentiability of all functions $f_j$.

\begin{corollary}[Convergence rate of {\tt SAGA}]\label{cor:saga_as2}  Let $ \alpha  = \min_j \frac{n \pRj}{4v_j + n\sigma}$. Then the iteration complexity of Algorithm~\ref{alg:SAGA_AS_ESO} is $\max_j \left(  \frac{4v_j  + n\sigma }{n \sigma \pRj}\right) \log\frac1\epsilon$. 
\end{corollary}

\begin{remark}
Corollary~\ref{cor:saga_as2} is slightly more general than Theorem 4.6 from~\cite{qian2019saga} does not explicitly require linear models and $\mM$ smoothness implied by the linearity.
\end{remark}

\section{Special Cases: {\tt SEGA}-like Methods}

Let $n=1$. Note that now operators $\cS$ and $\cU$ act on $d\times n$ matrices, i.e., on vectors in $\R^d$. To simplify notation, instead of $\mX\in \R^{d\times n}$ we will write $x = (x_1,\dots,x_d) \in \R^{d}$.

\subsection{Basic variant of {\tt SEGA}~\cite{hanzely2018sega}  \label{sec:sega}}

Suppose that  $f$ is $m$-smooth (i.e., $\mM_1=m\mI_d$) with $m>0$. To recover  basic SEGA from~\cite{hanzely2018sega}, consider the following choice of random operators $\cS$ and $\cU$: 
\[
(\forall i) \text{ with probability } \frac1d :  \quad \cS x = \eLi \eLi^\top x = x_i \eLi \quad \text{and} \quad   \cU x = d \eLi \eLi^\top x =  dx_i \eLi.
 \]
The resulting algorithm is stated as Algorithm~\ref{alg:SEGA}.

\begin{algorithm}[h]
    \caption{{\tt SEGA} \cite{hanzely2018sega}}
    \label{alg:SEGA}
    \begin{algorithmic}
        \Require Stepsize $\alpha>0$, starting point $x^0\in\R^d$
        \State Set $h^0 = 0$
        \For{ $k=0,1,2,\ldots$ }
        \State{Sample  $i\in \{1,2,\dots d \}$ uniformly at random}
        \State{Set $h^{k+1} = h^{k} +( \nabla_i f(x^k) - h^{k}_i)\eLi$}
        \State{$g^k = h^k+ d (\nabla_i f(x^k) - h_i^k) \eLi$}
        \State{$x^{k+1} = \prox(x^k - \alpha g^k)$}
        \EndFor
    \end{algorithmic}
\end{algorithm}

\begin{corollary}[Convergence rate of SEGA]\label{cor:sega}  Let $\alpha =  \frac{1}{4md+ \sigma d}$. Then the iteration complexity of Algorithm~\ref{alg:SEGA} is $\left( 4\frac{md}{\sigma}+d \right) \log\frac1\epsilon$. 
\end{corollary}

\subsection{{\tt SEGA} with arbitrary sampling \label{sec:sega_is_v1}}
Consider a more general setup to that in Section~\ref{sec:sega} and let us allow the smoothness matrix to be an arbitrary diagonal (positive semidefinite) matrix: $\mM = \diag(m_1,\dots,m_d)$ with $m_1,\dots,m_d>0$. In this regime, we will establish a convergence rate for an arbitrary sampling strategy, and then use this to develop importance sampling. 

Let $\pL\in \R^d$ be a probability vector with entries $\pLi = \Prob{i\in L}$. Consider the following choice of random operators $\cS$ and  $\cU$: 
\begin{equation} \label{eq:sega_choice}
(\forall L) \text{ with prob. } \pLL :  \; \cS x = \sum_{i\in L} \eLi \eLi^\top x = \sum_{i\in L} x_i \eLi  \quad \text{and} \quad   \cU x =  \sum_{i\in L} \frac{1}{\pLi} \eLi \eLi^\top x = \sum_{i\in L} \frac{x_i}{\pLi} \eLi .
\end{equation}
The resulting algorithm is stated as Algorithm~\ref{alg:SEGAAS}.

\begin{algorithm}[h]
    \caption{{\tt SEGA} with arbitrary sampling}
    \label{alg:SEGAAS}
    \begin{algorithmic}
        \Require Stepsize $\alpha>0$, starting point $x^0\in\R^d$, random sampling $L\subseteq \{1,2,\dots,d\}$
        \State Set $h^0 = 0$
        \For{ $k=0,1,2,\ldots$ }
        \State{Sample random  $L^k \subseteq \{1,2,\dots,d\}$}
        \State{Set $h^{k+1} = h^{k} +\sum \limits_{i\in L^k}( \nabla_i f(x^k) - h^{k}_i)\eLi$}
        \State{$g^k = h^k+\sum \limits_{i\in L^k}  \frac{1}{\pLi}(\nabla_i f(x^k) - h_i^k)\eLi$}
        \State{$x^{k+1} = \prox(x^k - \alpha g^k)$}
        \EndFor
    \end{algorithmic}
\end{algorithm}

\begin{corollary}[Convergence rate of {\tt SEGA}]\label{cor:sega_is_11}  Iteration complexity of Algorithm~\ref{alg:SEGAAS} with $\alpha  = \min_i \frac{\pLi}{4 m_i+ \sigma}$ is $  \max_i\left(\frac{4 m_i +\sigma}{\pLi\sigma} \right)\log\frac1\epsilon$. 
\end{corollary}
Corollary~\ref{cor:sega_is_11} indicates an up to constant factor optimal choice $\pLi \propto m_i$, which yields, up to a constant factor, $\frac{\sum_{i=1}^d m_i}{\sigma}\log\frac1\epsilon$ complexity.  
In the applications where $m$ is not unique\footnote{For example when a general matrix smoothness holds; one has to upper bound it by a diagonal matrix in order to comply with the assumptions of the section. In such case, there is an infinite array of possible choices of $m$.}, it is the best to choose one which minimizes $m^\top \eL$.

\begin{remark}
Note that if $\pLi=1$ for all $i$ (i.e., if $\cU=\cI$), we recover proximal gradient descent as a special case.
\end{remark}

\subsection{{\tt SVRCD} with arbitrary sampling \label{sec:svrcd_is2}}

As as a particular special case of Algorithm~\ref{alg:SketchJac} we get a new method, which  we call {\em Stochastic Variance Reduced Coordinate Descent} ({\tt SVRCD}). The algorithm is similar to {\tt SEGA}. The main difference is that {\tt SVRCD} does not {\em update} a subset $L$ of coordinates of vector $h^k$ each iteration. Instead, with probability $\probx$, it {\em sets} $h^k$ to $\nabla f(x^k)$.  

We choose $\cS$ and $\cU$ via
\[
  \cS \mX = 
  \begin{cases}
    0  & \text{w.p.}\quad 1- \probx \\
    \mX              & \text{w.p.}\quad  \probx
\end{cases} 
\quad \text{and} \quad   (\forall L)\quad \text{w.p.}\quad \pLL : \,  \cU \mX =  \sum_{i\in L} \frac{1}{\pLi} \eLi \eLi^\top \mX,
 \]
 where again $\pLi = \Prob{i\in L}$. The randomness of $\cS$ is independent from the randomness of $\cU$ (which comes from the randomness of $L$). The resulting algorithm is stated as Algorithm~\ref{alg:SVRCD}.

\begin{algorithm}[h]
    \caption{{\tt SVRCD} {\bf [NEW METHOD]}}
    \label{alg:SVRCD}
    \begin{algorithmic}
        \Require starting point $x^0\in\R^d$, random sampling $L \subseteq \{1,2,\dots,d\}$, probability $\probx$, stepsize $\alpha>0$
        \State Set $h^0 = 0$
        \For{ $k=0,1,2,\ldots$ }
        \State{Sample random  $L^k \subseteq \{1,2,\dots,d\}$}
        \State{$g^k = h^k+\sum \limits_{i\in L^k}  \frac{1}{\pLi}(\nabla_i f(x^k) - h_i^k) \eLi$}
        \State{$x^{k+1} = \prox(x^k - \alpha g^k)$}
        \State{Set $h^{k+1} =   \begin{cases}
    h^k  & \text{with probability} \quad 1- \probx \\
    \nabla f(x^k)              & \text{with probability} \quad \probx
\end{cases} $}
        \EndFor
    \end{algorithmic}
\end{algorithm}

As in Section~\ref{sec:sega_is_v1}, we shall assume that $f$ is $\mM= \diag(m_1, \dots, m_d)$- smooth.

\begin{corollary}\label{cor:svrcd}
Iteration complexity of Algorithm~\ref{alg:SVRCD} with $\alpha  = \min_i \frac{1}{4m_i / \pLi + \sigma / \probx}$ is $ \left(\frac{1}{\probx}  +  \max_i \frac{4m_i}{\pLi\sigma}\right) \log\frac1\epsilon$. 
\end{corollary}

Corollary~\ref{cor:svrcd} indicates optimal choice $p \propto m$.

\begin{remark}
If $\pLi=1$ for all $i$ and $\probx=1$, we recover proximal gradient descent as a special case.
\end{remark}

\section{Special Cases: {\tt SGD-star}} \label{sec:SGD-AS-star}

Suppose that $\mG(x^*) $ is known. We will show that shifted a version of {\tt SGD-AS} converges with linear rate in such case.  Let $\mJ^0 = \mG(x^*)$. Consider the following choice of random operators $\cS$,  $\cU$: 
\[
\cS \mX = 0  \quad \text{and} \quad    (\forall R) \text{ with probability } \pRR :  \quad \cU \mX =  \mX  \sum_{j\in R} \frac{1}{\pRj} \eRj \eRj^\top.
 \]
The resulting algorithm is stated as Algorithm~\ref{alg:SGD_AS}, which is in fact arbitrary sampling version of {\tt SGD-star} from~\cite{sigmak}.

\begin{algorithm}[h]
    \caption{{\tt SGD-star} \cite{sigmak} }
    \label{alg:SGD_AS}
    \begin{algorithmic}
        \Require learning rate $\alpha>0$, starting point $x^0\in\R^d$,  random sampling  $R \subseteq \{1,2,\dots,n\}$
        \For{ $k=0,1,2,\ldots$ }
        \State{Sample random  $R^k \subseteq \{1,2,\dots,n\}$ }
        \State{$g^k = \frac{1}{n}\mG(x^*) \eR +\sum \limits_{j\in R^k} \frac{1}{n \pRj} \left( \nabla f_j(x^k) -\nabla f_j(x^*) \right)$}
        \State{$x^{k+1} = \prox(x^k - \alpha g^k)$}
        \EndFor
    \end{algorithmic}
\end{algorithm}

\begin{corollary}[Convergence rate of {\tt SGD-AS-star}]\label{cor:sgd}  Suppose that $f_j$ is $\mM_j$-smooth for all $j$ and suppose that $v$ satisfies~\eqref{eq:ESO_saga}. Let $\alpha =n \min_j \frac{\pRj}{v_j} $. Then, iteration complexity of Algorithm~\ref{alg:SGD_AS} is $ \max_j \left( \frac{v_j}{n \pRj \sigma} \right) \log\frac1\epsilon$. 
\end{corollary}

\begin{remark}
In overparameterized models, one has  $\mG(x^*) =0 $. In such a case, Algorithm~\ref{alg:SGD_AS} becomes {\tt SGD-AS} \cite{gower2019sgd}, and we recover its tight convergence rate.
\end{remark}

\section{Special Cases: Loopless {\tt SVRG} with Arbitrary Sampling ({\tt LSVRG})} \label{sec:LSVRG-AS}

In this section we extend Loopless {\tt SVRG} (i.e., {\tt LSVRG}) from~\cite{hofmann2015variance, LSVRG} to arbitrary sampling.

The main difference to {\tt SAGA} is that {\tt LSVRG} does not update $\mJ^k$ at all with probability $1-\probx$. However, with probability $1-\probx$, it sets $\mJ^k$ to $\mG(x^k)$. 

Define $\cS$ and $\cU$ as follows:
\[
  \cS \mX = 
  \begin{cases}
    0  & \text{w.p.}\quad 1- \probx \\
    \mX              & \text{w.p.}\quad  \probx
\end{cases} 
\quad \text{and} \quad   (\forall R) \text{ with probability }\; \pRR : \,  \cU \mX =  \mX  \sum_{i\in R} \frac{1}{\pRj} \eRj \eRj^\top,
 \]
 where $\pRj = \Prob{j\in R}$.

The resulting algorithm is stated as Algorithm~\ref{alg:LSVRG-AS}.

\begin{algorithm}[h]
    \caption{{\tt LSVRG} ({\tt LSVRG} \cite{LSVRG} with arbitrary sampling) {\bf [NEW METHOD]}}
    \label{alg:LSVRG-AS}
    \begin{algorithmic}
        \Require learning rate $\alpha>0$, starting point $x^0\in\R^d$, random sampling $R\subseteq \{1,2,\dots,n\}$
        \State Set $\phi = x^0$
        \For{ $k=0,1,2,\ldots$ }
        \State{Sample a random subset $R^k \subseteq \{1,2,\dots n \}$ }
        \State{$g^k = \frac{1}{n}\sum\limits_{j=1}^n \nabla f_j(\phi^k) +\sum \limits_{j\in R^k} \frac{1}{n \pRj} \left( \nabla f_j(x^{k}) - \nabla f_j(\phi^k) \right)$}
        \State{$x^{k+1} = \prox(x^k - \alpha g^k)$}
        \State{Set $\phi^{k+1} = \begin{cases} x^k & \text{with probabiliy} \quad  \probx \\ \phi^{k} & \text{with probabiliy} \quad  1- \probx \end{cases}$}
        \EndFor
    \end{algorithmic}
\end{algorithm}

In order to give tight rates under $\mM$-smoothness, we shall consider ESO assumption~\eqref{eq:ESO_saga} and Assumption~\ref{as:pseudoinverse} (same as for {\tt SAGA-AS}).

The next corollary shows the convergence result. 
\begin{corollary}[Convergence rate of {\tt LSVRG}]\label{cor:lsvrg_as}  Let $ \alpha  = \min_j \frac{n}{4 \frac{v_j}{ \pRj} + \frac{ \sigma n}{\probx}}$. Then, iteration complexity of Algorithm~\ref{alg:LSVRG-AS} is $\max_j \left(  4 \frac{v_j}{n \sigma \pRj }  + \frac{1}{ \probx } \right) \log\frac1\epsilon$. 
\end{corollary}

\begin{remark}
One can consider a slightly more general setting with \[\cS \mX = 
  \begin{cases}
    0  & \text{w.p.}\quad 1- \probx \\
    \mX   \sum_{i\in R'} \eRj\eRj^\top    & \text{w.p.}\quad  \probx \end{cases},\] where distribution of $R'\subseteq [n] $ is arbitrary. Clearly, such methods is a special case of Algorithm~\ref{alg:SketchJac}, and setting $R' =[n]$ with probability 1, {\tt LSVRG} is obtained. However, in a general form, such algorithm resembles {\tt SCSG}~\cite{lei2017less}. However, unlike {\tt SCSG}, the described method converges linearly, thus is superior to {\tt SCSG}.
\end{remark}

\section{Special Cases: Methods with Bernoulli $\cU$}
Throughout this section, we will suppose that $ \mM_j = m \mI_d$ for all $j$. This is sufficient to establish strong results. Indeed, Bernoulli $\cU$ does not allow for an efficient importance sampling and hence one can't develop arbitrary sampling results similar to those in Section~\ref{sec:saga_as} or Section~\ref{sec:sega_is_v1}.  

\subsection{{\tt B2} (Bernoulli $\cS$)} \label{sec:B2}
Let $n=1$. Note that now operators $\cS$ and $\cU$ act on $d\times n$ matrices, i.e., on vectors in $\R^d$. To simplify notation, instead of $\mX\in \R^{d\times n}$ we will write $x = (x_1,\dots,x_d) \in \R^{d}$. Given  probabilities $0< \probx, \proby \leq 1$, let both $\cS$ and $ \cU$ be  Bernoulli (i.e., scaling) sketches: 
\[
\cS x= 
  \begin{cases}
    0  & \text{w.p.}\quad 1- \probx \\
    x              & \text{w.p.}\quad  \probx
\end{cases} 
\quad \text{and} \quad
\cU x = 
  \begin{cases}
    0  & \text{w.p.}\quad 1- \proby \\
   \frac{1}{\proby}x             & \text{w.p.}\quad  \proby
\end{cases} .
 \]
The resulting algorithm is stated as Algorithm~\ref{alg:B2}.

\begin{algorithm}[h]
    \caption{{\tt B2} {\bf [NEW METHOD]}}
    \label{alg:B2}
    \begin{algorithmic}
        \Require learning rate $\alpha>0$, starting point $x^0\in\R^d$, probabilities $\proby \in (0,1]$ and $\probx \in (0,1]$
        \State Set $\phi = x^0$
        \For{ $k=0,1,2,\ldots$ }
        \State{$g^k =   \begin{cases}
    \nabla f(\phi^k) & \text{w.p.}\quad 1- \proby \\
   \frac{1}{\proby}\nabla f(x^k) - \left(\frac{1}{\proby}-1\right) \nabla f(\phi^k)             & \text{w.p.}\quad  \proby
\end{cases} $}
        \State{$x^{k+1} = \prox(x^k - \alpha g^k)$}
        \State{Set $\phi^{k+1} = \begin{cases} x^k & \text{ w.p. }  \quad \probx\\ \phi^{k} &  \text{ w.p. }  \quad 1 -\probx \end{cases}$}
        \EndFor
    \end{algorithmic}
\end{algorithm}

\begin{corollary}[Convergence rate {\tt B2}]\label{cor:B2}  Suppose that $f$ is $m$-smooth. Let $\alpha = \frac{1}{4\frac{m}{ \proby} + \frac{\sigma}{ \probx} }$. Then, iteration complexity of Algorithm~\ref{alg:B2} is $\left( 4 \frac{m}{\sigma \proby}+ \frac{1}{\probx} \right) \log\frac1\epsilon$. 
\end{corollary}

\begin{remark}
It is possible to choose correlated $\cS$ and $\cU$ without any sacrifice in the rate. 
\end{remark}

 \subsection{{\tt LSVRG-inv} (Right $\cS$)} \label{sec:SVRG-1}

Given a probability scalar $0< \proby \leq 1$, consider choosing operators $\cS$ and $\cU$ as follows: 
\[
\cS \mX = \mX \sum_{j\in R} \eRj \eRj^\top \quad \text{w.p.}\quad \pRR
\quad \quad\text{and} \quad\quad
\cU \mX = 
  \begin{cases}
    0  & \text{w.p.}\quad 1- \proby \\
   \frac{1}{\proby} \mX              & \text{w.p.}\quad  \proby.
\end{cases} 
 \]
The resulting algorithm is stated as Algorithm~\ref{alg:invsvrg}.

\begin{algorithm}[h]
    \caption{{\tt LSVRG-inv} {\bf [NEW METHOD]}}
    \label{alg:invsvrg}
    \begin{algorithmic}
        \Require starting point $x^0\in\R^d$, random sampling $R\subseteq \{1,2,\dots,n\}$, probability $\proby \in (0,1]$ , learning rate $\alpha>0$
        \State Set $\phi_j^0 = x^0$ for $j=1,2,\dots,n$
        \For{ $k=0,1,2,\ldots$ }
        \State{$g^k =   \begin{cases}
   \frac1n\sum \limits_{j=1}^n \nabla f_j(\phi_j^k) & \text{w.p.}\quad 1- \proby \\
   \frac{1}{\proby}\nabla f(x^k) - \left(\frac{1}{\proby}-1 \right)  \frac1n\sum \limits_{j=1}^n \nabla f_j(\phi_j^k)            & \text{w.p.}\quad  \proby
\end{cases} $}
        \State{$x^{k+1} = \prox(x^k - \alpha g^k)$}
    \State{Sample a random subset $R^k \subseteq \{1,2,\dots n \}$ }
        \State{Set $\phi_j^{k+1} = \begin{cases} x^k &  \quad j\in R^k\\ \phi_j^{k} &  \quad j \notin R^k \end{cases}$}        
        \EndFor
    \end{algorithmic}
\end{algorithm}

\begin{corollary}[Convergence rate of {\tt LSVRG-inv}]\label{cor:inverse_svrg}  Suppose that each $f_i$ is $m$-smooth. Let $\alpha = \min_j \frac{1}{4\frac{m}{\proby} + \frac{\sigma}{\pRj}}$. Then, iteration complexity of Algorithm~\ref{alg:invsvrg} is $\max_j \left( 4\frac{m}{\sigma \proby}+ \frac{1}{\pRj} \right) \log\frac1\epsilon$. 
\end{corollary}

 \subsection{{\tt SVRCD-inv} (Left $\cS$)} \label{sec:SVRCD-inv}

Let $n=1$. Note that now operators $\cS$ and $\cU$ act on $d\times n$ matrices, i.e., on vectors in $\R^d$. To simplify notation, instead of $\mX\in \R^{d\times n}$ we will write $x = (x_1,\dots,x_d) \in \R^{d}$.

Consider again setup where $n=1$. Choose operators $\cS$ and $\cU$ as follows: 
\[
\cS x= \sum_{i\in L} \eLi \eLi^\top x \quad \text{w.p.}\quad  \pLL
\quad \text{and} \quad
\cU x = 
  \begin{cases}
    0  & \text{w.p.}\quad 1- \proby \\
   \frac{1}{\proby} x              & \text{w.p.}\quad  \proby \;.
\end{cases} 
 \]
For convenience, let $\pL$ be the probability vector defined as: $\pLi = \Prob{i\in L}$.

The resulting algorithm is stated as Algorithm~\ref{alg:B_sega}.

\begin{algorithm}[h]
    \caption{{\tt SVRCD-inv} {\bf [NEW METHOD]}}
    \label{alg:B_sega}
    \begin{algorithmic}
        \Require  starting point $x^0\in\R^d$, random sampling $L\subseteq \{1,2,\dots,d\}$, probability $\proby \in (0,1]$, learning rate $\alpha>0$
        \State Choose $h^0 \in \R^d$
        \For{ $k=0,1,2,\ldots$ }
        \State{$g^k =   \begin{cases}
    h^k & \text{w.p.}\quad 1- \proby \\
   \frac{1}{\proby}\nabla f(x^k) - \left(\frac{1}{\proby}-1 \right)h^k             & \text{w.p.}\quad  \proby
\end{cases} $}
        \State{$x^{k+1} = \prox(x^k - \alpha g^k)$}
            \State{Sample a random subset $L^k \subseteq \{1,2,\dots d \}$ }
        \State{Set $h^{k+1} = h^k + \sum \limits_{i\in L^k} (\nabla_i f(x^k)  - h^k_i )\eLi$}
        \EndFor
    \end{algorithmic}
\end{algorithm}

\begin{corollary}[Convergence rate of {\tt SVRCD-inv}]\label{cor:SVRCD-inv}  Suppose that each $f_j$ is $m$-smooth. Let $\alpha = \min_i \frac{1}{4 \frac{m}{\proby} + \frac{\sigma}{ \pLi}}$. Then, iteration complexity of Algorithm~\ref{alg:B_sega} is $\max_i \left( 4\frac{m}{\sigma \proby}+ \frac{1}{\pLi} \right) \log\frac1\epsilon$. 
\end{corollary}

\section{Special Cases: Combination of Left and Right Sketches}

 \subsection{{\tt RL} (right sampling $\cS$, left unbiased sampling $\cU$)}  \label{sec:RL}

Consider choosing $\cS$ and $\cU$ as follows: 
\[
\cS \mX= \mX \sum_{j\in R} \eRj \eRj^\top \quad \text{w.p.} \quad \pR_{R}
\quad \text{and} \quad
\cU \mX= \sum_{i\in L} \frac{1}{\pLi} \eLi  \eLi^\top \mX  \quad \text{w.p.} \quad \pL_{L} \;.
 \]
 
The resulting algorithm is stated as Algorithm~\ref{alg:RL}.

\begin{algorithm}[h]
    \caption{{\tt RL} {\bf [NEW METHOD]}}
    \label{alg:RL}
    \begin{algorithmic}
        \Require starting point $x^0\in\R^d$, random sampling $L\subseteq \{1,2,\dots,d\}$,  random sampling $R\subseteq \{1,2,\dots,n\}$, learning rate $\alpha>0$
        \State Set $\phi_j^0 = x^0$ for each $j$
        \For{ $k=0,1,2,\ldots$ }
        \State{Sample random $R^k \subseteq \{1,2,\dots,n\}$ }
        \State{Set $\phi_j^{k+1} = \begin{cases} x^k & \quad j\in R^k\\ \phi_j^{k} & \quad j\not \in R^k \end{cases}$}
        \State{Sample random $L^k \subseteq \{1,2,\dots,d\}$}
        \State{$g^k = \frac{1}{n}\sum\limits_{j=1}^n\nabla f_j(\phi_j^k) +\sum \limits_{i\in L^k} \frac{1}{\pLi}\left( \nabla_i f(x^k) -\frac{1}{n}\sum\limits_{j=1}^n\nabla_i f_j(\phi_j^k) \right)\eLi$}
        \State{$x^{k+1} = \prox(x^k - \alpha g^k)$}
        \EndFor
    \end{algorithmic}
\end{algorithm}

\begin{corollary}[Convergence rate of {\tt RL}]\label{cor:RL}  Suppose that each $f_j$ is $\diag(m^j)$-smooth, where $m^j\in \R^d$ and $\diag(m^j)\succ 0$. Let $\alpha =\min_{i,j} \left( 4 \frac{m_i^j }{\pLi} + \frac{\sigma}{ \pRj} \right)^{-1}$. Then, iteration complexity of Algorithm~\ref{alg:RL} is $\max_{i,j} \left( 4\frac{m_i^j}{\sigma \pLi} +  \frac{1}{\pRj} \right) \log\frac1\epsilon$. 
\end{corollary}

 \subsection{{\tt LR} (left sampling $\cS$, right unbiased sampling $\cU$)}  \label{sec:LR}
 
 Consider choosing $\cS$ and $\cU$ as follows: 
\[
\cS \mX=  \sum_{i \in L} \eLi \eLi^\top  \mX\quad \text{w.p.} \quad \pL_{L}
\quad \text{and} \quad
\cU \mX=  \mX  \sum_{j\in R} \frac{1}{\pRj} \eRj  \eRj^\top \quad \text{w.p.} \quad \pR_{R} \;.
 \]

The resulting algorithm is stated as Algorithm~\ref{alg:LR}.

\begin{algorithm}[h]
    \caption{{\tt LR} {\bf [NEW METHOD]}}
    \label{alg:LR}
    \begin{algorithmic}
        \Require starting point $x^0\in\R^d$, random sampling $L\subseteq \{1,2,\dots,d\}$,  random sampling $R\subseteq \{1,2,\dots,n\}$, learning rate $\alpha>0$
        \State Set $h^0 = x^0$ for each $j$
        \For{ $k=0,1,2,\ldots$ }
        \State{Sample random $L^k \subseteq \{1,2,\dots,d\}$ }
        \State{Set $h^{k+1} = h^{k} +\sum \limits_{i\in L^k} ( \nabla_i f(x^k) - h^{k}_i)\eLi$}
        \State{Sample random $R^k \subseteq \{1,2,\dots,n\}$ }
        \State{$g^k = \nabla f(h^k) +\sum \limits_{j\in R^k} \frac{1}{n\pRj}\left( \nabla f_j(x^k) - \nabla f_j(h^k) \right)$}
        \State{$x^{k+1} = \prox(x^k - \alpha g^k)$}
        \EndFor
    \end{algorithmic}
\end{algorithm}

\begin{corollary}[Convergence rate of {\tt LR}]\label{cor:LR}  Suppose that each $f_j$ is $\mM_j$-smooth, and suppose that $v \in \R^{n}$ is such that~\eqref{eq:ESO_saga} holds. Let $\alpha = \min_{i,j} \frac{1}{4 v_j\pRj^{-1} + \sigma\pLi^{-1}}$. Then, iteration complexity of Algorithm~\ref{alg:LR} is $\max_{i,j} \left( 4\frac{v_i}{\sigma \pRj}+ \frac{1}{\pLi} \right) \log\frac1\epsilon$. 
\end{corollary}

\section{Special Cases: Joint Left and Right Sketches}

\subsection{{\tt SAEGA} \label{sec:SAEGA}}

Another  new special case of Algorithm~\ref{alg:SketchJac} we propose is {\tt SAEGA} (the name comes from the combination of names {\tt SAGA} and {\tt SEGA}). In {\tt SAEGA}, both $\cS$ and $\cU$ are fully correlated and consist of right and left sketch. However, the mentioned right and left sketches are independent. In particular, we have
\[
\cS \mX = \mX_{LR}  = \left( \sum_{i\in L} \eLi \eLi^\top \right)\mX  \left( \sum_{j\in R} \eRj \eRj^\top \right), \quad L\subset[d], R\subset[n] \quad \text{ are independent random sets.}
\]
Next, $\cU$ is chosen as 
\[
\cU \mX = \cS \left( \left( {\pL}^{-1} \left(\pR^{-1}\right)^\top\right) \circ \mX\right)
\] 
where $\pLi = \Prob{i\in  L}$ and $\pRj =  \Prob{j\in  R}$. The resulting algorithm is stated as Algorithm~\ref{alg:saega}. 

\begin{algorithm}[h]
  \caption{{\tt SAEGA} {\bf [NEW METHOD]}}
  \label{alg:saega}
\begin{algorithmic}
\State{\bfseries Input: }{$x^0\in\R^d$, random sampling $L\subseteq \{1,2,\dots,d\}$, random sampling  $R \subseteq \{1,2,\dots,n\}$, stepsize $\alpha$ }
\State $\mJ^0  = 0$
  \For{$k=0,1,\dotsc$}
         \State Sample random $ L^k \subseteq \{1,2,\dots,d\}$ and $ R^k \subseteq \{1,2,\dots,n\}$
        \State Compute $\nabla_i f_{j}(x^k)$ for all $i\in L^k$ and $j \in R^k $ 
        \State $\mJ^{k+1}_{ij} = \begin{cases} \nabla_{i} f_{j}(x^k) &   i\in L^k \text{ and } j\in  R^k \\  \mJ^k_{ij} & \text{otherwise} \end{cases}$
    \State $g^k = \left(\mJ^k + \left( {\pL}^{-1} \left(\pR^{-1}\right)^\top\right)\circ(\mJ^{k+1}- \mJ^k)\right) \eR$
    \State $x^{k+1} =\prox(x^k - \alpha g^k)$
  \EndFor
\end{algorithmic}
\end{algorithm}

Suppose that for all $j\in [n]$, $\mM_j = \diag(m^j)\succ 0$ is diagonal matrix\footnote{A block diagonal matrix $\mM_j$ with blocks such that $\cM\cP_\cS = \cP_\cS \cM $ would work as well}. 

Let $\PR \in \R^{n\times n}$ be the probability matrix with respect to $R$-sampling , i.e., $\PR_{ jj'} = \Prob{j\in R, j' \in R}$. 

\begin{corollary}\label{cor:saega}  Consider any (elementwise) positive vector $\qR$ such that $ \diag(\pR)^{-1} \PR \diag(\pR)^{-1}  \preceq  \diag(\qR)^{-1}$.  Let $\alpha =  \min_{i,j} \frac{n\pLi \qRj}{4m^j_i + n\sigma}$. Then, iteration complexity of Algorithm~\ref{alg:saega} is $  \max_{i,j}   \left( 4\frac{m^j_i}{\sigma n \pLi \qRj}  + \frac{1}{\pLi} \frac{1}{\qRj}\right)
\log\frac1\epsilon$. 
\end{corollary}

\subsection{{\tt SVRCDG} \label{sec:SVRCDG}} 

% \filip{Name: SAEGA - (SAGA + SEGA) + (SVRCD + LSVRG)}

Next new special case of Algorithm~\ref{alg:SketchJac} we propose is {\tt SVRCDG}. {\tt SVRCDG} uses the same random operator $\cU$ as {\tt SAEGA}. The difference to {\tt SAEGA} lies in operator $\cS$ which is Bernoulli random variable:
\[
\cS \mX = 
\begin{cases}
0 &   \text{w.p.}\quad 1-\probx \\
\mX & \text{w.p.}\quad \probx \\
\end{cases},\qquad  
\cU \mX = \mI_{L:}\left( \left( {\pL}^{-1} \left(\pR^{-1}\right)^\top\right) \circ \mX\right)\mI_{:R},
\] 
where $L\subseteq [d]$, and $R\subseteq [n]$ are independent random sets and  $\pLi = \Prob{i\in  L}$ and $\pRj =  \Prob{j\in  R}$.

 The resulting algorithm is stated as Algorithm~\ref{alg:svrcdg}. 

\begin{algorithm}[h]
  \caption{{\tt SVRCDG} {\bf [NEW METHOD]}}
  \label{alg:svrcdg}
\begin{algorithmic}
\State{\bfseries Input: }{$x^0\in\R^d$, random sampling $L\subseteq \{1,2,\dots,d\}$, random sampling  $R \subseteq \{1,2,\dots,n\}$, stepsize $\alpha$, probability $\probx$ }
\State $\mJ^0  = 0$
  \For{$k=0,1,\dotsc$}
         \State Sample random $ L^k \subseteq \{1,2,\dots,d\}$ and $ R^k \subseteq \{1,2,\dots,n\}$
        \State Observe $\nabla_i f_{j}(x^k)$ for all $i\in L^k$ and $j \in R^k$ 
    \State $g^k = \left(\mJ^k + \left( {\pL}^{-1}\left(\pR^{-1}\right)^\top\right)\circ \left(\mI_{L^k:} \left(\mG(x^k)- \mJ^k \right) \mI_{:R^k} \right)\right) \eR$
    \State $x^{k+1} =\prox(x^k - \alpha g^k)$
                \State  $\mJ^{k+1} = \begin{cases} \mG(x^k) & \text{with probability} \quad\probx \\
                 \mJ^k & \text{with probability} \quad 1-\probx \end{cases}$
  \EndFor
\end{algorithmic}
\end{algorithm}

Suppose that for all $j$, $\mM_j = \diag(m^j)$ is diagonal matrix\footnote{Block diagonal $\mM_j$ with blocks such that $\cM\cS = \cS \cM $ would work as well}. 

For notational simplicity, denote $\mM'\in \R^{d\times n}$ to be the matrix with $j$-th column equal to $m_j$.

Let $\PR \in \R^{n\times n}$ be the probability matrix with respect to $R$ - sampling , i.e., $\PR_{ jj'} = \Prob{j\in R, j' \in R}$. 

\begin{corollary}\label{cor:svrcdg}   Consider any (elementwise) positive vector $\qR$ such that $ \diag(\pR)^{-1} \PR \diag(\pR)^{-1}  \preceq  \diag(\qR)^{-1}$.  Let $\alpha =  \min_{i,j} \frac{1}{4\frac{m^j_i}{\pLi \qRj n} + \frac{1}{\probx}\sigma}$. Then, iteration complexity of Algorithm~\ref{alg:svrcdg} is $  \max_{i,j}   \left( 4\frac{m^j_i}{\sigma n \pLi \qRj}  + \frac{1}{\probx}\right)
\log\frac1\epsilon$. 

%\filip{ I believe we have some answers. This is just ESO for $\mM$ which is a matrix of ones (for diagonal $\mM_j$, randomness and smoothness can be decomposed; thus we obtain what is above). So if one samples a single function each iteration, then $\pR=\qR$. If more, then ESO arguments would follow... }

\end{corollary}

\subsection{{\tt ISAEGA}  (with distributed data) \label{sec:ISAEGA}}

In this section, we consider a distributed setting from~\cite{mishchenko201999}. In particular,~\cite{mishchenko201999} proposed a strategy of running coordinate descent on top of various optimization algorithms such as {\tt GD}, {\tt SGD} or {\tt SAGA}, while keeping the convergence rate of the original method. This allows for sparse communication from workers to master. 

However, {\tt ISAGA} (distributed {\tt SAGA} with {\tt RCD} on top of it), as proposed, assumes zero gradients at the optimum which only holds for overparameterized models. It was stated as an open question whether it is possible to derive {\tt SEGA} on top of it such that the mentioned assumption can be dropped. We answer this question positively, proposing {\tt ISAEGA} (Algorithm~\ref{alg:isaega}). Next, algorithms proposed in~\cite{mishchenko201999} only allow for uniform sampling under simple smoothness. In contrast, we develop an arbitrary sampling strategy for general matrix smoothness\footnote{We do so only for {\tt ISAEGA}. However, our framework allows obtaining arbitrary sampling results for {\tt ISAGA}, {\tt ISEGA} and {\tt ISGD} (with no variance at optimum) as well. We omit it for space limitations}. 

Assume that we have $\TR$ parallel units, each owning set of indices $\NRt$ (for $1\leq \tR\leq \TR$). Next, consider distributions $\cDR_\tR$ over subsets of $\NRt$ and distributions $\cDL_{\tR}$ over subsets coordinates $[d]$ for each machine. Each iteration we sample $R_\tR \sim\cDR_{\tR}, L_\tR \sim\cDL_{\tR}$ (for $1\leq \tR\leq \TR$) and observe the corresponding part of Jacobian $\mJ^k_{\cap_{\tR} (L_\tR,R_\tR)}$. Thus the corresponding random Jacobian sketch becomes 
\[
\cS \mX = \mX_{\cap_\tR (L_\tR,R_\tR)}  = \sum_{\tR=1}^\TR  \left( \sum_{i\in L_\tR} \eLi \eLi^\top \right)\mX_{:\NRt} \left( \sum_{j\in R_\tR} \eRj \eRj^\top \right).
\]

 Next, for each $1\leq \tR \leq \TR$ consider vector $\ptL \in \R^d$, $\ptR\in \R^{|\NRt|}$ such that $\Prob{ i\in L_\tR} =\ptLi$ and $\Prob{ j\in R_\tR} =\ptRj$. Given the notation, random operator $\cU$ is chosen as 

\[
\cU \mX = \sum_{\tR=1}^\TR   \left( \left(\ptL\right)^{-1} \left(\left(\ptR\right)^{-1} \right)^\top\right) \circ \left( \left( \sum_{i\in L_\tR} \eLi \eLi^\top \right)\mX_{:\NRt} \left( \sum_{j\in R_\tR} \eRj \eRj^\top \right)\right). 
\] 
The resulting algorithm is stated as Algorithm~\ref{alg:isaega}. 

\begin{algorithm}[h]
  \caption{{\tt ISAEGA} {\bf [NEW METHOD]}}
  \label{alg:isaega}
\begin{algorithmic}
\State{\bfseries Input: }{$x^0\in\R^d$, \# parallel units $\TR$, each owning set of indices $N_\tR$ (for $1\leq \tR\leq \TR$), distributions $\cDR_t$ over subsets of $ \NRt$, distributions $\cDL_{\tR}$ over subsets coordinates $[d]$, stepsize $\alpha$ }
\State $\mJ^0  = 0$
  \For{$k=0,1,\dotsc$}
    \For{$\tR=1,\dotsc,\TR$ in parallel}
         \State Sample $ R_\tR \sim \cDR_\tR$; $R_\tR\subseteq  \NRt$ (independently on each machine)
        \State Sample $ L_\tR \sim \cDL_\tR$; $L_\tR\subseteq [d]$  (independently on each machine)
        \State Observe $\nabla_{L_\tR} f_{j}(x^k)$ for $j \in R_\tR $ 
        \State For $i\in [d], j\in \NRt$ set $\mJ^{k+1}_{i,j} = \begin{cases}
        \nabla_{i} f_{j}(x^k) & \text{if} \quad i\in [d], j\in R_\tR, i\in L_{\tR} \\
        \mJ^{k}_{i,j} & \text{otherwise}
        \end{cases} $
        \State Send $\mJ^{k+1}_{:\NRt }- \mJ^k_{:\NRt }$ to master \Comment{Sparse; low communication}
    \EndFor
    \State $g^k = \left(\mJ^k + \sum \limits_{\tR=1}^\TR   \left( {{\ptL}^{-1}} {{\ptR}^{-1}}^\top\right) \circ \left( \left( \sum_{i\in L_\tR} \eLi \eLi^\top \right)\left(\mJ^{k+1}-\mJ^k   \right)_{:\NRt} \left( \sum_{j\in R_\tR} \eRj \eRj^\top \right)\right) \right)\eR$
    \State $x^{k+1} =\prox(x^k - \alpha g^k)$
  \EndFor
\end{algorithmic}
\end{algorithm}

Suppose that for all $1\leq j\leq n$, $\mM_j = \diag(m^j)$ is diagonal matrix\footnote{block diagonal $\mM_j$ with blocks such that $\cM\cS = \cS \cM $ would work as well}. Let $\PtR \in \R^{ \| \NRt \|\times  \| \NRt \|}$ be the probability matrix with respect to $R_{\tR}$ - sampling , i.e., $\PtR_{jj'} = \Prob{j\in R_{\tR}, j' \in R_{\tR}}$.

\begin{corollary}\label{cor:isaega} 
For all $\tR$ consider any (elementwise) positive vector $\qtR$ such that $ \diag(\ptR)^{-1} \PtR \diag(\ptR)^{-1}  \preceq  \diag(\qtR)^{-1}$. Let $\alpha =  \min_{j\in \NRt, i,\tR} \frac{1}{4 m^j_i \left( 1+ \frac{1}{n\ptLi \qtRj}\right) + \frac{\sigma}{\ptLi }\qtRj}$. Then, iteration complexity of Algorithm~\ref{alg:isaega} is $  \max_{j\in \NRt, i,\tR}  \left( 4\frac{m^j_i}{\sigma}\left( 1+ \frac{1}{ n\ptLi\qtRj} \right)  +  \frac{1}{\ptLi\qtRj} \right)
\log\frac1\epsilon$. 
\end{corollary}
Thus, for all $j$, it does not make sense to increase sampling size beyond point where $\ptLi\qtRj \geq \frac1n$as the convergence speed would not increase significantly\footnote{For indices $i,j,\tR$ which maximize the rate from Corollary~\ref{cor:isaega}.} .

\begin{remark}
In special case when $R_\tR = \NRt$ always, {\tt ISAEGA} becomes {\tt ISEGA} from~\cite{mishchenko201999}. However~\cite{mishchenko201999} assumes that $|\NRt|$ is constant in $\tR$ and $L_{\tR} = \eLi$ with probability $\frac1d$. Thus, even special case of Corollary~\ref{cor:isaega} generalizes results on {\tt ISEGA} from~\cite{mishchenko201999}. For completeness, we state {\tt ISEGA} as Algorithm~\ref{alg:isega} and Corollary~\ref{cor:isega} provides its iteration complexity. 
\end{remark}

\begin{algorithm}[h]
  \caption{{\tt ISEGA} ({\tt ISEGA}~\cite{mishchenko201999} with arbitrary sampling) {\bf [NEW METHOD]}}
  \label{alg:isega}
\begin{algorithmic}
\State{\bfseries Input: }{$x^0\in\R^d$, \# parallel units $\TR$, each owning set of indices $N_\tR$ (for $1\leq \tR\leq \TR$), distributions $\cDL_{\tR}$ over subsets coordinates $[d]$, stepsize $\alpha$ }
\State $\mJ^0  = 0$
  \For{$k=0,1,\dotsc$}
    \For{$\tR=1,\dotsc,\TR$ in parallel}
        \State Sample $ L_\tR \sim \cDL_\tR$; $L_\tR\subseteq [d]$  (independently on each machine)
        \State Observe $\nabla_{L_\tR} f_{j}(x^k)$ for $j \in \NRt $ 
        \State For $i\in [d], j\in \NRt$ set $\mJ^{k+1}_{i,j} = \begin{cases}
        \nabla_{i} f_{j}(x^k) & \text{if} \quad i\in [d], j\in \NRt , i\in L_{\tR} \\
        \mJ^{k}_{i,j} & \text{otherwise}
        \end{cases} $
        \State Send $\mJ^{k+1}_{:\NRt }- \mJ^k_{:\NRt }$ to master \Comment{Sparse; low communication}
    \EndFor
    \State $g^k = \left(\mJ^k + \sum_{\tR=1}^\TR   \left( {{\ptL}^{-1}} {\eR}^\top\right) \circ \left( \left( \sum_{i\in L_\tR} \eLi \eLi^\top \right)\left(\mJ^{k+1}-\mJ^k   \right)_{:\NRt} \right) \right)\eR$
    \State $x^{k+1} =\prox(x^k - \alpha g^k)$
  \EndFor
\end{algorithmic}
\end{algorithm}

 \begin{corollary}\label{cor:isega} 
 Let $\alpha =  \min_{j\in \NRt, i,\tR} \frac{1}{4 m^j_i \left( 1+ \frac{1}{n\ptLi |\NRt |}\right) + \frac{\sigma}{\ptLi }|\NRt |}$. Then, iteration complexity of Algorithm~\ref{alg:isaega} is $  \max_{j\in \NRt, i,\tR}  \left( 4\frac{m^j_i}{\sigma}\left( 1+ \frac{1}{ n\ptLi|\NRt |} \right)  +  \frac{1}{\ptLi|\NRt |} \right)
\log\frac1\epsilon$. 
\end{corollary}

\section{Special Cases: {\tt JacSketch} \label{sec:jacsketch}}
As next special case of {\tt GJS} (Algorithm~\ref{alg:SketchJac}) we present {\tt JacSketch} ({\tt JS}) motivated by~\cite{gower2018stochastic}. The algorithm observes every iteration a single right sketch of the Jacobian and constructs operators $\cS, \cU$ in the following fashion:
\[
  \cS \mX = \mX \mR 
\quad \text{and} \quad   \cU \mX = \mX \mR \E{\mR}^{-1}
 \]
where $\mR\in \R^{n\times n}$ is random projection matrix.

\begin{algorithm}[!h]
\begin{algorithmic}[1]
\State \textbf{Parameters:} Stepsize $\alpha>0$, Distribution $\cD$ over random projector matrices $\mR\in \R^{n\times n}$
\State \textbf{Initialization:} Choose solution estimate $x^0 \in \R^d$ and Jacobian estimate $ \mJ^0\in \R^{d\times n}$ 
\For{$k =  0, 1, \dots$}
\State Sample realization of $\mR\sim \cD$ perform sketches $\mG(x^k)\mR$ 
\State  $\mJ^{k+1} = \mJ^k - (\mJ^k - \mG(x^k)\mR)$ 
\State $g^k = \frac1n \mJ^k \eR + \frac1n  \left(\mG(x^k) -\mJ^k\right)\mR \E{\mR}^{-1}\eR$  
    \State $x^{k+1} = \prox (x^k - \alpha g^k)$ 
\EndFor
\end{algorithmic}
\caption{{\tt JS} (JacSketch)}
\label{alg:jacsketch}
\end{algorithm}

Note that Algorithm~\ref{alg:jacsketch} differs to what was proposed in~\cite{gower2018stochastic} in the following points. 
\begin{itemize}
\item Approach from~\cite{gower2018stochastic} uses a scalar random variable $\theta_\mR $ to set $\cU \mX = \theta_\mR \mX \mR$. Instead, we set $\E{\cU}= \mX \mR \E{\mR}^{-1}$. This tweak allows Algorithm~\ref{alg:SketchJac} to recover the tightest known analysis of {\tt SAGA} as a special case. Note that the approach from~\cite{gower2018stochastic} only recovers tight rates for {\tt SAGA} under uniform sampling.
\item Unlike~\cite{gower2018stochastic}, our setup allows for proximable regularizer, thus is more general.
\item Approach from~\cite{gower2018stochastic} allows projections under a general weighted norm. Algorithm~\ref{alg:SketchJac} only allows for non-weighted norm; which is only done for the sake of simplicity as the paper is already  notation-heavy. However, {\tt GJS} (Algorithm~\ref{alg:SketchJac}) is general enough to alow for an arbitrary weighted norm. 
\end{itemize} 

The next corollary shows the convergence result. 
\begin{corollary}[Convergence rate of {\tt JacSketch}]\label{cor:jacsketh} Suppose that operator $\cM$ is commutative with right multiplication by $\mR$ always. Consider any $\mB\in \R^{n\times n}$ which commutes with $\mR$ always.
Denote
\[
\mM^{\frac12} \eqdef  \begin{pmatrix}
\mM_1^{\frac12} & & \\
& \ddots &\\
& & \mM_n^{\frac12}
\end{pmatrix} \quad \text{and} \quad \ugly \eqdef 
\lambda_{\max}\left( 
{\mM^{\frac12}}^\top 
 \left(\E{ \mR \E{\mR}^{-1} \eR \eR^\top \E{\mR}^{-1} \mR} \otimes \mI_{d}\right)
\mM^{\frac12}
 \right).
\]
  Let 
  \[
\alpha =\frac{\lambda_{\min} \left(\mB^\top \E{\mR} \mB \right) }{ 4n^{-1}\ugly \lambda_{\max}\left(  \mB^\top \E{\mR} \mB  \right) + \sigma\lambda_{\max} \left(\mB^\top \mB\right) }.
\]
Then, iteration complexity of Algorithm~\ref{alg:jacsketch} is 

\[
 \frac{ 4n^{-1}\ugly \sigma^{-1} \lambda_{\max}\left(  \mB^\top \E{\mR} \mB  \right) + \lambda_{\max} \left(\mB^\top \mB\right) }{\lambda_{\min} \left(\mB^\top \E{\mR} \mB \right) } \log\frac1\epsilon.
 \]

\end{corollary}
\clearpage

\clearpage
 \section{Special Cases:  Proofs}

In this section, we provide the proofs of all corollaries listed in previous sections.

 For simplicity, we will use the following notation throughout this section: $\Gamma(\mX)=  \cU(\mX)e$.
 
\subsection{SAGA methods: Proofs}
\subsubsection{Setup for Corollary~\ref{cor:saga}}
Note first that the choice of $\cS, \cU$ yields
\begin{eqnarray*}
\E{\cS(\mX)} &=& \frac1n \mX \qquad \\
\E{\|\Gamma(\mX) \|^2}  &=& n^2  \E{\< \mX^\top,  \eRj \eRj^\top \eR \eR^\top \eRj \eRj^\top \mX^\top> }  = n\|\mX\|^2.
\end{eqnarray*}

Next, as we have no prior knowledge about $\mG(x^*)$, let $\cR\equiv \cI$; i.e. $\Range{ \cR} = \R^{d\times n}$. Lastly, consider $\cB$ operator to be a multiplication with constant $\beta$: $\cB (\mX)  = \beta \mX$.

Thus for \eqref{eq:small_step} we should have

\[
\frac{2\alpha}{n}m + \beta^2 \left(1-\frac1n\right) \leq (1-\alpha \sigma) \beta^2
\]
and for~\eqref{eq:small_step2} we should have
\[
\frac{2\alpha}{n} m + \frac{\beta^2}{n} \leq \frac{1}{n}.
\]

It remains to notice that choices $\alpha = \frac{1}{4m + \sigma n}$ and $\beta^2 = \frac{1}{2}$ are valid to satisfy the above bounds.

\subsubsection{Setup for Corollary~\ref{cor:saga_as2} \label{sec:cor:saga_as2}}
First note that $\E{\cS (\mX)} = \mX \diag (\pR)$.

Next, due to~\eqref{eq:j_in_range},~\eqref{eq:g_in_range}, inequalities~\eqref{eq:small_step} and \eqref{eq:small_step2} with choice $\mY=  {\cM^\dagger}^{\frac12}\mX$ become respectively:
\begin{equation}\label{eq:linear_1}
\frac{2\alpha}{n^2} \E{\left \|  
 \sum_{j\in R} \pRj^{-1}
\mM_j^{\frac12} \mY_{:j}\right \|^2} + 
\left\|\left(\cI-\E{\cS}\right)^{\frac{1}{2}}\cB(\mY)\right\|^{2} \leq(1-\alpha \sigma)\|\cB(\mY)\|^{2}
\end{equation}

\begin{equation}\label{eq:linear_2}
\frac{2\alpha}{n^2}  \E{\left \|  \sum_{j\in R} \pRj^{-1}
\mM_i^{\frac12} \mY_{:i}\right \|^2}  + 
\left\|\left(\E{\cS}\right)^{\frac{1}{2}}\cB(\mY)\right\|^{2} \leq \frac1n  \|\mY \|^2
\end{equation}

Note that 
\[
\E{\left \|  
 \sum_{j\in R} \pRj^{-1}
\mM_j^{\frac12} \mY_{:j}\right \|^2}  =
  \E{   \left \| \sum_{j\in R} 
\mM_j^{\frac12}(\pRj^{-1} \mY_{:j}) \right \|^2}  
\leq 
\sum_{j=1}^n \pRj^{-1}v_j \|\mY_{:j}\|^2
\]

where we used ESO assumption~\eqref{eq:ESO_saga} in the last bound above.

Next choose $\cB$ to be right multiplication with $\diag(b)$. Thus, for~\eqref{eq:linear_1} it suffices to have for all $j\in [n]$
\[
\frac{2\alpha}{n^2} v_j \pRj^{-1} + b_j^2(1-\pRj) \leq b_j^2 (1-\alpha \sigma) \qquad \Rightarrow \qquad
\frac{2\alpha}{n^2} v_j \pRj^{-1} + b_j^2\alpha \sigma \leq b_j^2\pRj 
\]
For~\eqref{eq:linear_2} it suffices to have for all $j\in [n]$
\[
\frac{2\alpha}{n^2} v_j \pRj^{-1} + b_j^2 \pRj \leq \frac{1}{n}
\]
It remains to notice that choice $b_j^2 = \frac{1}{2n\pRj}$ and $\alpha  = \min_j \frac{n \pRj}{4v_j + n\sigma}$ is valid.

\subsection{SEGA methods: Proofs}

\subsubsection{Setup for Corollary~\ref{cor:sega}}

 Note that 
\begin{eqnarray*}
\E{\cS x} &=& \frac{1}{d} x\\
\E{\|\Gamma(x) \|^2}  &=& d^2  \E{\< x,  \eLi \eLi^\top \eLi \eLi^\top x> }  = d\|x\|^2.
\end{eqnarray*}

Next, choose operator $\cB$  to be constant; in particular $\cB x = \beta x$. Thus to satisfy~\eqref{eq:small_step} it suffices to have
\[
2\alpha d m  +  \beta^2\left(1-\frac1d \right  ) \leq \beta^2(1-\alpha \sigma ) \qquad 
\Rightarrow \qquad 
2\alpha d m + \alpha \sigma \beta^2 \leq \frac{\beta^2}{d}.
\]
To satisfy~\eqref{eq:small_step2}, it suffices to have
\[
2\alpha d  m+ \frac{\beta^2}{d} \leq 1.
\]
It remains to notice that $\beta^2 = \frac{d}{2}$ and $\alpha= \frac{1}{4md+ \sigma d}$ satisfies the above conditions.

\subsubsection{Setup for Corollary~\ref{cor:sega_is_11}}
Note that 
\begin{eqnarray*}
\E{\cS (x)} &=& \diag(\pL) x
\end{eqnarray*}
and
\begin{equation*}
\E{\left\|\Gamma(x)\right\|^{2}} = \| x\|^2_{ \E{\sum_{i\in L} \frac{1}{\pLi} \eLi \eLi^\top \sum_{i\in L} \frac{1}{\pLi} \eLi \eLi^\top } }   =  \| x\|^2_{\pL^{-1}}.
\end{equation*}

Let us consider $\cB$ to be the operator corresponding to left multiplication with matrix $\diag(b)$: $\cB(x) = \diag(b)x$. Thus, for~\eqref{eq:small_step} it suffices to have for all $i$
\[
2\alpha m_i \pLi^{-1} + b_i^2(1-\pLi) \leq b_i^2 (1-\alpha \sigma) \qquad \Rightarrow \qquad
2\alpha m_i \pLi^{-1} + b_i^2\alpha \sigma \leq b_i^2\pLi
\]
For~\eqref{eq:small_step2} it suffices to have for all $i$
\[
2\alpha m_i \pLi^{-1} + b_i^2 \pLi \leq 1 
\]
It remains to notice that choice $b_i^2 = \frac{1}{2\pLi}$ and $\alpha  = \min_i \frac{\pLi}{4m_i+ \sigma}$ is valid.

\subsubsection{Setup for Corollary~\ref{cor:svrcd}}

Note that 
\begin{eqnarray*}
\E{\cS (x)} &=& \probx x
\end{eqnarray*}

and
\begin{equation*}
\E{\left\|\Gamma(x)\right\|^{2}} = \| x\|^2_{ \E{\sum_{i\in L} \frac{1}{\pLi} \eLi \eLi^\top \sum_{i\in L} \frac{1}{\pLi} \eLi \eLi^\top } }   =  \| x\|^2_{\pL^{-1}}.
\end{equation*}

Let us consider $\cB$ to be the operator corresponding to scalar multiplication with $\beta$. Thus, for~\eqref{eq:small_step} it suffices to have for all $i$
\[
2\alpha w_i \pLi^{-1} + \beta^2(1-\probx) \leq \beta^2 (1-\alpha \sigma) \qquad \Rightarrow \qquad
2\alpha w_i  \pLi^{-1} +\beta^2\alpha \sigma \leq\beta^2\probx.
\]
For~\eqref{eq:small_step2} it suffices to have for all $i$
\[
2\alpha w_i  \pLi^{-1} + \beta^2 \probx \leq 1 .
\]
It remains to notice that choice $\beta^2 = \frac{1}{2\probx}$ and $\alpha  = \min_i \frac{1}{4w_i\pLi^{-1}+ \sigma\probx^{-1}}$ is valid.

\subsection{Setup for Corollary~\ref{cor:sgd}}
Choose $\cB$ to be operator which maps everything into 0. On top of that, by construction we have $\cR = 0$ and thus~\eqref{eq:small_step} is satisfied for free. Moreover, from~\eqref{eq:ESO_saga} we have (following the steps from Section~\ref{sec:cor:saga_as2}):
\[
\E{\|\Gamma(\cM^{\frac12}(\mX))\|^2}\leq \sum_{j=1}^n p_j^{-1}v_j \| \mX_{:j} \|^2.
\]
Further, due to~\eqref{eq:g_in_range} and~\eqref{eq:j_in_range}, to satisfy \eqref{eq:small_step2} we shall have 
\[
\frac{2\alpha}{n^2} \sum_{j=1}^n p_j^{-1}v_j \| \mY_{:j} \|^2 \leq \frac1n \| \mY\|^2
\]
which simplifies to
\[
\frac{2\alpha}{n} \frac{v_j}{p_j}\leq 1
\]
and thus it suffices to choose $\alpha  =  \frac{n}{2} \min_j \frac{p_j}{v_j} $.

\begin{remark}
Factor 2 can be omitted since for Lemma~\ref{lem:g_lemma}, the second factor is 0 and thus we no longer need the Jensen's inequality.
\end{remark}

\subsection{Setup for Corollary~\ref{cor:lsvrg_as}}

First note that $\E{\cS (\mX)} = \probx \mX $.

Next, due to~\eqref{eq:j_in_range},~\eqref{eq:g_in_range}, inequalities~\eqref{eq:small_step} and \eqref{eq:small_step2} with choice $\mY=  {\cM^\dagger}^{\frac12}\mX$ become respectively:
\begin{equation}\label{eq:linear_1_svrg}
\frac{2\alpha}{n^2} \E{\left \|  
 \sum_{j\in R} \pRj^{-1}
\mM_j^{\frac12} \mY_{:j}\right \|^2} + 
\left\|\left(\cI-\E{\cS}\right)^{\frac{1}{2}}\cB(\mY)\right\|^{2} \leq(1-\alpha \sigma)\|\cB(\mY)\|^{2}
\end{equation}

\begin{equation}\label{eq:linear_2_svrg}
\frac{2\alpha}{n^2}  \E{\left \|  \sum_{j\in R} \pRj^{-1}
\mM_i^{\frac12} \mY_{:i}\right \|^2}  + 
\left\|\left(\E{\cS}\right)^{\frac{1}{2}}\cB(\mY)\right\|^{2} \leq \frac1n  \|\mY \|^2
\end{equation}

Note next that
\[
\E{\left \|  
 \sum_{j\in R} \pRj^{-1}
\mM_j^{\frac12} \mY_{:j}\right \|^2}  =
  \E{   \left \| \sum_{j\in R} 
\mM_j^{\frac12}(\pRj^{-1} \mY_{:j}) \right \|^2}  
\leq 
\sum_{j=1}^n \pRj^{-1}v_j \|\mY_{:j}\|^2
\]

where we used ESO assumption~\eqref{eq:ESO_saga} in the last bound above.

Next choose $\cB$ to be multiplication with scalar $\beta$. Thus, for~\eqref{eq:linear_1_svrg} it suffices to have for all $j\in [n]$
\[
\frac{2\alpha}{n^2} v_j \pRj^{-1} + \beta^2(1-\probx) \leq \beta^2 (1-\alpha \sigma) \qquad \Rightarrow \qquad
\frac{2\alpha}{n^2} v_j \pRj^{-1} + \beta^2\alpha \sigma \leq \beta^2\probx
\]
For~\eqref{eq:linear_2_svrg} it suffices to have for all $j\in [n]$
\[
\frac{2\alpha}{n^2} v_j \pRj^{-1} + \beta^2 \probx \leq \frac{1}{n}
\]
It remains to notice that choice $\beta^2 = \frac{1}{2n\probx}$ and $\alpha  = \min_j \frac{n }{4v_j\pRj^{-1} + n\sigma \probx^{-1}}$ is valid.

\subsection{Methods with Bernoulli $\cU$: Proofs}

\subsubsection{Setup for Corollary~\ref{cor:B2}}

Note first that the choice of $\cS, \cU$ yield
\begin{eqnarray*}
\E{\cS(x)} &=& \probx x  \\
\E{\|\Gamma(x) \|^2}  &=& \E{\| \cU x\|^2 }  = \proby^{-1}\|x \|^2
\end{eqnarray*}

Next, consider $\cB$ operator to be a multiplication with a constant $ b$.

Thus for \eqref{eq:small_step} we should have

\[
2\alpha\proby^{-1}L + b^2 \left(1-\probx\right) \leq (1-\alpha \sigma) b^2
\]
and for~\eqref{eq:small_step2} we should have
\[
2\alpha \proby^{-1} L + \probx b^2 \leq 1
\]

It remains to notice that choices $\alpha = \frac{1}{4\proby^{-1}L + \sigma \probx^{-1}}$ and $b^2 = \frac{1}{2 \probx}$ are valid to satisfy the above bounds.

\subsubsection{Setup for Corollary~\ref{cor:inverse_svrg}}

Note first that the choice of $\cS, \cU$ yields
\begin{eqnarray*}
\E{\cS(\mX)} &=&  \mX \diag(\pR) \\
\E{\|\Gamma(\mX) \|^2}  &=& \E{\| \cU(\mX)e\|^2 }  = \proby^{-1}\|\mX e\|^2 \leq  \proby^{-1}n \|\mX \|^2
\end{eqnarray*}

Next, as we have no prior knowledge about $\mG(x^*)$, consider $\cR$ to be identity operator; i.e. $\Range{ \cR} = \R^{d\times n}$. Lastly, consider $\cB$ operator to be a right multiplication with $ \diag(b)$.

Thus for \eqref{eq:small_step} we should have

\[
\forall j:  \quad 
\frac{2\alpha}{n}\proby^{-1}m +\alpha \sigma b^2_j  \leq b^2_j \pRj
\]
and for~\eqref{eq:small_step2} we should have
\[
\forall j: \quad  \frac{2\alpha}{n} \proby^{-1} m + \pRj b^2_j \leq \frac{1}{n}
\]

It remains to notice that choices $\alpha = \min_j \frac{1}{4\proby^{-1}m + \sigma \pRj^{-1}}$ and $b^2_j = \frac{1}{2n \pRj}$ are valid to satisfy the above bounds.

\subsubsection{Setup for Corollary~\ref{cor:SVRCD-inv}}

Note first that the choice of $\cS, \cU$ yields
\begin{eqnarray*}
\E{\cS(x)} &=& \pL \circ x  \\
\E{\|\Gamma(x) \|^2}  &=& \E{\| \cU(x)\|^2 }  = \proby^{-1}\|x \|^2 
\end{eqnarray*}

Next, as we have no prior knowledge about $\mG(x^*)$, consider $\cR$ to be identity operator; i.e. $\Range{ \cR} = \R^{d\times n}$. Lastly, consider $\cB$ operator to be left multiplication with matrix $\diag( b)$.

Thus for \eqref{eq:small_step} we should have

\[
\forall i: \quad 
2\alpha\proby^{-1}m + b^2_i \alpha \sigma \leq b^2_i \pLi
\]
and for~\eqref{eq:small_step2} we should have
\[
2\alpha\proby^{-1}m + \pLi b^2_i \leq 1
\]

It remains to notice that choices $\alpha = \min_i \frac{1}{4\proby^{-1}m + \sigma \pLi^{-1}}$ and $b^2_i= \frac{1}{2 \pLi^{-1}}$ are valid to satisfy the above bounds.

\subsection{Combination of left and right sketches (in different operators): Proofs}

\subsubsection{Setup for Corollary~\ref{cor:RL}}
Note first that the choice of $\cS, \cU$ yields

\begin{eqnarray*}
\E{\cS(\mX)} &=&\mX \diag(\pR),   \\
\E{\|\piop(\mX) \|^2}  &=& \| \cM^{\frac12}(\mX)\eR \|^2_{\diag({\pL}^{-1})} \leq n\sum_{j=1}^n \| \mM_j \mX_{:j}\|^2_{\diag({\pL}^{-1})} = n \sum_{j=1}^n \|  \mX_{:j}\|^2_{ \diag(m^j\circ {\pL}^{-1})}.
\end{eqnarray*}

Let $\cB$ be right multiplication by $\diag(b)$.
Thus for \eqref{eq:small_step} we should have

\[
\forall i,j: \quad 2\frac{\alpha}{n}{m_i}^j {\pLi}^{-1} + b^2_j \alpha \sigma \leq b^2_j \pRj
\]
and for~\eqref{eq:small_step2} we should have
\[
\forall i,j: \quad 2\frac{\alpha}{n}{m_i}^j {\pLi}^{-1}+ \pRj b^2_j \leq \frac1n.
\]

It remains to notice that choices $\alpha = \min_{i,j} \frac{1}{4{m_i}^j {\pLi}^{-1} + \sigma \pRj^{-1}}$ and $b^2_j=\frac{1}{ 2\pRj n}$ are valid to satisfy the above bounds. 

\subsubsection{Setup for Corollary~\ref{cor:LR}}
Note first that the choice of $\cS, \cU$ yields

\begin{eqnarray*}
\E{\cS(\mX)} &=&    \diag(\pL)  \mX  \\
\E{\|\Gamma(\mX) \|^2}  &\leq & \sum_{j=1}^n \pRj^{-1} v_j \|\mX_{:j}\|^2
\end{eqnarray*}
The second inequality is a direct consequence of ESO (which is shown is Section~\ref{sec:cor:saga_as2}).

Let $\cB$ be left multiplication by $\diag(b)$. Thus for \eqref{eq:small_step} we should have

\[
\forall i,j: \quad 2\frac{\alpha}{n}v_j {\pRj}^{-1} + b^2_i \alpha \sigma \leq b^2_i \pLi
\]
and for~\eqref{eq:small_step2} we should have
\[
\forall i,j: \quad  2\frac{\alpha}{n}v_j{\pRj}^{-1} + \pLi b^2_i \leq \frac1n
\]

It remains to notice that choices $\alpha = \min_{i,j} \frac{1}{4{v_j} {\pRj}^{-1} + \sigma \pLi^{-1}}$ and $b^2_i = \frac{1}{ 2\pLi n}$ are valid to satisfy the above bounds.

\subsection{Joint Sketches: Proofs}

\subsubsection{Setup for Corollary~\ref{cor:saega} \label{sec:corsaega}}

For notational simplicity, denote ${\mM'}^{\frac12}\in \R^{d\times n}$ to be the matrix with $j$-th column equal to (elementwise) square root of $m_j$. We have 
\[
\E{\cS(\mX)}  = \left( {\pL}{\pR}^\top\right)\circ \mX
\]
and 

\begin{eqnarray} \nonumber
  \E{\|  \piop(\cM^{\frac12 }\mX)\|^2} &   = &
 \E{ \left\| \left( \left( {\pL}^{-1}{\pR^{-1}}^\top\right) \circ \left( \left( \sum_{i\in L} \eLi \eLi^\top \right) (\mM^{'\frac12}\circ\mX) \left( \sum_{j\in R} \eRj \eRj^\top \right) \right) \right)\eR \right\|^2} 
 \\ \nonumber
 &   = & 
  \E{ \left\|   \left( \left( \sum_{i\in L, j\in R} \eLi \eRj^\top \right) \circ\left( {\pL}^{-1}{\pR^{-1}}^\top\right) \circ \mM^{'\frac12}\circ\mX \right)\eR \right\|^2}
   \\ \nonumber
    &   = & 
  \E{ \left\|  \left( \sum_{i\in L} \eLi \eLi^\top \right)  \left( \left( {\pL}^{-1}{\pR^{-1}}^\top\right) \circ \mM^{'\frac12}\circ\mX\right)\eRR \right\|^2}
   \\ \nonumber
       &   = & 
  \EE{R}{ \left\|  \left(     \left( {\pL}^{-\frac12}{\pR^{-1}}^\top\right) \circ \mM^{'\frac12}\circ\mX \right)\eRR \right\|^2}
   \\ \nonumber
          &   = & 
  \EE{R}{  \trace\left(   \left(     \left( {\pL}^{-\frac12}{\pR^{-1}}^\top\right) \circ \mM^{'\frac12}\circ\mX \right)\mI_{R,R } \left(     \left( {\pL}^{-\frac12}{\pR^{-1}}^\top\right)^\top \circ {\mM^{'\frac12}}^\top\circ\mX^\top \right) \right)}
   \\ \nonumber
             &   = & 
\trace\left(   \left(     \left( {\pL}^{-\frac12}{\pR^{-1}}^\top\right) \circ \mM^{'\frac12}\circ\mX \right)\PR \left(     \left( {\pL}^{-\frac12}{\pR^{-1}}^\top\right)^\top \circ  {\mM^{'\frac12}}^\top\circ\mX^\top \right) \right)
   \\ \nonumber
                &   = & 
\trace\left(   \left(     \left( {\pL}^{-\frac12}{\eR}^\top\right) \circ \mM^{'\frac12}\circ\mX \right) \diag(\pR)^{-1}\PR \diag(\pR)^{-1} \left(     \left( {\pL}^{-\frac12}{\eR}^\top\right)^\top \circ {\mM^{'\frac12}}^\top\circ\mX^\top \right) \right)
   \\ \nonumber
                   &   \leq & 
\trace\left(   \left(     \left( {\pL}^{-\frac12}{\eR}^\top\right) \circ \mM^{'\frac12}\circ\mX \right)  \diag(\qR)^{-1} \left(     \left( {\pL}^{-\frac12}{\eR}^\top\right)^\top \circ {\mM^{'\frac12}}^\top\circ\mX^\top \right) \right)
   \\
                      &   = & 
\left\| \mX \circ \mM^{'\frac12} \circ  \left( {\pL}^{-\frac12}{\qR^{-\frac12}}^\top \right) \right\|^2.
 \label{eq:SAEGA_ESO}
\end{eqnarray}

Next, choose operator $\mB$ to be such that $\cB(\mX) \eqdef \mB \circ \mX$ for $\mB\in \R^{d\times n}$. 
Thus, for~\eqref{eq:small_step} and \eqref{eq:small_step2} we shall have respectively
\[
\forall i, j: \quad 
\frac{2\alpha}{n^2}\left( \frac{m^j_i}{\pLi \qRj}\right)+ \mB_{i,j}^2\alpha \sigma \leq \mB_{ij}^2 \pLi \qRj
\]
and
\[
\forall i, j: \quad 
\frac{2\alpha}{n^2}\left( \frac{m^j_i}{\pLi \qRj}\right)+ \mB_{ij}^2 \pLi \qRj \leq \frac{1}{n}.
\]

It remains to choose $\mB_{i,j}^2= \frac{1}{2n \pLi \qRj}$ and $\alpha = \min_{i,j} \frac{n\pLi \qRj}{4m^j_i + n\sigma}$.

\subsubsection{Setup for Corollary~\ref{cor:svrcdg}}

We have 
\[
\E{\cS(\mX)}  =\probx\mX.
\]

Next, choose operator $\mB$ to be such that $\cB(\mX) \eqdef \beta \circ \mX$ for scalar $\beta$ which would be specified soon. 
Proceeding with bound~\eqref{eq:SAEGA_ESO}, for~\eqref{eq:small_step} and \eqref{eq:small_step2} we shall have respectively
\[
\forall i, j: \quad 
\frac{2\alpha}{n^2}\left( \frac{m^j_i}{\pLi \qRj}\right)+ \beta^2\alpha \sigma \leq \beta^2 \probx
\]
and
\[
\forall i, j: \quad 
\frac{2\alpha}{n^2}\left( \frac{m^j_i}{\pLi \qRj}\right)+ \beta^2 \probx \leq \frac{1}{n}.
\]

It remains to choose $\beta^2= \frac{1}{2n \probx}$ and $\alpha = \min_{i,j} \frac{1}{4\frac{m^j_i}{n\pLi \qRj} + \probx^{-1}\sigma}$.

\subsubsection{Setup for Corollary~\ref{cor:isaega}}

For notational simplicity, denote $\mM'\in \R^{d\times n}$ to be a matrix with $j$-th column equal to $m_j$.

Let $  \piop_\tR(\mX_{:\NRt}) =  \left( {{\ptL}^{-1}} {{\ptR}^{-1}}^\top\right) \circ \left( \left( \sum_{i\in L_\tR} \eLi \eLi^\top \right)\mX_{:\NRt} \left( \sum_{j\in R_\tR} \eRj \eRj^\top \right)\right) e_{\NRt}$. 
Thus 
\[
\E{\cS(\mX)}  = 
\sum_{\tR=1}^\TR
 \left( {{\ptL}} {{\ptR}}^\top\right) \circ\mX_{:\NRt}  
\]
and 
\begin{eqnarray}
\nonumber
\E{\|  \piop(\mX)\|^2} &=& \E{ \left\|   \sum_{\tR=1}^\TR  \piop_\tR(\mX_{:\NRt}) \right \|^2 } 
\\ \nonumber
&=&
\E{ \left\| \sum_{\tR=1}^\TR  \piop_\tR(\mX_{:\NRt})  -  \E{\sum_{\tR=1}^\TR  \piop_\tR(\mX_{:\NRt})  } \right \|^2 } 
+ 
 \left\|   \E{\sum_{\tR=1}^\TR  \piop_\tR(\mX_{:\NRt})  } \right \|^2 
 \\ \nonumber
 &=&
 \E{ \left\| \sum_{\tR=1}^\TR\left(    \piop_\tR(\mX_{:\NRt}) - \mX_{:\NRt}e_{\NRt}  \right)\right \|^2 } 
+ 
 \left\|   \mX \eR \right \|^2 
 \\ \nonumber
  &=&
 \sum_{\tR=1}^\TR  \E{ \left\|     \piop_\tR(\mX_{:\NRt}) - \mX_{:\NRt}e_{\NRt}  \right \|^2 } 
+ 
 \left\|     \mX \eR \right \|^2 
 \\ \nonumber
   &\leq&
 \sum_{\tR=1}^\TR\E{ \left\|   \piop_\tR(\mX_{:\NRt})\right \|^2 } 
+ 
 \left\|    \mX \eR \right \|^2 
 \\ \label{eq:ISEAGA_ESO}
    &\leq&
 \sum_{\tR=1}^\TR \E{ \left\|   \piop_\tR(\mX_{:\NRt}) \right \|^2 } 
+ 
 n \left\|  \mX \right \|^2.
\end{eqnarray}

Using the bounds from Section~\ref{sec:corsaega} we further get
\begin{eqnarray*}
  \E{\|  \piop(\cM^{\frac12 }\mX)\|^2} 
  &  \stackrel{\eqref{eq:ISEAGA_ESO}+\eqref{eq:SAEGA_ESO}}{\leq}&
 \sum_{\tR=1}^\TR   \left\|   \mX_{:\NRt} \circ \left( {\ptL}^{-\frac12} {{\qtR}^{-\frac12}}^\top \right) \circ \mM'_{:\NRt} \right \|^2 
+ 
 n \left\| \mM' \circ \mX \right \|^2. \\
\end{eqnarray*}

Next, choose operator $\mB$ to be such that for any $\mX$: $\cB(\mX) \eqdef \mB \circ \mX$ where $\cB \in \R^{d\times n}$. 
Thus, for~\eqref{eq:small_step} and \eqref{eq:small_step2} we shall have respectively
\[
\forall i, \tR, j\in \NRt: \quad 
\frac{2\alpha}{n^2}m^j_i \left( \frac{1}{ \ptLi \qtRj}+ n\right)+ \mB_{i,j}^2\alpha \sigma \leq \mB_{i,j}^2 \ptLi \qtRj
\]
and
\[
\forall i, \tR, j\in \NRt: \quad 
\frac{2\alpha}{n^2}m^j_i \left( \frac{1}{ \ptLi \qtRj}+ n\right)+ \mB_{i,j}^2\ptLi \qtRj \leq \frac{1}{n}.
\]

It remains to choose $\mB_{i,j}^2= \frac{1}{2n \ptL \qtRj }$ and $\alpha = \min_{j\in \NRt, i,\tR} \frac{1}{4 m^j_i \left( 1+ \frac{1}{n\ptL \qtRj}\right) + \frac{\sigma}{\ptL \qtRj}}$.

\subsection{Setup for Corollary~\ref{cor:jacsketh}}
Let $x$ be column-wise vectorization of $\mX$. 
Note that 
\[
\Gamma(\cM^\frac12(\mX)) = \cM^\frac12(\mX) \mR\E{\mR}^{-1}\eR = \left( \eR^\top \E{\mR}^{-1} \mR \otimes \mI_{d}\right)\begin{pmatrix}
\mM_1^{\frac12} & & \\
& \ddots &\\
& & \mM_n^{\frac12}
\end{pmatrix}x. 
\]
Thus 

\[\E{\left\| \Gamma(\cM^\frac12(\mX))\right\|^2 } \leq 
\|\mX\|^2 \ugly.\]
Let $\cB(\mX) = \beta \mX \mB$. Thus, we have 
\begin{eqnarray*}
  (1-\alpha \sigma) \NORMG{ \cB\mX } - \NORMG{  \left(\cI - \E{\cS} \right)^{\frac12}\cB  \mY} 
 &=& \beta^2 \trace\left(\mX \mB^\top ( \E{\mR} - \alpha\sigma\mI) \mB\mX^\top \right) \\
 &\leq &\beta^2 \lambda_{\min} \left(\mB^\top ( \E{\mR} - \alpha\sigma\mI) \mB \right) \|\mX\|^2 \\
  &\leq &\beta^2 \left( \lambda_{\min} \left(\mB^\top \E{\mR} \mB \right) -  \alpha\sigma\lambda_{\max} \left(\mB^\top \mB\right)\right) \|\mX\|^2.
\end{eqnarray*}

Further, 
\[
\NORMG{\left(\E{\cS}\right)^{\frac12}  \cB  {\cM^\dagger}^{\frac12}\mX  }   =  \beta^2\trace \left( \mX \mB^\top \E{\mR} \mB \mX^\top\right) \leq  \beta^2\|\mX\|^2 \lambda_{\max}\left(  \mB^\top \E{\mR} \mB  \right).
\]

Using the derived bounds together with~\eqref{eq:j_in_range},~\eqref{eq:g_in_range}, for conditions~\eqref{eq:small_step} and \eqref{eq:small_step2} it suffices to have:
\begin{equation}\label{eq:linear_1_js}
\frac{2\alpha}{n^2}\ugly + \beta^2\alpha\sigma\lambda_{\max} \left(\mB^\top \mB\right)
  \leq 
\beta^2\lambda_{\min} \left(\mB^\top \E{\mR} \mB \right)  ,
\end{equation}
and
\begin{equation}\label{eq:linear_2_js}
\frac{2\alpha}{n^2}\ugly + \beta^2 \lambda_{\max}\left(  \mB^\top \E{\mR} \mB  \right) \leq \frac1n.
\end{equation}
It remains to notice that choices $\beta^2 = \frac{1}{2n \lambda_{\max}\left(  \mB^\top \E{\mR} \mB  \right)}$ and 
\[
\alpha =\frac{\lambda_{\min} \left(\mB^\top \E{\mR} \mB \right) }{ 4n^{-1}\ugly \lambda_{\max}\left(  \mB^\top \E{\mR} \mB  \right) + \sigma\lambda_{\max} \left(\mB^\top \mB\right) }
\]
are valid.

\clearpage

\section{Convergence Under Strong Growth Condition \label{sec:sg}}

In this section, we extend the result of Algorithm~\ref{alg:SketchJac} to the case when $F\eqdef f+\psi$ satisfies a strong growth condition instead of quasi strong convexity. Note that strong growth is weaker (more general) than quasi strong convexity~\cite{karimi2016linear}.

Suppose that $\cX^*$ is a set of minimizers of convex function $F$. Clearly, $\cX^*$ must be convex. Define $[x]^*$ to be a projection of $x$ onto $\cX^*$.
\begin{assumption}\label{as:strong_growth}

Suppose that $F$ satisfies strong growth, i.e. for every $x$: 
\begin{equation} \label{eq:strong_growth}
F(x)- F([x]^*) \geq \frac{\sigma}{2}\|x-[x]^* \|^2.
\end{equation}
\end{assumption}

\subsection{Technical proposition and lemma}

In order to establish the convergence results, it will be useful to establish Proposition~\ref{prop:41717892449y8} and Lemma~\ref{lem:gminusnablafk}.

\begin{proposition} \cite{xiao2014proximal, qian2019saga} \label{prop:41717892449y8}
Let $f$ be $\mM$-smooth and suppose that~\eqref{eq:strong_growth} holds. 
Suppose that $x^{k+1} = x^k  - \alpha g^k$ where $\E{g^k} = \nabla f(x^k)$ and $\alpha \leq \frac{1}{3\lambda_{\max}(\mM)}$. Then
 $$
\mathbb{E}_{k}\left[\left\|x^{k+1}-\left[x^{k+1}\right]^{*}\right\|^{2}\right] \leq \frac{1}{1+\sigma \alpha} \mathbb{E}_{k}\left[\left\|x^{k}-\left[x^{k}\right]^{*}\right\|^{2}\right]+\frac{2 \alpha^{2}}{1+\sigma \alpha} \mathbb{E}_{k}\left[\left\|g^{k}-\nabla f\left(x^{k}\right)\right\|^{2}\right].
$$
\end{proposition}

\begin{lemma}\label{lem:gminusnablafk}
For any $x^* \in \cX^*$ we have
\begin{equation} \label{eq:gminusnablafk}
\E{\| g^k - \nabla f(x^k)\|^2} \leq \frac{2}{n^2} \E{\left \| \cU (\mG(x^*)-\mJ^k)\eR \right\|^2} +\frac{2}{n^2}\E{\left\| \cU(\mG(x^k)-\mG(x^*))\eR  \right\|^2}.
\end{equation}
\end{lemma}

\begin{proof}
\begin{eqnarray*}
&& \E{\| g^k - \nabla f(x^k)\|^2}  \\
&=& \E{\left \| \frac1n \mJ^k \eR - \frac1n\cU(\mG(x^k)-\mJ^k)\eR - \frac1n\mG(x^k)\eR \right\|^2} \\
&=& \frac{1}{n^2} \E{ \left \| ( \mJ^k - \mG(x^*)) \eR- \cU(\mG(x^*)-\mJ^k)\eR + \cU(\mG(x^k)-\mG(x^*))\eR + (\mG(x^*)- \mG(x^k))\eR \right \|^2} \\
&\leq & 
\frac{2}{n^2} \E{\left \| (\mJ^k  - \mG(x^*))\eR- \cU(\mG(x^*)-\mJ^k)\eR \right\|^2}\\
&& \qquad+\frac{2}{n^2}\E{\left\| \cU(\mG(x^k)-\mG(x^*))\eR + (\mG(x^*)- \mG(x^k))\eR \right\|^2}\\
&\leq &
\frac{2}{n^2} \E{\left \| \cU (\mG(x^*)-\mJ^k)\eR \right\|^2} +\frac{2}{n^2}\E{\left\| \cU(\mG(x^k)-\mG(x^*))\eR  \right\|^2}.
\end{eqnarray*}
\end{proof}

Lastly, it is necessary to assume the null space consistency of solution set $\cX^*$ under $\mM$ smothness. A similar assumption was considered in~\cite{qian2019saga}.
\begin{assumption} \label{as:null_consistency}
For any $x^*, y^* \in \cX$ we have 
\begin{equation}\label{eq:null_consistency}
{\cM^\dagger}^\frac12\mG(x^*) = {\cM^\dagger}^\frac12\mG(y^*).
\end{equation}
\end{assumption}

\subsection{Theorem}
We next state the convergence result of Algorithm~\ref{alg:SketchJac} under strong growth condition.

\begin{theorem} \label{thm:strong}

Suppose that~\eqref{eq:strong_growth} holds. Let $\cB$ be any linear operator commuting with $\cS$, and assume ${\cM^\dagger}^{\nicefrac{1}{2}}$ commutes with $\cS$. Let $\cR$ be any linear operator for which $\cR(\mJ^k) = \cR(\mG(x^*))$ for every $k\geq 0$. Define the Lyapunov function $\Psi^k$ as per~\eqref{eq:Lyapunov} for any $x^*\in \cX^*$. Suppose that $\alpha\leq \frac{1}{\lambda_{\max}(\mM)}$ and $\cB$ are chosen so that 
\begin{eqnarray} \nonumber 
&&  \frac{2\alpha}{n^2} \left(  \frac{3+ \sigma \alpha}{1+\sigma\alpha}\right) \E{ \norm{ \cU  \mX \eR }^2 }   +   \NORMG{  \left(\cI - \E{\cS} \right)^{\frac12}\cB  {\cM^\dagger}^{\frac12} \mX } \\
& & \qquad \qquad \qquad \leq \left(1-\frac{\alpha \sigma}{2+ 2\alpha \sigma}\right)  \NORMG{ \cB {\cM^\dagger}^{\frac12}\mX }\label{eq:sg_small_step}
\end{eqnarray}
whenever  $ \mX\in \Range{\cR}^\perp$ and
\begin{eqnarray}
\frac{2\alpha}{n^2}\left(  \frac{3+ \sigma \alpha}{1+\sigma\alpha}\right) \E{  \norm{ \cU  \mX  \eR }^2  } 
+   \NORMG{\left(\E{\cS}\right)^{\frac12}  \cB  {\cM^\dagger}^{\frac12}\mX  }   &  \leq & \frac{1}{n} \norm{{\cM^\dagger}^{\frac12}\mX}^2
\label{eq:sg_small_step2}
\end{eqnarray}
for all $\mX\in \R^{d\times n}$. Then  for all $k\geq 0$, we have
\[ \E{\Psi^{k}}\leq \left(1-\frac{\alpha \sigma}{2+ 2\alpha \sigma}\right)^k \Psi^{0}.\]

\end{theorem}
\begin{proof}

Consider any $x^*\in \cX^*$. Due to non-expansiveness of the prox operator we have
\begin{eqnarray*}
\E{\norm{x^{k+1} -[x^{k+1}]^*}_2^2 } &\leq&
\E{\norm{x^{k+1} -[x^{k}]^*}_2^2 }\\
&\overset{ \eqref{eq:prox_opt}}{=} &
 \E{\norm{\prox(x^k-\alpha g^k) - \prox([x^k]^*-\alpha \nabla f([x^k]^*))  }_2^2}  
 \\
 &\leq&
 \E{\norm{x^k - \alpha g^{k} -( [x^k]^* -\alpha \nabla f([x^k]^*) ) }_2^2}  
\\
& =& 
 \norm{x^k  -[x^k]^*}_2^2 -2\alpha \dotprod{\nabla f(x^k)- \nabla f([x^k]^*) , x^k  -[x^k]^*}   \\
 && \qquad + \alpha^2\E{\norm{g^{k}- \nabla f([x^k]^*) }_2^2}
 \\
& \stackrel{\eqref{eq:smooth_dotprod}}{\leq}& 
 \norm{x^k  -[x^k]^*}_2^2 -\frac{2\alpha}{n} \left\| {\cM^{\dagger}}^\frac12( \mG(x^k)- \mG([x^k]^*)) \right \|^2 \\
 && \qquad + \alpha^2\ED{\norm{g^{k}- \nabla f([x^k]^*) }_2^2}.
\end{eqnarray*}

Combining the above bound with Proposition~\ref{prop:41717892449y8} yields
\begin{eqnarray*}
&& \E{\norm{x^{k+1} -[x^{k+1}]^*}_2^2 }\\
&\leq & 
\left(\frac{1}{2+2\alpha \sigma}+\frac12\right)\left\|x^{k}-\left[x^{k}\right]^{*}\right\|^{2}-\frac{\alpha}{n}\left\| {\cM^{\dagger}}^\frac12( \mG(x^k)- \mG([x^k]^*)) \right\|^2 \\
&& 
+\frac12\alpha^{2}\E{\left\|g^{k}-\nabla f\left(\left[x^{k}\right]^{*}\right)\right\|^{2}
}+\frac{ \alpha^{2} }{1+\sigma \alpha} \E{\left\|g^{k}-\nabla f\left(x^{k}\right)\right\|^{2}}
\\
&\leq & 
\left(\frac{\alpha \sigma +2}{2+2\alpha \sigma}\right)\left\|x^{k}-\left[x^{k}\right]^{*}\right\|^{2}-\frac{\alpha}{n} \left\| {\cM^{\dagger}}^\frac12( \mG(x^k)- \mG([x^k]^*)) \right\|^2 
\\
&& 
+\frac12\alpha^{2}\E{\left\|g^{k}-\nabla f\left(\left[x^{k}\right]^{*}\right)\right\|^{2}}+\frac{ \alpha^{2} }{1+\sigma \alpha}\E{\left\|g^{k}-\nabla f\left(x^{k}\right)\right\|^{2}}
\\
&\stackrel{\eqref{eq:gminusnablafk}}{\leq} & 
\left(\frac{\alpha \sigma +2}{2+2\alpha \sigma}\right)\left\|x^{k}-\left[x^{k}\right]^{*}\right\|^{2}-\frac{\alpha}{n}\left\| {\cM^{\dagger}}^\frac12( \mG(x^k)- \mG([x^k]^*))\right\|^2 
\\
&& 
+\frac{ 2\alpha^{2} }{n^2(1+\sigma \alpha)} \left(\E{\left \| \cU(\mG(x^*)-\mJ^k)\eR \right\|^2}  +  \E{\left\| \cU(\mG(x^k)-\mG(x^*))\eR  \right\|^2} \right) 
\\
&&
+\frac12\alpha^{2}\E{\left\|g^{k}-\nabla f\left(\left[x^{k}\right]^{*}\right)\right\|^{2}}
\\
&\stackrel{\eqref{eq:g_lemma}}{\leq} & 
\left(\frac{\alpha \sigma +2}{2+2\alpha \sigma}\right)\left\|x^{k}-\left[x^{k}\right]^{*}\right\|^{2}-\frac{\alpha}{n} \left\| {\cM^{\dagger}}^\frac12 ( \mG(x^k)- \mG([x^k]^*)) \right\|^2 
\\
&& 
+\frac{\alpha^{2}}{n^2}\left( \frac{ 2 }{1+\sigma \alpha}+1\right)\left(\E{\left \| \cU(\mG(x^*)-\mJ^k)\eR \right\|^2}  +  \E{\left\| \cU(\mG(x^k)-\mG(x^*))\eR  \right\|^2} \right)
\\
&\stackrel{\eqref{eq:null_consistency}}{\leq} & 
\left(\frac{\alpha \sigma +2}{2+2\alpha \sigma}\right)\left\|x^{k}-\left[x^{k}\right]^{*}\right\|^{2}-\frac{\alpha}{n} \left\| {\cM^{\dagger}}^\frac12 ( \mG(x^k)- \mG(x^*)) \right\|^2 
\\
&& 
+\frac{\alpha^{2}}{n^2}\left( \frac{ 2 }{1+\sigma \alpha}+1\right)\left(\E{\left \| \cU(\mG(x^*)-\mJ^k)\eR \right\|^2}  +  \E{\left\| \cU(\mG(x^k)-\mG(x^*))\eR  \right\|^2} \right).
\end{eqnarray*}
 Since, by assumption,  both $\cB$ and ${\cM^\dagger}^{\frac12}$ commute with $\cS$, so does their composition $\cA \eqdef \cB {\cM^\dagger}^{\frac12}$. Applying Lemma~\ref{lem:nb98gd8fdx}, we get
 \begin{eqnarray}\label{eq:J_jac_bound9999}\nonumber
\E{ \NORMG{\cB {\cM^\dagger}^{\frac12} \left(\mJ^{k+1}-\mG(x^*)  \right) }}  &=&  \NORMG{ (\cI - \E{\cS})^{\frac12}  \cB {\cM^\dagger}^{\frac12} \left(\mJ^k-\mG(x^*) \right) }  \\
&& +  \NORMG{\E{\cS}^{\frac12}  \cB {\cM^\dagger}^{\frac12} \left(\mG(x^k)-\mG(x^*) \right) } .
\end{eqnarray}

 Adding $\alpha$ multiple of~\eqref{eq:J_jac_bound9999} to the previous bounds yields
 \begin{eqnarray*}
&& \E{\norm{x^{k+1} -[x^{k+1}]^*}_2^2 }  + \alpha\E{ \NORMG{\cB {\cM^\dagger}^{\frac12} \left(\mJ^{k+1}-\mG(x^*)  \right) }}    \\
&\leq & 
\left(1- \frac{\alpha \sigma }{2+2\alpha \sigma}\right)\left\|x^{k}-\left[x^{k}\right]^{*}\right\|^{2}-\frac{\alpha}{n} \left\| {\cM^{\dagger}}^\frac12 ( \mG(x^k)- \mG(x^*)) \right\|^2 
\\
&& 
+\frac{\alpha^{2}}{n^2}\left( \frac{ 3+ \sigma\alpha }{1+\sigma \alpha}\right)\left(\E{\left \| \cU(\mG(x^*)-\mJ^k) \eR\right\|^2}  +  \E{\left\| \cU(\mG(x^k)-\mG(x^*)) \eR \right\|^2} \right)
\\
&& 
+ \alpha \NORMG{ (\cI - \E{\cS})^{\frac12}  \cB {\cM^\dagger}^{\frac12} \left(\mJ^k-\mG(x^*) \right) }   + \alpha\NORMG{\E{\cS}^{\frac12}  \cB {\cM^\dagger}^{\frac12} \left(\mG(x^k)-\mG(x^*) \right) }
\\
&\stackrel{\eqref{eq:sg_small_step2}}{\leq} & 
\left(1- \frac{\alpha \sigma }{2+2\alpha \sigma}\right)\left\|x^{k}-\left[x^{k}\right]^{*}\right\|^{2}
+\frac{\alpha^{2}}{n^2}\left( \frac{ 3+ \sigma\alpha }{1+\sigma \alpha}\right)\E{\left \| \cU(\mG(x^*)-\mJ^k)\eR \right\|^2}  
\\
&& 
+ \alpha \NORMG{ (\cI - \E{\cS})^{\frac12}  \cB {\cM^\dagger}^{\frac12} \left(\mJ^k-\mG(x^*) \right) }
\\
&\stackrel{\eqref{eq:sg_small_step}}{\leq} & 
\left(1- \frac{\alpha \sigma }{2+2\alpha \sigma}\right) \left( \left\|x^{k}-\left[x^{k}\right]^{*}\right\|^{2}+  \alpha\NORMG{\cB {\cM^\dagger}^{\frac12} \left(\mJ^{k}-\mG(x^*)  \right) } \right).
\end{eqnarray*}
\end{proof}

\begin{remark}
Since $2+2\alpha \sigma = \cO(1)$ and $\frac{3\sigma \alpha }{1+ \sigma\alpha} =\cO(1)$ the convergence rate under strong growth provided by Theorem~\ref{thm:strong} is of the same order as the convergence rate under quasi strong convexity (Theorem~\ref{thm:main}).  
\end{remark}

\end{document}